\documentclass[10pt,reqno]{amsart}

\usepackage{enumerate}
\usepackage{amsmath, amssymb, amsthm}
\usepackage{mathrsfs}
\usepackage{esint}
\usepackage{xcolor}
\usepackage{mathtools}
\usepackage{hyperref}
\usepackage{bm}
\usepackage{kotex}
\usepackage{accents}
\usepackage{todonotes}

\makeatletter

\hypersetup{bookmarksdepth=2}

\numberwithin{equation}{section}

\newtheorem{thm}{Theorem}[section]
\newtheorem{theorem}[thm]{Theorem}
\newtheorem{lemma}[thm]{Lemma}

\newtheorem{lem}[thm]{Lemma}
\newtheorem{prop}[thm]{Proposition}
\theoremstyle{definition}

\newtheorem{definition}[thm]{Definition}
\newtheorem{example}[thm]{Example}
\newtheorem{assumption}[thm]{Assumption}
\theoremstyle{remark}

\newtheorem{remark}[thm]{\bf{Remark}}

\newcommand\bD{\mathbb{D}}
\newcommand\bE{\mathbb{E}}

\newcommand\bH{\mathbb{H}}

\newcommand\bL{\mathbb{L}}

\newcommand\bN{\mathbb{N}}

\newcommand\bP{\mathbb{P}}

\newcommand\bR{\mathbb{R}}
\newcommand\bS{\mathbb{S}}

\newcommand\bZ{\mathbb{Z}}

\newcommand\cB{\mathcal{B}}

\newcommand\cD{\mathcal{D}}

\newcommand\cF{\mathcal{F}}

\newcommand\cH{\mathcal{H}}

\newcommand\cP{\mathcal{P}}

\newcommand\cS{\mathcal{S}}

\newcommand\frH{\mathfrak{H}}

\begin{document}

\title[SPDE\MakeLowercase{s} with nonlocal operators in open sets]{The Dirichlet problem for stochastic partial differential equations with nonlocal operators in $C^{1,\sigma}$ open sets}

\author{Kyeong-Hun Kim}
\address{Department of Mathematics, Korea University, 145 Anam-ro, Seongbuk-gu, Seoul,
02841, Republic of Korea}
\email{kyeonghun@korea.ac.kr}
\thanks{}

\author{Junhee Ryu}
\address{School of Mathematics, Korea Institute for Advanced Study, 85 Hoegi-ro, Dongdaemun-gu, Seoul, 02455, Republic of Korea} \email{junhryu@kias.re.kr}
\thanks{K.-H. Kim was supported by the National Research Foundation of Korea(NRF) grant funded by the Korea government(MSIT) (No. RS-2025-00556160). J. Ryu was supported by a KIAS Individual Grant (MG101501) at Korea Institute for Advanced Study.}

\subjclass[2020]{60H15, 35R60, 35B65, 45K05}

\keywords{Stochastic partial differential equations, Dirichlet problem, nonlocal equations, Gaussian noise, maximal $L_p$-regularity}

\begin{abstract}
This paper provides a comprehensive  Sobolev regularity theory for the Dirichlet problem of stochastic partial differential equations in $C^{1,\sigma}$ open sets. We consider substantially large classes of nonlocal operators and generalized Gaussian noise.  Our main results include the existence and uniqueness of strong solutions in weighted Sobolev spaces, along with maximal $L_p$-regularity estimates for the solutions.
\end{abstract}

\maketitle


\section{Introduction}
This paper is devoted to the study of Sobolev regularity theory for stochastic partial differential equations (SPDEs) with nonlocal operators in $C^{1,\sigma}$ open sets. More precisely, we consider the Dirichlet problem of the prototype
 \begin{equation} \label{eq_intro}
\begin{cases}
du=(L_tu + f(u))dt +  h(u)\dot{W}, \quad &(t,x)\in(0,\tau)\times D,
\\
u(0,x)=u_0,\quad & x\in D,
\\
u(t,x)=0,\quad &(t,x)\in[0,\tau]\times D^c,
\end{cases}
\end{equation}
where
\begin{itemize}
    \item $D$ is a bounded $C^{1,\sigma}$ open set with $\sigma\in(0,1)$,

    \item $L_t$ is a time-inhomogeneous random symmetric nonlocal operator of order $\alpha\in(0,2)$,

    \item $W$ is a generalized centered Gaussian noise that is white in time and spatially homogeneous,

    \item the functions $f(u)$ and $h(u)$ are either semi-linear or  super-linear; for instance one may take $h(u)=\xi |u|^{1+\lambda}$, where $\lambda\geq 0$ and $\xi$ is a real-valued function.
\end{itemize}

We first discuss the nonlocal operator $L_t$, which is defined by
\begin{align} \label{oper}
L_t u(x) = \frac{1}{2}\int_{\bR^d} \left( u(x+y)+u(x-y)-2u(x) \right)\, \nu_t(\omega, dy), \quad t\in(0,\tau),
\end{align}
where $\nu_t$ is a nondegenerate $\alpha$-stable symmetric L\'evy measure for each $t\in(0,\infty)$ and $\omega\in \Omega$.
It is worth noting that the infinitesimal generator of any symmetric $\alpha$-stable L\'evy process is of the form \eqref{oper}, with $\nu_t$ independent of both $t$ and $\omega$.
The most well-known example is the fractional Laplacian $-(-\Delta)^{\alpha/2}$, which is the infinitesimal generator of a rotationally symmetric $\alpha$-stable L\'evy process. In this case, $\nu_t(dy)=c|y|^{-d-\alpha}dy$ for some $c>0$. We also introduce a singular example; the generator of $d$ independent one-dimensional symmetric stable L\'evy processes, which is defined as
\begin{equation*}
    L_tu(x):=\sum_{i=1}^d -(-\Delta)_{x_i}^{\alpha/2} = c(\alpha) \sum_{i=1}^d \int_{\bR} \frac{f(x+y_ie_i) + f(x-y_ie_i) - 2f(x)}{|y_i|^{1+\alpha}} dy_i,
\end{equation*}
where $e_i$ is the unit vector in the $i$-th coordinate. This operator is called singular because the spectral measure of the L\'evy measure (see \eqref{eq1022000}) is given by a sum of $2d$ Dirac measures defined on the unit sphere. We emphasize that in this paper, we cover not only deterministic cases but also those that are merely measurable in $(\omega,t)$.

The noise $W$ is a mean zero Gaussian random field that is white in time and spatially homogeneous. Formally, the covariance is given by
\begin{equation*}
    \bE(W(t,x)W(s,y))=\delta_0(t-s)\Pi(x-y),
\end{equation*}
where $\delta_0$ is the centered Dirac delta distribution and $\Pi$ is a nonnegative definite tempered measure on $\bR^d$. A rigorous formulation of the noise is given in Section \ref{sec_defnoise}. When $\Pi(x-y)=\delta_0(x-y)$, $W$ becomes a space-time white noise, which is extremely singular. In this case, real-valued solutions to the SPDEs can be obtained only when $d=1$ (see e.g. \cite{CS20}). In this paper, we consider not only this singular noise but also more regular noises to handle SPDEs in higher dimensions.

In the literature, there are numerous results on SPDEs in the whole space $\bR^d$, involving the operators and noise introduced above. 
We briefly review Sobolev regularity results for linear equations.
For equations with local operators and white noise, we refer the reader to \cite{K96, KAA}, where Krylov firstly introduced an analytic approach to obtain maximal regularity of solution. Since the work of \cite{K96, KAA}, the analytic approach has been extended to equations with various types of spatial operators, such as \eqref{oper}. See \cite{CL12, KK20, KKK13, MP12, MP21}.
In \cite{LV21, NVW12, PV19}, the method based on $H^{\infty}$-calculus was presented.
For equations driven by spatially homogeneous noise, related results can be found in \cite{CH21, CHP24, FS06}. We also refer the reader to \cite{CD23, D99, Dalang, DS80} for the basic theory of colored noise.

A number of results have also been established for super-linear equations
\begin{equation*}
    du=\Delta u dt + |u|^{1+\lambda}\dot{W}, \quad (t,x)\in (0,\tau)\times D; u(0,\cdot)=u_0.
\end{equation*}
By combining the results of \cite{M91, M00, S25}, when $W$ is a space-time white noise, one can conclude that real-valued solutions can be obtained only when $d=1$ and $0\leq\lambda\leq1/2$. See also \cite{H20, KAA} for Sobolev regularity results. In \cite{CH21, CHP24, H21}, nonlocal operators or general Gaussian noise were considered in the whole space $D=\bR^d$.

As far as we know, the existing results in the literature on Sobolev regularity theory for SPDEs involving nonlocal operators only address equations defined on the whole space.
In this paper, we focus on SPDEs driven by nonlocal operators and generalized Gaussian noise in open sets. Our aim is to obtain maximal regularity of solutions in the weighted Sobolev spaces $\bH_{p,\theta}^\gamma(D,\tau):=L_p(\Omega\times(0,\tau);H_{p,\theta}^\gamma(D))$, where $p\geq2$, $\theta,\gamma\in \bR$, and $\tau$ is a stopping time.
For a special case when $\gamma=0,1,2,\dots$,
\begin{align*}
\|u\|_{H_{p,\theta}^{\gamma}(D)} := \sum_{m\leq \gamma} \left( \int_D |d_x^{m} D_x^m u|^p d_x^{\theta-d} dx \right)^{1/p},
\end{align*}
where $d_x=dist(x,\partial D)$. To the best of our knowledge, the spaces $H_{p,\theta}^{\gamma}(D)$ were firstly introduced in \cite[ Section 2.6.3]{ML68} for the specific case $p = 2$ and $\theta = d$, and they were generalized in a unified manner for $p \in (1,\infty)$ and $\theta ,\gamma \in \bR$ in \cite{K99} in order to establish an $L_p$-theory of SPDEs.

A common obstacle in the study of both deterministic nonlocal equations and local SPDEs in domains is that the highest order derivatives of solutions blow up near the boundary. Even in the simple elliptic model $-(-\Delta)^{\alpha/2}u=1$ with zero exterior condition, for $\beta\in(\alpha/2,\alpha)$,
\begin{equation*}
    \lim_{d_x\to0} \frac{(-\Delta)^{\beta/2}u}{d_x^{\alpha/2-\beta}} \, \text{ exists.}
\end{equation*}
See \cite{Dyda12} for the precise asymptotics.
To control the boundary behavior of solution, weighted Sobolev spaces were introduced in \cite{CKR23, DR24}. See also \cite{R16, RS14} for weighted H\"older theory, and \cite{AG23, G14, Gfrac} for $L_p$-maximal regularity results in the $\mu$-transmission space. For local SPDEs, Krylov's weighted theory plays an analogous role; see \cite{K04, K94, KL99}.
The equations considered in this paper show similar boundary singularities, so we employ weighted Sobolev spaces to handle the boundary behavior of solutions.

Now we introduce the main contributions of this paper.

\vspace{1mm}
$(1)$ Theorem \ref{thm_white}: Semi-linear SPDEs driven by infinite dimensional Wiener noise.

Under Lipschitz conditions on $f$ and $g$, we prove existence, uniqueness, and maximal weighted $L_p$ estimates for 
        \begin{equation*}
du=\left(L_tu+f(u)\right)dt + \sum_{k=1}^\infty g^k(u)dw^k_t,\quad (t,x)\in(0,\tau)\times D,
\end{equation*}
with $u(0,\cdot)=u_0$ and zero exterior condition. Here, $w_t^k$ a sequence of  independent one-dimensional Wiener processes. In particular, we show that for any $p\geq2$ and $\gamma\in[0,\alpha]$,
\begin{align*}
    &\|\psi^{-\alpha/2}u\|_{\bH_{p,\theta}^{\gamma}(D,\tau)}
    \\
    &\leq N \left(\|u_0\|_{U_{p,\theta}^{\gamma}(D)} + \|\psi^{\alpha/2} f(0)\|_{\bH_{p,\theta}^{\gamma-\alpha}(D,\tau)} + \|g(0)\|_{\bH_{p,\theta}^{\gamma-\alpha/2}(D,\tau,l_2)}\right),
\end{align*}
where $U_{p,\theta}^{\gamma}(D)$ is the weighted Besov space, which is introduced to handle the initial data (see Section \ref{sec_func}). When $D$ is convex, then the range of $\theta\in(d-1,d-1+p)$ is sharp. However, for general open sets, the range is restricted due to a technical nature (cf. Remark \ref{rem4061023}). To the best of our knowledge, Theorem \ref{thm_white} provides the first maximal regularity results for nonlocal SPDEs in domains, even in the linear case.

To prove the above estimate, we exploit an analytic approach based on Krylov's kernel-free approach (see \cite{KL99}).
We first obtain higher-order regularity of solutions by using an estimate of the commutator term. Then we prove zeroth-order estimates by using It\^o's formula together with a version of nonlocal integration by parts presented in \cite{DR24}.

\vspace{1mm}
$(2)$ Theorem \ref{mainthm}: SPDEs driven by spatially homogeneous noise.

We consider a class of generalized Gaussian noise including a space-time white noise. We show that under the reinforced Dalang's condition on the noise (see Assumption \ref{ass_D}), \eqref{eq_intro} admits a unique real-valued solution even in higher dimensions. In this result, Lipschitz conditions on $f$ and $h$ are imposed. Moreover, we obtain regularity of solutions, whose upper bound is determined by the noise's spatial covariance.

The key idea in the proof is to utilize the fact that the noise can be represented as an infinite series of Wiener processes. Based on this representation, we analyze the negative regularity of the stochastic term.

\vspace{1mm}
$(3)$ Theorem \ref{thm_super}: SPDEs with super-linear multiplicative noise term.

We study \eqref{eq_intro} in the case when $h(u)=\xi |u|^{1+\lambda}$. Our results include solvability and a characterization of the admissible range of the regularity parameter $\gamma$ and the exponent $\lambda$ depending on the noise.

Our strategy is to approximate the equation by a sequence of semi-linear equations. More precisely, we consider the truncated nonlinear term $h_m(u)=|0\vee u\wedge m|^{1+\lambda}$ and the corresponding solution $u_m$. By showing that each $u_m$ does not explode in finite time, we can conclude that the limit of $u_m$ is the solution to the desired super-linear equation.

\vspace{1mm}
$(4)$ Theorem \ref{thm_max}: Maximum principle.

We establish a version of the maximum principle, which is a fundamental tool for the analysis of SPDEs.
We prove this by using It\^o's formula and integration by parts for nonlocal operators.

Now, we introduce the organization of this paper. In Section \ref{sec_2}, we
introduce the function spaces and provide precise definitions of the operators and noise. We state our main results in Section \ref{sec_3}. In Sections \ref{proof semilinear}-\ref{proofsuper}, we provide the proofs of our main results. The proof of the maximum principle is placed in Section \ref{appA}. Finally, in Section \ref{appB}, we prove certain properties of function spaces in the whole space to derive those in open sets.

We introduce notations used in this paper.
We use $``:="$ or $``=:"$ to denote a definition.  $\bN$ and $\bZ$ denote the natural number system and the integer number system, respectively. We denote $\bN_+:=\bN\cup\{0\}$, and  as usual $\bR^d$ stands for the Euclidean space of points $x=(x^1,\dots,x^d)$,
$$
B_r(x)=\{y\in\bR^d : |x-y|<r\}, \quad \bR^{d}_+=\{(x^1,\dots,x^d)\in\bR : x^1>0\}.
$$
In particular, $\bR:=\bR^1$. For $a,b\in \bR$, $a\wedge b:=\min\{a,b\}$ and $a\vee b:=\max\{a,b\}$.
By $\cF_d$ and $\cF^{-1}_d$, we denote the $d$-dimensional Fourier transform and the inverse Fourier transform respectively, i.e.,
$$
\cF_d[f](\xi):=\int_{\bR^d} e^{-i\xi\cdot x} f(x) dx, \quad \cF_d^{-1}[f](\xi):=\frac{1}{(2\pi)^d}\int_{\bR^d} e^{i\xi\cdot x} f(\xi) d\xi.
$$
For nonnegative functions $f$ and $g$, we write $f(x)\approx g(x)$  if there exists a constant $N>0$, independent of $x$,  such that $N^{-1} f(x)\leq g(x) \leq N f(x)$. For functions $u(x)$, we write
$$
 D_iu(x):=\frac{\partial u}{\partial x^i}.
$$
We also use  $D^n_xu$ to denote the  partial derivatives of order $n\in\bN_+$ with respect to the space variables. For an open set $U\subset \bR^d$, $C_b(U)$ denotes the space of continuous functions $u$ in $U$ such that $|u|_{C_b(U)}:=\sup_U |u(x)|<\infty$.  By $C^2_b(U)$ we denote the space of functions  whose derivatives of order up  to $2$ are in $C(U)$.   For an open set $V\subset \bR^m$, where $m\in \bN$,   by $C_c^\infty(V)$ we denote the space of infinitely differentiable functions with compact support in $V$. For  a Banach space $F$ and $\delta\in (0,1]$,   $C^{\delta}(V;F)$ denotes the space of $F$-valued continuous functions $u$ on $V$  such that
\begin{eqnarray*}
    |u|_{C^{\delta}(V;F)}&:=&|u|_{C(V;F)}+[u]_{C^{\delta}(V;F)}
    \\
    &:=&
     \sup_{x\in V}|u(x)|_F+\sup_{x,y\in V}\frac{|u(x)-u(y)|_F}{|x-y|^{\delta}}<\infty.
\end{eqnarray*}
Also, for  $p>1$ and a measure $\mu$ on $V$, $L_p(V, \mu; F)$  denotes the set of $F$-valued Lebesgue measurable functions $u$ such that 
$$
\|u\|_{L_p(V, \mu; F)}:=\left(\int_V |u|^p_F \,d\mu\right)^{1/p}<\infty.
$$
 We drop $F$ and $\mu$ if $F=\bR$ and $\mu$ is the Lebesgue  measure. 
 By $\cD'(D)$, where $D$ is an open set in $\bR^d$, 
we denote the space of all distributions on $D$, and  for given $f\in \cD'(D)$, the action of $f$ on $\phi \in C_c^\infty(D)$ is denoted by
\begin{equation} \label{eq3171811}
    (f, \phi)_{D} :=f(\phi).
\end{equation}
For $f\in\cD'(D)$, we say that $f\geq 0$ if
\begin{equation} \label{eq6292218}
  (f,\phi)_D  \geq 0
\end{equation}
for any $\phi\in C_c^\infty(D)$ such that $\phi\geq0$.
Next, $\cS(\bR^d)$ stands for the space of Schwartz functions on $\bR^d$. When $D=\bR^d$, we omit $D$ in \eqref{eq3171811}.
If we write $N=N(a,b,\cdots)$, then this means that the constant $N$ depends only on $a,b,\cdots$. For functions depending on $\omega,t$, and $x$, the argument $\omega\in\Omega$ is omitted.

\section{Preliminaries}  \label{sec_2}

\subsection{Function spaces} \label{sec_func}

In this subsection we introduce function spaces and some of their key properties that will be used throughout the paper.

Let $l_2$ be the set of all sequences $a=(a^1,a^2,\cdots)$ such that
$$
|a|_{l_2}:=\left(\sum_{k=1}^\infty |a^k|^2 \right)^{1/2}<\infty.
$$
We begin by recalling the definitions of Sobolev and Besov spaces on $\bR^d$. For $p\in(1,\infty)$ and $\gamma\in\bR$, the Sobolev space $H_p^\gamma=H_p^\gamma(\bR^d)$ is defined as the set of all tempered distributions $u$  on $\bR^d$ such that
$$
\|u\|_{H_p^\gamma}=\|(1-\Delta)^{\gamma/2}u\|_{L_p}= \left\| \cF^{-1}_d \left\{ (1+|\cdot|^2)^{\gamma/2}\cF_d(u)(\cdot) \right\} \right\|_{L_p} <\infty.
$$
It is well-known that if $\gamma\in\bN$, this space coincides with the standard Sobolev space:
$$
H_p^\gamma=W_p^\gamma:=\{u: D_x^\beta u \in L_p, |\beta|\leq \gamma\}.
$$
Furthermore, for $\gamma_1<\gamma_2$,
\begin{align} \label{inclu}
H_p^{\gamma_2}\subset H_p^{\gamma_1}.
\end{align}
We also consider  the $l_2$-valued Sobolev space $H_p^\gamma(l_2)=H_p^\gamma(\bR^d;l_2)$,  consisting of   all $l_2$-valued tempered distributions $v$ on $\bR^d$ for which
$$
\|v\|_{H_p^\gamma(l_2)}:=\| |(1-\Delta)^{\gamma/2}v|_{l_2} \|_{L_p}<\infty.
$$
 To define the Besov space, we choose a function $\Psi$ whose Fourier transform $\cF_d(\Psi)$ is infinitely differentiable, supported in an annulus $\{\xi\in\bR^d : \frac{1}{2} \leq |\xi| \leq 2\}$, $\cF_d(\Psi)\geq0$ and
$$
\sum_{j\in\bZ} \cF_d(\Psi)(2^{-j}\xi)=1, \quad \forall \xi\neq0.
$$
For a tempered distribution $u$ and $j\in \bZ$,  define
$$
\Delta_j u(x):= \cF_d^{-1}\left( \cF_d(\Psi)(2^{-j}\cdot) \cF_d(u) \right)(x), \quad S_0 u(x):=\sum_{j=-\infty}^0 \Delta_j u(x).
$$
Then for $p>1$ and $\gamma\in\bR$, the Besov space $B_{p,p}^\gamma=B_{p,p}^\gamma(\bR^d)$ consists of  all tempered distributions $u$ such that
$$
\|u\|_{B_{p,p}^\gamma}:=\|S_0u\|_{L_p} + \left( \sum_{j=1}^\infty 2^{j p \gamma} \|\Delta_j u\|_{L_p}^p \right)^{1/p} <\infty.
$$
It is known that for $p\geq2$, we have the continuous embedding $H_p^\gamma\subset B_{p,p}^\gamma$ (see e.g. \cite[Proposition 2.3.2]{H83}).

Next, we introduce weighted Sobolev and Besov spaces on $D\subset\bR^d$. Throughout this section, we always assume  $D$ is an open set with a nonempty boundary.

Let $d_x:=dist(x,\partial D)$, and define the weighted Lebesgue space $L_{p,\theta}(D):=L_{p}(D,d_x^{\theta-d}dx)$ for $\theta\in\bR$. Then for $\theta\in\bR$ and $n\in\bN_+$, we define the weighted Sobolev space
$$
H_{p,\theta}^n(D):=\{u\in \cD'(D) : u,d_x D_x u, \cdots, d_x^n D_x^n u \in L_{p,\theta}(D) \},
$$
equipped with the norm
\begin{align} \label{eq1032130}
\|u\|_{H_{p,\theta}^n(D)} := \sum_{m\leq n} \left( \int_D |d_x^{m} D_x^m u|^p d_x^{\theta-d} dx \right)^{1/p}.
\end{align}
To extend these spaces to arbitrary (possibly fractional) orders, we take a sequence of nonnegative functions $\{\zeta_n\}_{n\in\bN}$ in $C_c^\infty(D)$ satisfying the followings:
\begin{align}
    &(i)\,\,supp (\zeta_n) \subset \{x\in D : c_1e^{-n}< d_x<c_2e^{-n}\}, \quad c_2>c_1>0, \label{zeta1}
    \\
    &(ii)\,\,\sup_{x\in\bR^d}|D^m_x \zeta_n (x)| \leq N(m)e^{mn},\quad \forall m\in\bN_+, \label{zeta2}
    \\
    &(iii)\,\,\sum_{n\in\bZ} \zeta_n(x) > c>0,\quad\forall x\in D. \label{zeta3}
\end{align}
Such a sequence can be constructed by  mollifying indicator functions of the type $\{x\in D : c_3e^{-n}< d_x < c_4e^{-n}\}$.
If $\{x\in D : c_1e^{-n}< d_x <c_2e^{-n}\}=\emptyset$ for a partcular $n$, then we simply define $\zeta_n=0$.

We denote by $H_{p,\theta}^\gamma(D)$, $H_{p,\theta}^\gamma(D,l_2)$, and $B_{p,p;\theta}^\gamma(D)$ the spaces of distributions $u\in \cD'(D)$ for which
\begin{align} \label{def. Hptheta}
\|u\|^p_{H_{p,\theta}^\gamma(D)}:= \sum_{n\in\bZ} e^{n\theta} \| \zeta_{-n}(e^n\cdot) u(e^n\cdot)\|^p_{H_p^\gamma}<\infty,
\end{align}
$$
\|u\|^p_{H_{p,\theta}^\gamma(D,l_2)}:= \sum_{n\in\bZ} e^{n\theta} \| \zeta_{-n}(e^n\cdot) u(e^n\cdot)\|^p_{H_p^\gamma(l_2)}<\infty,
$$
and
\begin{align*}
\|u\|^p_{B_{p,p;\theta}^\gamma(D)}:= \sum_{n\in\bZ} e^{n\theta} \| \zeta_{-n}(e^n\cdot) u(e^n\cdot)\|^p_{B_{p,p}^\gamma}<\infty,
\end{align*}
respectively.  Note that a distribution on $D$ with compact support in $D$ can naturally be  regarded as a distribution on $\bR^d$. In particular, for any distribution $v$ on $D$, the action of $\zeta_n v$ on a test function $\phi\in C_c^\infty(\bR^d)$ is given by
$$
(\zeta_n v, \phi)_{\bR^d} = (v,\zeta_n \phi)_D, \quad \phi\in C_c^\infty(\bR^d).
$$
Conversely, for a distribution $u$ defined on $\bR^d$, if its restriction to $D$ belongs to $H_{p,\theta}^\gamma(D)$, we  still write $u\in H_{p,\theta}^\gamma(D)$. This also applies to the spaces $H_{p,\theta}^\gamma(D,l_2)$ and $B_{p,p;\theta}^\gamma(D)$.

It is known that (see e.g. \cite[Proposition 2.2]{L00}) the definitions of  $H_{p,\theta}^\gamma(D)$ and $B_{p,p;\theta}^\gamma(D)$ are independent of the particular choice of $\{\zeta_n\}$. More specifically, if $\{\xi_n\in C_c^\infty(D)\}$ satisfies \eqref{zeta1} and \eqref{zeta2}, then we have
$$
\sum_{n\in\bZ} e^{n\theta} \| \xi_{-n}(e^n\cdot) u(e^n \cdot)\|_{H_p^\gamma}^p \leq N \|u\|_{H_{p,\theta}^\gamma(D)}^p,
$$
and the reverse inequality also holds if $\{\xi_n\}$ additionally satisfies \eqref{zeta3}.  The same result applies to the Besov norm $\|u\|_{B_{p,p;\theta}^\gamma(D)}$.  Furthermore, when $\gamma=n\in\bN_+$, then norms \eqref{eq1032130} and \eqref{def. Hptheta} are equivalent (cf. \cite[Proposition 2.2]{L00}).
From the inclusion relation \eqref{inclu}, it follows that for any $p>1$ and $\theta\in\bR$,
\begin{align} \label{eq4250951}
H_{p,\theta}^{\gamma_2}\subset H_{p,\theta}^{\gamma_1} \text{ if } \gamma_1<\gamma_2.
\end{align}
If the open set $D$ is bounded, then there exists $n_0=n_0(D)\in \bZ$ such that $\zeta_{-n}=0$ for all $n\geq n_0$, and
\begin{equation} \label{eq5111835}
    H_{p,\theta}^\gamma \subset H_{p,\theta'}^\gamma \text{ if } \theta\leq \theta'.
\end{equation}

Next, take a smooth function $\psi$ in $D$ such that $\psi\approx d_x$ (i.e., $\psi$ behaves like the distance to the boundary) in $D$, and   for any $m\in\bN_+$ it satisfies
$$
\sup_{x\in D} |d_x^m D_x^{m+1}\psi(x)| \leq N(m)<\infty.
$$
An example of such a function is  $\psi:=\sum_{n\in\bZ} e^{-n} \zeta_n$, which fulfills the required properties.
For $\nu \in \bR$, we define 
$$\psi^{-\nu} H_{p,\theta}^{\gamma}(D):=\{u: \psi^\nu u \in H_{p,\theta}^\gamma(D)\}, \quad \psi^{-\nu} B_{p,p;\theta}^{\gamma}(D):=\{u: \psi^\nu u \in B_{p,p;\theta}^\gamma(D)\}.
$$

Below we present some properties of the function spaces introduced above.
We recall that $D$ is an open set with a nonempty boundary.

\begin{lem} \label{lem_prop}
Let $p\in(1,\infty)$, and $\gamma,\theta\in\bR$.

(i)  The spaces $H_{p,\theta}^\gamma(D)$ and $B_{p,p;\theta}^\gamma(D)$ are Banach spaces, and the set $C_c^\infty(D)$ is dense in both of these spaces.

(ii) 
For $\delta\in\bR$,  we have identities 
$$H_{p,\theta}^\gamma (D) = \psi^\delta H_{p,\theta+\delta p}^\gamma(D), \quad B_{p,p;\theta}^\gamma (D) = \psi^\delta B_{p,p;\theta+\delta p}^\gamma(D).
$$ Furthermore, the corresponding norms are  equivalent:
$$
\|u\|_{H_{p,\theta}^\gamma(D)} \approx \|\psi^{-\delta} u\|_{H_{p,\theta+\delta p}^\gamma(D)}, \quad \|u\|_{B_{p,p;\theta}^\gamma(D)}\approx \|\psi^{-\delta} u\|_{B_{p,p;\theta+\delta p}^\gamma(D)}.
$$

(iii) (Duality)
Let
\begin{equation*}
1/p+1/p'=1, \quad \theta/p+\theta'/p' = d.
\end{equation*}
Then the dual spaces of $H_{p,\theta}^\gamma(D)$ and $B_{p,p;\theta}^\gamma(D)$ are $H_{p',\theta'}^{-\gamma}(D)$ and $B_{p',p';\theta'}^{-\gamma}(D)$, respectively. Moreover, for $u\in H_{p,\theta}^\gamma(D)$ (resp. $u\in B_{p,p;\theta}^\gamma(D)$), the pairing $(u,\phi)_{D}$, initially defined on test functions $\phi\in C_c^\infty(D)$, can be extended continuously to all of  $H_{p',\theta'}^{-\gamma}(D)$ (resp.$B_{p',p';\theta'}^{-\gamma}(D)$).

(iv) Assume
$\gamma-\frac{d}{p} \geq n+\delta$ for  $n\in\bN_+$ and $\delta\in(0,1)$. Then for all $k\leq n$,
\begin{equation*}
    |\psi^{k+\frac{\theta}{p}} D_x^k u|_{C(D)}+[\psi^{n+\frac{\theta}{p}+\delta} D_x^{n}u]_{C^{\delta}(D)}\leq N(d,\gamma,p,\theta)\|u\|_{H_{p,\theta}^{\gamma}(D)}.
\end{equation*}

(v) (Interpolation)
Let $\kappa\in(0,1)$, $\gamma_i,\theta_i\in\bR$, $i=1,2$, and
$$
\theta=\kappa\theta_1+(1-\kappa)\theta_2, \quad \gamma=\kappa\gamma_1 + (1-\kappa)\gamma_2.
$$
Then, we have the interpolation results
\begin{align*}
H_{p,\theta}^\gamma(D)=[H_{p,\theta_1}^{\gamma_1}(D),H_{p,\theta_2}^{\gamma_2}(D)]_\kappa, \quad B_{p,\theta}^\gamma(D)=(H_{p,\theta_1}^{\gamma_1}(D),H_{p,\theta_2}^{\gamma_2}(D))_\kappa,
\end{align*}
where $[A,B]_\kappa$ (resp. $(A,B)_\kappa$) is the complex (resp. real) interpolation space of $A$ and $B$.

(vi)
The assertions in (i)-(v) also hold true for $H_{p,\theta}^\gamma(D,l_2)$ in place of $H_{p,\theta}^\gamma(D)$.

\end{lem}

\begin{proof}
See Proposition 2.2 and Theorem 4.3 in \cite{L00} (cf. \cite{K21}). Although the results in \cite{L00} are stated for the case when $D$ is a domain, the arguments provided there remain valid for more general setting considered here, where $D$ is any open set.
\end{proof}

\begin{remark}
Assume that $\theta<d-1+p+\alpha p/2$, $\gamma\in\bR$, and $u\in \psi^{\alpha/2}H_{p,\theta}^{\gamma}(D)$. We claim that  for any  test function $\phi\in C_c^\infty(\bR^d)$, its restriction to $D$, denoted by $\phi|_{D}$,  belongs to the dual space of $\psi^{\alpha/2}H_{p,\theta}^\gamma(D)$. 
As a consequence, we can define a distribution $v$ on $\bR^d$ via the pairing:
$$
(v,\phi)_{\bR^d} := (u,\phi|_D)_D \quad \forall \phi\in C_c^\infty(\bR^d).
$$
This definition is well-posed, and $v$ becomes a distribution on $\bR^d$ with $supp(v)\in \overline{D}$. In particular, if $\gamma\geq0$  and  we can extend  $u\in \psi^{\alpha/2}H_{p,\theta}^{\gamma}(D)$ to all of $\bR^d$ by 
defining it to be zero on $D^c$, then we have  $u\in\cD'(\bR^d)$.

Now we prove the  claim. For any $x_0\in \partial D$, $r>0$, and $\lambda>-1$,  it is known  (see e.g. page 16 of \cite{A91} or \cite[Appendix A.4]{CKR23}) that
$$
\int_{B_r(x_0)} d_x^\lambda dx < \infty.
$$
Using this, one can easily show that
$\psi^n D_x^n \phi|_D \in L_{p',\theta'+\alpha p'/2}(D)$ for all $n\in \bN_+$,
where $1/p+1/p'=1$ and $\theta/p+\theta'/p' = d$.  Applying Lemma \ref{lem_prop} $(iii)$, it follows that $\phi|_D$ lies in the dual space of $\psi^{\alpha/2}H_{p,\theta}^\gamma(D)$, completing the proof of the claim.
\end{remark}

Next, we introduce stochastic Banach spaces.  Throughout this paper, we fix a complete probability space $(\Omega,\cF,\bP)$ equipped with a filtration $\{\cF_t, t\geq0\}$ satisfying the usual conditions.
Let $\cP$ denote the predictable $\sigma$-field generated by $\cF_t$, that is, $\cP$ is the smallest σ-field containing all sets of the form $A\times (s,t]$ where $s<t$ and $A\in \cF_s$. Let $\{w^k_t\}_{k\in\bN}$ be a sequence of  independent one-dimensional Wiener processes adapted to the filtration $\{\cF_t, t\geq0\}$.

For $p\geq 2$, $\theta,\gamma\in\bR$,  and a stopping time $\tau\in(0,\infty]$,  we define the following stochastic Banach spaces on $D$;
$$
\bH_{p,\theta}^\gamma(D,\tau):=L_p(\Omega\times(0,\tau), \cP; H_{p,\theta}^\gamma(D)), \quad \bL_{p,\theta}(D,\tau):=\bH_{p,\theta}^0(D,\tau),
$$
$$
\bH_{p,\theta}^\gamma(D,\tau,l_2):=L_p(\Omega\times(0,\tau), \cP; H_{p,\theta}^\gamma(D,l_2)), \quad \bL_{p,\theta}(D,\tau,l_2):=\bH_{p,\theta}^0(D,\tau,l_2),
$$
$$
U_{p,\theta}^\gamma(D)=L_{p}(\Omega,\cF_0; \psi^{\alpha/2-\alpha/p}B_{p,p;\theta}^{\gamma-\alpha/p}(D)).
$$
We define $\bH_0^\infty(D,\tau)$ to be the set  of all functions of the type
\begin{equation} \label{eq6301049}
    f(t,x)=\sum_{i=1}^{n}1_{(\tau_{i-1},\tau_{i}]}(t) f_{i}(x)
\end{equation}
where $f_{i}\in C_c^\infty(D)$ and $0\leq \tau_{0} \leq \cdots \leq \tau_{n} \leq \tau$ are bounded stopping times. Following the proof of \cite[Theorem 3.10]{KAA} which treats the case $D=\bR^d$, one can show that $\bH_0^\infty(D,\tau)$ is dense in $\bH_{p,\theta}^\gamma(D,\tau)$.  Similarly, we define
 $\bH_0^\infty(D,\tau,l_2)$ to be the space of   all $l_2$-valued functions $g=(g^1,g^2,\cdots)$ such that $g_k=0$ for all sufficiently large $k$, and each component $g^k$ is in $\bH_0^\infty(D,\tau)$.  One can also check that $\bH_0^\infty(D,\tau,l_2)$ is dense in $\bH_{p,\theta}^\gamma(D,\tau,l_2)$.

The space $\frH_{p,\theta,\alpha}^\gamma(D,\tau)$, defined below, will be used as the solution space for the  SPDEs studied in this paper.

\begin{definition}
\label{defn solution}
  Let $p\geq 2, \alpha\in (0,2)$, and $\theta,\gamma\in \bR$.   For any $\cD'(D)$-valued function $u$ defined on $\Omega\times[0,\tau)$, we write $u\in \frH_{p,\theta,\alpha}^\gamma(D,\tau)$ if $u\in \psi^{\alpha/2}\bH_{p,\theta}^\gamma(D,\tau)$, $u(0,\cdot)\in U_{p,\theta}^\gamma(D)$, and there exist $f\in \psi^{-\alpha/2}\bH_{p,\theta}^{\gamma-\alpha}(D,\tau)$ and $g\in \bH_{p,\theta}^{\gamma-\alpha/2}(D,\tau,l_2)$ such that
$$
du = fdt + \sum_{k=1}^\infty g^kdw^k_t, \quad t\leq \tau
$$
in the sense of distributions, that is, 
 for any $\phi \in C_c^\infty(D)$ the equality
\begin{equation} \label{eq3291308}
    (u(t,\cdot),\phi)_D=(u(0,\cdot),\phi)_D+ \int_0^t (f(s,\cdot),\phi)_D ds + \sum_{k=1}^\infty \int_0^t (g^k(s,\cdot),\phi)_D dw^k_t
\end{equation}
holds for all $t\leq \tau$ (a.s.). In this case, we write $\bD u:=f$ and  $\bS u:=g$. 
The norm in $\frH_{p,\theta,\alpha}^\gamma(D,\tau)$ is defined as
\begin{align} \label{eq3291304}
\|u\|_{\frH_{p,\theta,\alpha}^\gamma(D,\tau)}&:= \| \psi^{-\alpha/2} u \|_{\bH_{p,\theta}^\gamma(D,\tau)}+ \|u(0,\cdot)\|_{U_{p,\theta}^\gamma(D)} \nonumber
\\
&\quad + \|\psi^{\alpha/2} f\|_{\bH_{p,\theta}^{\gamma-\alpha}(D,\tau)} + \|g\|_{\bH_{p,\theta}^{\gamma-\alpha/2}(D,\tau,l_2)}.
\end{align}

We also say $u\in \frH_{p,\theta,\alpha,loc}^\gamma(D,\tau)$ if there is a sequence of bounded stopping times $\tau_n\uparrow\tau$ so that $u\in \frH_{p,\theta,\alpha}^\gamma(D,\tau_n)$ for each $n$.
\end{definition}

Below we provide some properties of the space $\frH_{p,\theta,\alpha}^{\gamma}(D,\tau)$. In particular, in Proposition \ref{lem2072238} $(i)$, we show that $\frH_{p,\theta,\alpha}^\gamma(D,\tau)$ is a Banach space.

\begin{prop}
\label{prop holder}
Let   $\gamma, \theta\in \bR$, $T\in(0,\infty)$, $\tau\leq T$ be a bounded stopping time, and $u\in \frH_{p,\theta,\alpha}^{\gamma}(D,\tau)$. 

(i)  If $p\in(2,\infty)$, $1/p<\mu<\nu\leq1/2$, and $\gamma-\nu\alpha-d/p \geq n+\delta$ where $n\in\bN_+$ and $\delta\in(0,1)$, then for each $k=0,1,\cdots,n$, 
\begin{equation} \label{eq4071511}
    \bE \Big|\psi^{\alpha(\nu-1/2)}u \Big|^p_{C^{\mu-1/p}([0,\tau];H_{p,\theta}^{\gamma-\nu \alpha}(D))}\leq N(T) \|u\|_{\frH_{p,\theta,\alpha}^{\gamma}(D,\tau)}^p,
\end{equation}
where $N$ depends only on $d$, $\mu$, $\nu$, $p$, $\theta$, $\alpha$, and $T$.

(ii) If $p\in [2,\infty)$, then $u\in L_p(C([0,\tau];H_{p,\theta}^{\gamma-\alpha/2}(D))$ and
\begin{equation} \label{embed_sup}
\bE \sup_{t\leq \tau} \|u(t,\cdot)\|_{H_{p,\theta}^{\gamma-\alpha/2}(D)}^p \leq N \|u\|_{\frH_{p,\theta,\alpha}^{\gamma}(D,\tau)}^p,
\end{equation}
where $N$ depends only on $d$, $p$, $\theta$, $\alpha$, and $T$.
\end{prop}

\begin{proof}
$(i)$
One needs to repeat the proof of \cite[Section 6]{K21} which treats the case $\alpha=2$.
The difference is that we need to apply \eqref{Lp whole C} with $a=e^{-np\alpha/2}$ instead of \cite[Corollary 4.12]{K21} which treats the case $\alpha=2$. Then we have
\begin{equation*}
    \bE \Big|\psi^{\alpha(\nu-1/2)}(u-u(0,\cdot)) \Big|^p_{C^{\mu-1/p}([0,\tau];H_{p,\theta}^{\gamma-\nu \alpha}(D))}\leq N T^{(\nu-\mu)p}\|u\|_{\frH_{p,\theta,\alpha}^{\gamma}(D,\tau)}^p.
\end{equation*}
Note that if $\gamma_1 > \gamma_2$ and $p\in[2,\infty)$, then $B_{p,p}^{\gamma_1}\subset H_p^{\gamma_2}$. Thus, we have $B_{p,p;\theta}^{\gamma-\alpha/p}(D) \subset H_{p,\theta}^{\gamma-\alpha/2}(D)$, which easily yields \eqref{eq4071511}.
Hence, $(i)$ is proved.

$(ii)$
The case $p>2$ easily follows from \eqref{eq4071511} with $\nu=1/2$.
For $p=2$, we again repeat the proof of \cite[Section 6]{K21} using \eqref{L2 sup} with $a=e^{-n\alpha}$ instead of \eqref{Lp whole C}. 
Thus, \eqref{embed_sup} is proved. It is worth noting that $N$ is independent of $T$ if $p=2$.
The proposition is proved.
\end{proof}

Now we prove that $\frH_{p,\theta,\alpha}^\gamma(D,\tau)$ is a Banach space.

\begin{prop} \label{lem2072238}
Let $\alpha\in(0,2), \gamma\in \bR$, $p\in[2,\infty)$, and $\tau$ be a stopping time.

(i) The space $\frH_{p,\theta,\alpha}^\gamma(D,\tau)$ is a Banach space.

(ii) For any $n\in \bN$ such that $n\geq\gamma$,
    \begin{equation} \label{eq1091823}
        \frH_{p,\theta,\alpha}^n(D,\tau) \bigcap \bigcup_{k=1}^\infty L_p(\Omega,C([0,\tau],C_c^n(G_k))),
    \end{equation}
    where $G_k:=\{d_x>1/k\}$, is dense in $\frH_{p,\theta,\alpha}^\gamma(D,\tau)$.

Moreover, if $u\in \frH_{p,\theta,\alpha}^\gamma(D,\tau)$ satisfies $u(0,\cdot)=0$, then it can be approximated by functions in \eqref{eq1091823} that also have zero initial data.
\end{prop}

\begin{proof}
    $(i)$ We only need to show the completeness of the space. Let $u_n\in \frH_{p,\theta,\alpha}^\gamma(D,\tau)$ be a Cauchy sequence. Then by \eqref{eq3291304}, $(u_n, u_n(0,\cdot), \bD u_n, \bS u_n)$ converges to some $(u, u(0,\cdot), \bD u, \bS u)$ in their corresponding spaces.
    Moreover, due to \eqref{embed_sup}, $u_n$ is Cauchy in $L_p(C([0,\tau\wedge T]; H_{p,\theta}^{\gamma-\alpha/2}(D))$ for any $T>0$.
    Thus, $(u_n(t,\cdot),\phi)\to (u(t,\cdot),\phi)$ as $n\to\infty$,
    which easily leads us that $u$ satisfies \eqref{eq3291308}. Thus, $(i)$ is proved.
    
    $(ii)$ One can prove the claim by following the proof of \cite[Theorem 2.9]{KL99}, which treats the case $\alpha=2$ and $D=\bR^d_+$. Here, we only address the differences briefly.
    In \cite{KL99}, the first step is to approximate $u \in \frH_{p,\theta,\alpha}^\gamma(D,\tau)$ by a sequence of functions compactly supported in $D$ by using \cite[Corollary 1.20]{K99}. In our case, we can use \cite[Proposition 2.2]{L00} to prove the corresponding claim.
 Then, instead of \cite[Theorems 7.1 and 7.2]{KAA}, one just needs to apply \cite[Theorem 2.2]{H21}. 
 The lemma is proved.
\end{proof}

\begin{remark} \label{rem2081915}
$(i)$ The argument presented in \cite[Theorem 2.9]{KL99} leads us that $u \in \frH_{p,\theta,\alpha}^\gamma(D,\tau)$ can be approximated by a sequence of functions, say $u_n$, in the set \eqref{eq1091823} where $u_n(0,\cdot)$, $\bD u_n$, and $\bS u_n$ are infinitely differentiable in $x$ and have compact supports.

$(ii)$ Suppose that $u$ is nonrandom and satisfies $\bS u=0$.  Then, following the argument in \cite[Remark 5.5]{K99}, one can easily check that there is a sequence of nonrandom  $u_n \in \frH_{p,\theta,\alpha}^{\gamma}(D,\tau)$ such that $u_n(\cdot,\cdot)\in C_c^\infty([0,\tau]\times D)$, $\bS u_n=0$  and $u_n\to u$ in $\frH_{p,\theta,\alpha}^{\gamma}(D,\tau)$. 

\end{remark}

\subsection{Nonlocal operator}

We now introduce the class of  nonlocal operators considered in this paper, along with the assumptions we impose on them.

We begin with the concept of a nondegenerate $\alpha$-stable symmetric L\'evy measure.
 A L\'evy measure $\nu$ on $\bR^d$ is a $\sigma$-finite positive measure satisfying $\nu(\{0\})=0$ and
$$
\int_{\bR^d} (1\wedge|y|^2) \nu(dy)<\infty.
$$
The measure $\nu$ is called symmetric if $\nu(-dx)=\nu(dx)$.
For $\alpha\in(0,2)$, we say that   $\nu$ is $\alpha$-stable if for any $c>0$ and measurable set $A\subset \bR^d$, it holds that 
\begin{equation*}
    \nu(cA)=c^{-\alpha}\nu(A).
\end{equation*}
It is well-known (see e.g. \cite[Theorem 14.3]{S13}) that there is a nonnegative finite measure $\mu$ on the unit sphere $S^{d-1}$, called the spherical part of $\nu$, such that
\begin{equation} \label{eq1022000}
  \nu(A)= \int_{S^{d-1}} \int_0^\infty 1_{A}(r\theta) \frac{dr}{r^{1+\alpha}} \,\mu(d\theta), \quad A\subset \bR^d.
\end{equation}

  We now state the  assumption on the family of L\'evy measures $\nu_t = \nu_t(\omega,dy)$ appearing in \eqref{oper}.  Throughout,  we assume that   for each $(\omega,t)\in \Omega\times (0,\tau)$, the measure $\nu_t$ is a symmetric
 $\alpha$-stable  L\'evy measure, and can  be expressed as
\begin{equation*}
  \nu_t(\omega,A)= \int_{S^{d-1}} \int_0^\infty 1_{A}(r\theta) \frac{dr}{r^{1+\alpha}}\,\mu_t(\omega,d\theta), \quad A\subset \bR^d
\end{equation*}
where $\mu_t(\omega,d\theta)$ is the spherical part of $\nu_t(\omega,d\theta)$. 

\begin{assumption} \label{ass_nu}
  $(i)$ For any measurable set $A\subset \bR^d$, $\nu_t(\omega,A)$ is predictable.

  $(ii)$ There exist $\Lambda_0>0$ and (nonrandom) nondegenerate symmetric $\alpha$-stable L\'evy measure $\nu^{(1)}(d\theta)$ such that
    \begin{align*} 
    \nu_t(\omega, d\theta) \geq \nu^{(1)}(d\theta), \quad \forall t\in(0,\tau),
  \end{align*}
  and
  \begin{align*}
\Lambda_0\leq \inf_{\rho\in S^{d-1}} \int_{S^{d-1}} |\rho\cdot\theta|^\alpha \mu^{(1)} (d\theta),
  \end{align*}
 where $\mu^{(1)}(d\theta)$ is the spherical part of $\nu^{(1)}(d\theta)$.

    $(iii)$ There exists $\Lambda_1>0$ such that
      \begin{align*}
    \int_{S^{d-1}} \mu_t(\omega, d\theta)\leq \Lambda_1 <\infty, \quad t\in(0,\tau),
  \end{align*}
  where $\mu_t(\omega, d\theta)$ is the spherical part of $\nu_t(\omega, d\theta)$.
\end{assumption}

In the following lemma, we show that $L_t$ is a  bounded operator from $\psi^{\alpha/2}\bH_{p,\theta}^\gamma(D,\tau)$ to $\psi^{-\alpha/2}\bH_{p,\theta}^{\gamma-\alpha}(D,\tau)$.

\begin{lemma} \label{lem1072135}
  Let $p\in(1,\infty)$ and Assumption \ref{ass_nu} hold. 
 Assume that $\theta\in(d-1-\alpha p/2,d-1+p+\alpha p/2)$ if $D$ is a convex domain with nonempty boundary, and $\theta\in (d-\alpha p/2, d+\alpha p/2)$ if $D$ is an open set with nonempty boundary. 
 
 $(i)$ For any $\gamma\in[0,\alpha]$ and $u\in \psi^{\alpha/2}\bH_{p,\theta}^\gamma(D,\tau)$, $L_tu$ defined as
  \begin{align} \label{eq4252225}
(L_t u, \phi)_D:=(u, L_t\phi)_{D}, \quad \phi\in C^{\infty}_c(D)
\end{align}
is well-defined and belongs to $\psi^{-\alpha}\bH_{p,\theta}^{\gamma-\alpha}(D)$. Moreover,
\begin{align*}
\| \psi^{\alpha/2} L_t u \|_{\bH_{p,\theta}^{\gamma-\alpha}(D,\tau)} \leq N \| \psi^{-\alpha/2} u \|_{\bH_{p,\theta}^{\gamma}(D,\tau)},
\end{align*}
where $N=N(d,p,\theta,\gamma,\alpha,\Lambda_0,\Lambda_1,D)$.

$(ii)$ If $L_t=-(-\Delta)^{\alpha/2}$, then the claims in $(i)$ hold true for any $\gamma\in\bR$ with the constant $N=N(d,p,\theta,\gamma,\alpha,D)$.
\end{lemma}

\begin{proof}
$(i)$
Since $\bH_0^\infty(D,\tau)$ is dense in $\psi^{\alpha/2}\bH_{p,\theta}^{\gamma}(D,\tau)$, we may assume that $u\in \bH_0^\infty(D,\tau)$.

First, the case $\gamma=\alpha$   follows from \cite[Lemma 4.10 $(i)$]{DR24} which treats  nonrandom operators. 
Now let $\gamma=0$. 
By using Lemma \ref{lem_prop} $(iii)$ and the above case $\gamma=\alpha$, for any $\phi\in C_c^\infty(D)$, \eqref{eq4252225} is well-defined, and 
\begin{align*} 
\left|\bE\int_0^\tau(L_t u, \phi)_D dt\right| &= \left|\bE\int_0^\tau(u, L_t\phi)_D dt\right| 
\\
&\leq N \|\psi^{-\alpha/2}u\|_{\bL_{p,\theta}(D,\tau)}\|\psi^{\alpha/2}L_t\phi\|_{\bL_{p',\theta'}(D,\tau)}
\\
&\leq N \|\psi^{-\alpha/2}u\|_{\bL_{p,\theta}(D,\tau)}\|\psi^{-\alpha/2}\phi\|_{\bH_{p',\theta'}^{\alpha}(D,\tau)},
\end{align*}
where $1/p+1/p'=1$ and $\theta/p+\theta'/p'=d$. Again by Lemma \ref{lem_prop} $(iii)$, the claim is proved for $\gamma=0$.
For the case $\gamma \in (0,\alpha)$, it is enough to use  the complex
interpolation of operators (see e.g. \cite[Theorem C.2.6]{HVVW16}).

$(ii)$ If $L_t=-(-\Delta)^{\alpha/2}$, one just needs to repeat the proof of \cite[Corollary 4.5]{CKR23}. We remark that the proof remains valid for general open sets, even though \cite{CKR23} specifically addresses $C^{1,1}$ open sets.
\end{proof}

\subsection{Spatially homogeneous noise} \label{sec_defnoise}

We consider  Gaussian noise defined on $\bR^{d+1}$ that is white in time and spatially homogeneous.  The spatial correlation of the noise is characterized by a 
 nonnegative and nonnegative definite tempered measure $\Pi(dx)$ on $\bR^d$, which satisfies  for some $k\geq0$ 
\begin{equation} \label{eq3101709}
    \int_{\bR^d} \frac{1}{1+|x|^k} \Pi(dx) <\infty 
\end{equation}
and 
$$
\int_{\bR^d} (\varphi*\overline{\varphi})(x) \Pi(dx)\geq0
$$
for any $\varphi \in \cS(\bR^d)$, where  $\overline{\varphi}(x):=\varphi(-x)$.  Condition \eqref{eq3101709} ensures that the integral above is well-defined.  In fact,
 for any $\varphi_1,\varphi_2\in \cS(\bR^d)$ and $l>d+k$, we have the following estimate:
\begin{equation} \label{eq5010022}
    \left|\int_{\bR^d} (\varphi_1*\overline{\varphi_2})(x) \Pi(dx)\right| \leq N\|(1+|y|^{l})\varphi_1\|_{L_\infty(\bR^d)}\|(1+|y|^{l})\varphi_2\|_{L_\infty(\bR^d)},
\end{equation}
for some constant $N$. 
To prove this,  note
\begin{align*}
    &|\varphi_1*\overline{\varphi_2}(x)| \nonumber
    \\
    &\leq N \|(1+|y|^{l})\varphi_1\|_{L_\infty(\bR^d)}\|(1+|y|^{l})\varphi_2\|_{L_\infty(\bR^d)} \int_{\bR^d} (1+|y|^{l})^{-1} (1+|x-y|^{l})^{-1} dy,
\end{align*}
and since $l > d+k$,
\begin{align*}
    &\int_{\bR^d} (1+|y|^{l})^{-1} (1+|x-y|^{l})^{-1} dy = \int_{|y|\leq |x|/2} \cdots + \int_{|y| > |x|/2} \cdots \nonumber
    \\
    &\leq N (1+|x|^{l})^{-1}\int_{|y|\leq |x|/2} dy + N\int_{|y| > (1+|x|)/2} |y|^{-l} dy \nonumber
    \\
    &\leq N(1+|x|^{l-d})^{-1} \leq N(1+|x|^{k})^{-1}.
\end{align*}
Thus, \eqref{eq3101709}  yields \eqref{eq5010022}.

We say that $W$ is a spatially homogeneous noise with a covariance measure $\Pi(dx)$ if, for any $\varphi_1,\varphi_2\in \cS(\bR^{d+1})$,   the random variables $W(\varphi_1)$ and $W(\varphi_2)$ are  mean zero Gaussian with  covariance  given by
\begin{equation} \label{eq3051340}
    \bE[W(\varphi_1)W(\varphi_2)] = \int_0^\infty \int_{\bR^d} \left(\varphi_1(t,\cdot)*\overline{\varphi}_2(t,\cdot)\right) (x) \, \Pi(dx) dt.
\end{equation}
The right-hand side of the expression above is well-defined due to \eqref{eq5010022} and the fact that $\Pi(dx)$ is nonnegative definite.
Since every tempered measure induces a distribution,  we can express \eqref{eq3051340} as follows:
\begin{equation*}
    \bE[W(\varphi_1)W(\varphi_2)] = \int_0^\infty \left(\Pi,\varphi_1(t,\cdot)*\overline{\varphi}_2(t,\cdot)\right)_{\bR^d} dt,
\end{equation*}
where $\Pi \in \cD'(\bR^d)$ is the nonnegative and nonnegative definite distribution corresponding to $\Pi(dx)$ (see also \cite[Theorem 2.1.1]{GV64}).  For convenience, we use the same symbol $\Pi$ to denote the measure and its corresponding distribution.

Next, we consider a stochastic integral with respect to $W$.
Let us define the semi-inner product on $\cS(\bR^d)$ by
\begin{equation} \label{eq4011722}
     \langle \varphi_1,\varphi_2 \rangle_{\cH} := \int_{\bR^d} \left(\varphi_1*\overline{\varphi}_2\right) (x) \, \Pi(dx) = (\Pi, \varphi_1*\overline{\varphi}_2)_{\bR^d},
\end{equation}
and
\begin{equation*}
    \|\varphi_1\|_{\cH}^2 := \langle \varphi_1,\varphi_1 \rangle_{\cH},
\end{equation*}
where $\varphi_1,\varphi_2 \in \cS(\bR^d)$. We remark that this object is well-defined due to \eqref{eq5010022}. Moreover, it follows from \eqref{eq3051340} that $\| \varphi,\varphi\|_{\cH} \geq0$ for any $\varphi \in\cS(\bR^d)$, which allows us to identify functions $\varphi_1$ and $\varphi_2$ satisfying $\| \varphi_1-\varphi_2 \|_{\cH} =0$. Under this identification, by $\cH$ we denote the completion of $\cS(\bR^d)$ with the norm $\| \cdot,\cdot \|_{\cH}$.

By employing the Walsh's stochastic integral, $W(\cdot)$ can be extended to a worthy martingale measure so that for any $g \in L_2(\Omega\times(0,\infty),\cP;\cH)$ the stochastic integral with respect to $W$,
\begin{equation} \label{eq3171503}
    \int_0^{t} \int_{\bR^d} g(s,y)W(dyds),
\end{equation}
is well-defined. For more details, we refer the reader to \cite[Section 2]{D99} or \cite{DF98, Dalang}.

We now present  an equivalent form of the stochastic integral above.  
First, recall that $\cS(\bR^d)$, the Schwartz space, is separable under the Schwartz semi-norms (see e.g. \cite[Definition 2.2.1]{G08} for the definition of the norms). Since $\cH$ is the completion of $\cS(\bR^d)$ under $\cH$-norm, it is also separable due to  inequality \eqref{eq5010022}.  This implies the existence of  a complete orthonormal system $\{e_k\}\subset \cS(\bR^d)$ in $\cH$.  For each $e_k$, since $1_{(0,t)}(s)e_k(x) \in L_2(\Omega\times(0,\infty),\cP;\cH)$, we can define
\begin{equation*}
    w_t^k:=\int_0^t \int_{\bR^d} e_k(y) W(dyds) = \int_0^{\infty} \int_{\bR^d} 1_{(0,t)}(s) e_k(y) W(dyds).
\end{equation*}
Using the covariance formula \eqref{eq3051340}, one can verify that the collection $\{w_t^k\}_{k\in\bN}$ forms a sequence of independent standard Wiener processes.  For example, 
\begin{align*}
    \bE(w_t^k w_s^l)= \int_0^\infty \int_{\bR^d} 1_{(0,t\wedge s)}(r) \langle e_k,e_l \rangle_{\cH} dr = (t \wedge s) \delta_{k,l},
\end{align*}
where $\delta_{k,l}$ is the Kronecker delta function such that $\delta_{k,l}=1$ only if $k=l$. Similarly, one can check that for each $k\neq l$, $w^k_t$ and $w^l_t$ are independent   Gaussian processes.

As a result,  stochastic integral \eqref{eq3171503} can be expressed as an infinite series of It\^o integrals:
\begin{equation} \label{eq3232128}
    \int_0^t \int_{\bR^d} g(s,y)W(dyds) = \sum_{k=1}^\infty \int_0^t \langle g(s,\cdot), e_k \rangle_{\cH} dw_s^k.
\end{equation}
This representation provides a computationally useful way of handling SPDEs driven by $W$.

\begin{example} \label{ex_noise}
    We present some examples of spatially homogeneous noises.

    $(i)$ Let $\Pi(dx):=\delta_0(dx)$. Then, $W$ becomes the space-time white noise. Indeed, by \eqref{eq3051340},
    \begin{equation*} 
    \bE[W(1_{(0,t)\times A})W(1_{(0,s)\times B})] = (t\wedge s) (1_{A}*\overline{1}_{B})(0) = (t\wedge s)|A\cap B|,
\end{equation*}
where $t,s>0$ and $A,B\subset \bR^d$.

    $(ii)$ For $\beta\in(0,d)$, $\Pi(dx):=|x|^{-\beta}dx$ is said to be the Riesz type kernel. Here we only show that $\Pi(dx)$ is nonnegative definite. It is well-known that the Fourier transform of $|x|^{-\beta}$ is given by $c|\xi|^{\beta-d}$ for some $c>0$. Thus,
    \begin{equation*}
        \int_{\bR^d} (\varphi*\overline{\varphi})(x)\Pi(dx) = c\int_{\bR^d} |\cF_d(\varphi)(\xi)|^2|\xi|^{\beta-d} d\xi \geq0.
    \end{equation*}

    $(iii)$ Let $f$ be a nonnegative bounded continuous function satisfying nonnegative definite condition. Then $\Pi(dx):=f(x)dx$ satisfies \eqref{eq3101709}. 
 For instance, we can consider $\Pi(dx):=e^{-|x|^\beta}$ with $\beta\in(0,2]$, which is called the Ornstein-Uhlenbeck type kernel.
    Lastly, let us consider the simplest case $f=1$. In this case, the norm in $\cH$ is given by
    \begin{equation*}
        \|\varphi\|_{\cH}^2=\int_{\bR^d} \varphi*\overline{\varphi}(x)dx = \left( \int_{\bR^d}\varphi dx \right)^2=:c^2,
    \end{equation*}
    where $c\in\bR$.
    Thus, for $e_0\in \cS(\bR^d)$ with unit integral,
        \begin{equation*}
        \|\varphi-ce_0\|_{\cH}^2 = \left( \int_{\bR^d}(\varphi-ce_0) dx \right)^2=0,
    \end{equation*}
    which implies that $\cH$ is a one dimensional space spanned by a single element $e_0$. Moreover, $W$ becomes a one dimensional Wiener process, i.e.,
    \begin{equation*}
        \int_0^t\int_{\bR^d} g(s,y)W(dyds)=\int_0^t\int_{\bR^d} g(s,y)dydw_s,
    \end{equation*}
    where $w_t$ is a Wiener process.
\end{example}

\section{Main results} \label{sec_3}

In this section we present our regularity theory for SPDEs driven by spatially homogeneous noises, as well as for SPDEs  with super-linear multiplicative noises.

\subsection{Semi-linear SPDE driven by Wiener processes}

Consider  the semi-linear SPDE
\begin{equation} \label{linear}
\begin{cases}
du=\left(L_tu+f(u)\right)dt + \sum_{k=1}^\infty g^k(u)dw^k_t,\quad &(t,x)\in(0,\tau)\times D,
\\
u(0,\cdot)=u_0,\quad & x\in D,
\\
u(t,x)=0,\quad &(t,x)\in [0,\tau]\times D^c.
\end{cases}
\end{equation}
Here, $\{w^k_t\}_{k\in\bN}$ is a sequence of  independent one-dimensional Wiener processes relative to the filtration $\{\cF_t, t\geq0\}$.

First, we  address the notion of solution to problem \eqref{linear}.

\begin{definition} \label{def_linear}
Let $\gamma\in\bR$ and $\tau$ be a stopping time.  
For a $\cD'(\bR^d)$-valued function $u$ defined on $\Omega\times [0,\tau)$, we say that $u$ is a solution to \eqref{linear} in  the space $\frH_{p,\theta,\alpha}^{\gamma}(D,\tau)$ if its restriction to $D$ belongs to  $\frH_{p,\theta,\alpha}^{\gamma}(D,\tau)$, $supp(u)\in \overline{D}$, and $u$ satisfies \eqref{linear} in the sense of distributions (see Definition \ref{defn solution}).
\end{definition}

Now we impose our assumption on    $f(u)$ and $g(u)$ which depend on $\omega,t,x$, and $u$. The argument $\omega$ is omitted as usual.

\begin{assumption} \label{ass_semi}
$(i)$ For any $u\in \psi^{\alpha/2}H_{p,\theta}^{\gamma}(D)$, $f(t,x,u)$ and $g(t,x,u)$ are predictable functions taking values in $\psi^{-\alpha/2}H_{p,\theta}^{\gamma-\alpha}(D)$ and $H_{p,\theta}^{\gamma-\alpha/2}(D,l_2)$, respectively. Moreover,
$$
f(0)\in \psi^{-\alpha/2}\bH_{p,\theta}^{\gamma}(D,\tau), \quad g(0) \in \bH_{p,\theta}^{\gamma+\alpha/2}(D,\tau,l_2).
$$

$(ii)$ For any $\varepsilon>0$, there exists a constant $N_\varepsilon>0$ such that for any $\omega,t$ and $u,v\in \psi^{\alpha/2}H_{p,\theta}^{\gamma}(D)$,
\begin{align} \label{eq106121}
&\|f(t,\cdot,u)-f(t,\cdot,v)\|_{H_{p,\theta+\alpha p/2}^{\gamma-\alpha}(D)}^p  + \|g(t,\cdot,u)-g(t,\cdot,v)\|_{H_{p,\theta}^{\gamma-\alpha/2}(D,l_2)}^p \nonumber
\\
&\leq \varepsilon \|u-v\|^p_{H_{p,\theta-\alpha p/2}^{\gamma}(D)} + N_\varepsilon \|u-v\|_{H_{p,\theta}^{\gamma-\alpha/2}(D)}^p.
\end{align}
\end{assumption}

\begin{remark}
Suppose $D$ is bounded. Then by \eqref{eq5111835} and Lemma \ref{lem_prop} $(v)$, if
    \begin{align*}
&\|f(t,\cdot,u)-f(t,\cdot,v)\|_{H_{p,\theta+\alpha p/2}^{\gamma-\alpha}(D)}^p  + \|g(t,\cdot,u)-g(t,\cdot,v)\|_{H_{p,\theta}^{\gamma-\alpha/2}(D,l_2)}^p \nonumber
\\
&\leq N \|u-v\|_{H_{p,\theta'}^{\gamma'}(D)}^p
\end{align*}
for some $\theta'>\theta-\alpha p/2$ and $\gamma'<\gamma$, then the condition \eqref{eq106121} holds true.
\end{remark}

 Here is the first main result of this subsection.  The proof  is placed in Section \ref{proof semilinear}.

\begin{thm} \label{thm_white}
Let $\alpha\in(0,2), \sigma\in(0,1)$, $p\in[2,\infty)$, and $T\in(0,\infty)$. Let $\tau\leq T$ be a bounded stopping time. 
Assume that $\theta\in(d-1,d-1+p)$ if $D$ is a bounded $C^{1,\sigma}$ convex domain, and $\theta\in(d-\alpha/2,d-\alpha/2+\alpha p/2)$ if $D$ is a bounded $C^{1,\sigma}$ open set. Also, suppose that Assumptions \ref{ass_nu} and \ref{ass_semi} are satisfied.

$(i)$ If $\gamma\in [0,\alpha]$, then for any $u_0\in U_{p,\theta}^{\gamma}(D)$ there is a unique solution $u$ to \eqref{linear} in the space $u\in \frH_{p,\theta,\alpha}^{\gamma}(D,\tau)$ (in the sense of Definition \ref{def_linear}). Moreover, we have
\begin{align} \label{eq1092122}
\|u\|_{\frH_{p,\theta,\alpha}^{\gamma}(D,\tau)}\leq N \left(\|u_0\|_{U_{p,\theta}^{\gamma}(D)} + \|\psi^{\alpha/2} f(0)\|_{\bH_{p,\theta}^{\gamma-\alpha}(D,\tau)} + \|g(0)\|_{\bH_{p,\theta}^{\gamma-\alpha/2}(D,\tau,l_2)}\right),
\end{align}
where $N=N(d,p,\theta,\gamma,\alpha,\Lambda_0,\Lambda_1,\sigma,D,T)$.

$(ii)$  If $L_t=-(-\Delta)^{\alpha/2}$, then the assertions in $(i)$ hold for any $\gamma\in\bR$ with the constant $N$ independent of $\Lambda_0$ and $\Lambda_1$.
\end{thm}

\begin{remark} \label{rem4061023}
(i) The condition $\theta\in (d-1,d-1+p)$ is necessary even for the corresponding linear deterministic equation (see \cite[Remark 2.5]{CKR23}).  In contrast,  the stronger assumption $\theta\in(d-\alpha/2,d-\alpha/2+\alpha p/2)$, which applies to $C^{1,\sigma}$ open set, is of a technical nature and is used specifically for handling the deterministic case.

(ii)
Theorem \ref{thm_white} extends the results of  \cite[Theorems 2.7 and 2.8]{DR24}, which treat deterministic equations. 
    In particular, by setting $g(u)=0$, we obtain the regularity results for  deterministic parabolic equations when forcing term $f$ lies in a negative-order Sobolev space. The corresponding elliptic result follows similarly by adapting the proof of \cite[Theorem 2.8]{DR24}.

(iii) The space-time H\"older regularity of solution $u$ can be established using Lemma \ref{lem_prop} $(iv)$ and Proposition \ref{prop holder}.
\end{remark}

The maximum principle is a fundamental tool in the analysis of elliptic and parabolic partial differential equations. Similarly, it plays a significant role in the theory of stochastic partial differential equations (SPDEs). In particular, Krylov investigated the maximum principle and its applications to second-order SPDEs in \cite{K07}. Below, we present a version of the maximum principle for  nonlocal equation \eqref{linear}. The proof is placed in Appendix \ref{appA}.

\begin{theorem}[Maximum principle] \label{thm_max}
Let $u_0\in U_{2,d}^{\alpha/2}(D)$, and $f\in \psi^{-\alpha/2}\bL_{2,d}(D,\tau)$ and $g\in \bL_{2,d}(D,\tau,l_2)$ be independent of $u$.
Let $u \in \frH_{2,d,\alpha}^{\alpha/2}(D,\tau)$ be a solution to \eqref{linear}. Assume that for any $\omega\in \Omega$, $u_0\geq0$, and
\begin{equation} \label{eq3302353}
   1_{u(t,\cdot)<0}f\geq0, \quad 1_{u(t,\cdot)<0}g^k = 0
\end{equation}
$t$-almost everywhere on $(0,\tau)$. Then we have $u(t,\cdot)\geq0$ for any $t\in(0,\tau)$ (a.s.).
\end{theorem}

\begin{remark}
    In the above theorem, $u(t,\cdot)\geq0$ is understood in the sense of \eqref{eq6292218}. Thus, one can deduce that $u(t,x)\geq 0$ for almost every $x\in D$, for any $t\in(0,\tau)$ (a.s.).
\end{remark}

\vspace{1mm}

\subsection{SPDE driven by  spatially homogeneous noise}

We consider the semi-linear SPDEs driven by a spatially homogeneous noise;
 \begin{equation} \label{eq_colored}
\begin{cases}
du=(L_tu + f(u))dt + \xi h(u)\dot{W}, \quad &(t,x)\in(0,\tau)\times D,
\\
u(0,x)=u_0,\quad & x\in D,
\\
u(t,x)=0,\quad &(t,x)\in[0,\tau]\times D^c,
\end{cases}
\end{equation}
where $\xi=\xi(\omega,t,x)$ is a real-valued function and $W$ is a spatially homogeneous noise. In particular, $W$ can be a spatially white or colored noise. The functions $f$ and $h$ are semi-linear real-valued functions.

For a real-valued function $u\in \frH_{p,\theta,\alpha}^{\gamma}(D,\tau)$, we write
$$
du=(L_tu + f(u))dt + \xi h(u)\dot{W}, \quad t\leq \tau; \quad u(0,\cdot)=u_0
$$
if for any  $\phi \in C_c^\infty(D)$ the equality
$$
    (u(t,\cdot),\phi)_D=(u(0,\cdot),\phi)_D+ \int_0^t (L_s u,\phi)_D ds + \int^t_0\int_{\bR^d} \xi h(u)\phi \, W (dsdy)
$$
holds for all $t\leq \tau$ (a.s.). 
Note that, by \eqref{eq4011722}, we have
\begin{equation*}
    \langle g(s,\cdot), e_k \rangle_{\cH} = (\Pi, g(s,\cdot)*\overline{e}_k)_{\bR^d} = (\Pi*e_k, g(s,\cdot))_{\bR^d},
\end{equation*}
where $(\Pi*e_k)(x)=(\Pi,e_k(x-\cdot))_{\bR^d}$. Using this identity along with 
 \eqref{eq3232128} naturally leads to the following definition.

\begin{definition} \label{def_solW}
Let $\gamma \geq 0$, $\{e_k\}$ be a complete orthonormal system of $\cH$, and $w^k_t:=\int^t_0\int_{\bR^d} e_k(y)W(dyds)$.  We say that $u$ is a solution to equation \eqref{eq_colored} in the space $\frH^{\gamma}_{p,\theta,\alpha}(D,\tau)$ if it satisfies equation \eqref{linear}  with $g^k(u)=\xi h(u)(\Pi*e_k)$ in the sense of Definition \ref{def_linear}.
\end{definition}

We impose the following reinforced Dalang’s condition on $\Pi(dx)$.

\begin{assumption}[$s,\gamma$] \label{ass_D}

For $s\in(1,\infty]$ and $\gamma\in(0,\alpha/2)$,

\begin{equation} \label{eq3071615}
    \begin{cases}
    \int_{|x|<1} |x|^{\alpha-2\gamma-d/s-d} \, \Pi(dx) < \infty \, &\text{if } \alpha/2-\gamma-d/2s \in (0,d/2),
    \\
    \int_{|x|< 1}\log(1/|x|)\, \Pi(dx)<\infty \, &\text{if }  \alpha/2-\gamma-d/2s = d/2\\
\int_{|x|<1} \Pi(dx)<\infty \, &\text{if }  \alpha/2-\gamma-d/2s > d/2.
\end{cases}
\end{equation}
\end{assumption}

\begin{remark} \label{rem5111836}
Note that if \eqref{eq3101709} holds, then for any $c>0$, we have $\int_{|x|<c}\Pi(dx)<\infty$. Therefore, 
  no additional restrictions are imposed on $\Pi(dx)$ when  $\alpha/2-\gamma-d/2s > d/2$. In other words, in this case, we are free to consider arbitrary nonnegative and nonnegative definite tempered measures.
\end{remark}

Next we  impose our assumption on $f$ and $h$.
Recall that  $f$ and $h$ depend on $(\omega,t,x,u)$, and $\xi$ depends on $(\omega,t,x)$.

\begin{assumption}[$s,\gamma$] \label{ass_h}
$(i)$ The functions $f(t,x,u)$ and $h(t,x,u)$ are real-valued $\cP\times\cB(\bR^{d+1})$-measurable functions. 

$(ii)$ For any $\omega,t,x,u$, and $v$,
\begin{equation} \label{eq5111852}
    |f(t,x,u)-f(t,x,v)|\leq N|u-v|, \quad |h(t,x,u)-h(t,x,v)|\leq N|u-v|,
\end{equation}
where $N>0$ is a constant.

$(iii)$ The function $\xi(t,x)$ is a real-valued $\cP\times\cB(\bR^{d})$-measurable function, and there exists a constant $\theta_0<(2s\gamma+d)\vee s(\alpha-d)$  such that
\begin{equation}
    \sup_{\omega \in\Omega, t\leq\tau} \|\xi(t,\cdot)\|_{L_{2s,\theta_0}(D)}<\infty, \label{eqn theta_0}
\end{equation}
where $\|\xi(t,\cdot)\|_{L_{\infty,\theta_0}(D)}:= \sup_{x\in D} |\xi(t,x)|$.
\end{assumption}

\begin{remark}
    While the case $s=\infty$ is more straightforward and easier  to handle, the case $s<\infty$ is essential for treating   super-linear equations, as discussed in Section \ref{sec_super}.  For remarks on the admissible range of $\theta_0$, see Remark \ref{rem theta_0}.
\end{remark}

Here is the main result in this subsection. The proof  is provided in Section \ref{sec_colored}.

\begin{thm} \label{mainthm}
Let $\alpha\in(0,2), \sigma\in(0,1), \gamma\in (0,\alpha/2)$, $s\in(1,\infty]$, $p\in[2s/(s-1),\infty)$, and $T\in(0,\infty)$. Let  $\tau\leq T$ be a bounded stopping time. 
Assume that $\theta\in(d-1,d-1+p)$ if $D$ is a bounded $C^{1,\sigma}$ convex domain, and $\theta\in(d-\alpha/2,d-\alpha/2+\alpha p/2)$ if $D$ is a bounded $C^{1,\sigma}$ open set. 
Suppose that Assumptions \ref{ass_nu}, \ref{ass_D} $(s,\gamma)$, and \ref{ass_h} $(s,\gamma)$ hold.
Let $u_0\in U_{p,\theta}^{\gamma}(D)$, $f(0)\in \psi^{-\alpha/2}\bH_{p,\theta}^{\gamma-\alpha}(D,\tau)$, and $h(0)\in \bL_{p,\theta-\theta_0p/2s+\overline{\theta}p}(D,\tau)$, where
\begin{equation*}
    \overline{\theta}:=(-\alpha/2+\gamma+d/2s)\vee(-d/2).
\end{equation*}
Then there is a unique solution $u$ to \eqref{eq_colored} in  $\frH_{p,\theta,\alpha}^{\gamma}(D,\tau)$. Moreover,  we have
\begin{align*}
\|u\|_{\frH_{p,\theta,\alpha}^{\gamma}(D,\tau)} \leq N \left(\|u_0\|_{U_{p,\theta}^{\gamma}(D)} + \|\psi^{\alpha/2} f(0)\|_{\bH_{p,\theta}^{\gamma-\alpha}(D,\tau)} + \|h(0)\|_{\bL_{p,\theta-\theta_0p/2s+\overline{\theta}p}}\right),
\end{align*}
where $N=N(d,p,\theta,s,\theta_0,\gamma,\alpha,\Lambda_0,\Lambda_1,\sigma,D,T)$.
\end{thm}

\begin{remark}
If $\sup_{t,\omega}\sup_x |\xi(t,x)|<\infty$, then \eqref{eqn theta_0} is valid  for all $\theta_0\in \bR$.  Moreover, the  condition $h(0)\in \bL_{p,\theta-\theta_0p/2s+\overline{\theta}p}(D,\tau)$ with $s=\infty$ becomes 
$$
h(0)\in \bL_{p,\theta+\overline{\theta}p}(D,\tau).
$$
\end{remark}

\begin{remark} \label{rem4121114}
    For given $\alpha, s$, and $\Pi$,   Assumption \ref{ass_D} $(s,\gamma)$ imposes restriction on the range of $\gamma$. Below, we examine the admissible range of $\gamma$ for each type of $\Pi$  considered in  Example \ref{ex_noise}.

    $(i)$ Space-time white noise:\\
 When  $\Pi(dx)=\delta_0(dx)$,  $W$ corresponds to  space-time white noise.  In this case,  the integrals in \eqref{eq3071615} diverge whenever $\alpha/2-\gamma-d/2s\in(0,d/2]$. Therefore, as noted in  Remark \ref{rem5111836},  the condition $\alpha/2-\gamma-d/2s>d/2$ is necessary for Theorem \ref{mainthm} to hold.  Since $\alpha\in(0,2)$ and $\gamma\in(0,\alpha/2)$,  satisfying this inequality requires $\alpha>1$ and $d=1$.  Under these conditions, the range of $\gamma$ becomes
    \begin{equation}
\label{theta_0}
        \gamma <\alpha/2-1/2-1/2s.
    \end{equation}
    In particular, when $s=\infty$, this reduces to $\gamma<\alpha/2-1/2$, which matches the condition in \cite[Theorem 3.7]{H21} for equations posed on the whole space $D=\bR$;  see also \cite[Theorem 4.5]{KK12}).

    $(ii)$ Riesz type kernel:\\
 Let $\Pi(dx):=|x|^{-\beta}dx$. In this setting,   Assumption \ref{ass_D} $(s,\gamma)$ holds if
    \begin{equation*}
        \gamma<\alpha/2-\beta/2-d/2s.
    \end{equation*}
   Unlike the space-time white noise case,  this condition allows for arbitrary spatial dimensions $d\geq1$.

    $(iii)$ Bounded continuous functions:\\
Let $\Pi(dx):=f(x)dx$, where $f$ is a nonnegative bounded continuous function satisfying nonnegative definite condition. In this case, the  admissible range for $\gamma$ is 
    \begin{equation*}
        \gamma<\alpha/2-d/s.
    \end{equation*}
 In particular,  when $\Pi(dx)=dx$ and $s=\infty$, Example \ref{ex_noise} $(iii)$ shows that equation \eqref{eq_colored} reduces to one    driven by a single Wiener process, which is a special case of \eqref{linear}.  By formally taking $\gamma\to\alpha/2$ in Theorem \ref{mainthm}, we  recover a result  corresponding to a special case of Theorem \ref{thm_white} with $\gamma=\alpha/2$.
\end{remark}

\begin{remark}
\label{rem theta_0}
We examine the admissible  range of  $\theta_0$  in Assumption \ref{ass_h} $(s,\gamma)$ $(iii)$,  which is $\theta_0<(2s\gamma+d)\vee s(\alpha-d)$.

$(i)$ When $s=\infty$:\\
 $\theta_0$ can be any real number; there is no restriction. 

    $(ii)$ When $W$ is a space-time white noise:\\
 As noted in  Remark \ref{rem4121114} $(i)$, this case requires $d=1$, and the condition becomes  $\theta_0<s(\alpha-1)$. This is because inequality \eqref{theta_0} is equivalent to $2s\gamma+1<s(\alpha-1)$.
    If we formally set $\alpha=2$, then this condition  matches that in \cite[Theorem 4.3]{H20}.
\end{remark}

\subsection{Equations with super-linear multiplicative noise term} \label{sec_super}

We consider the equation
\begin{equation} \label{eq4082244}
\begin{cases}
du=L_tu dt + \xi |u|^{1+\lambda} \dot{W}, \quad &(t,x)\in(0,\infty)\times D,
\\
u(0,x)=u_0,\quad & x\in D,
\\
u(t,x)=0,\quad &(t,x)\in[0,\infty)\times D^c,
\end{cases}
\end{equation}
where $\lambda\geq0$, $\xi=\xi(\omega,t,x)$ is a real-valued function, and $W$ is a spatially homogeneous noise on $\bR^{d+1}$ introduced in Section \ref{sec_defnoise}.
We say that $u\in \frH_{p,\theta,\alpha,loc}^\gamma(D,\infty)$ is a solution to \eqref{eq4082244} if, for any bounded stopping time $\tau$,  $u\in \frH_{p,\theta,\alpha}^\gamma(D,\tau)$ and  it satisfies equation \eqref{eq_colored} with $h(u)=|u|^{1+\lambda}$  on $[0,\tau]$ in the sense of Definition \ref{def_solW}.

\begin{assumption}[$\lambda,\gamma$] \label{ass_lambda}

For $\lambda\geq0$ and $\gamma\in(0,\alpha/2)$,

\begin{equation*}
    \begin{cases}
    \int_{|x|<1} |x|^{\alpha-2\gamma-2d\lambda-d} \, \Pi(dx) < \infty \, &\text{if } \alpha/2-\gamma-d\lambda \in (0,d/2),
    \\
    \int_{|x|< 1}\log(1/|x|)\, \Pi(dx)<\infty \, &\text{if }  \alpha/2-\gamma-d\lambda = d/2\\
  \int_{|x|< 1} \Pi(dx)<\infty \, &\text{if }  \alpha/2-\gamma-d\lambda > d/2.
\end{cases}
\end{equation*}
\end{assumption}

\begin{assumption}
\label{ass2}
 The function $\xi(t,x)$ is $\cP\times \cB(\bR^d)$-measurable and
\begin{equation*}
    \|\xi\|_{L_\infty(\Omega\times (0,\infty)\times D)}\leq N.
\end{equation*}
\end{assumption}

\begin{remark}
Note that Assumption \ref{ass_D} $(s,\gamma)$ is identical to  Assumption \ref{ass_lambda}$(\lambda,\gamma)$ when $2\lambda=1/s$.
\end{remark}
Here is the main result of this subsection.

\begin{theorem} \label{thm_super}
    Let $\lambda\in[0,1/2)$, $\gamma\in (0,\alpha/2)$, $\alpha\in(0,2), \sigma\in(0,1)$, 
    \begin{equation}
\label{p}
        p>\frac{d+\alpha}{\gamma} \wedge \frac{2}{1-2\lambda},
    \end{equation}
    and
    \begin{equation} \label{eq4081106}
        \theta < \left\{d+p-1-\frac{\alpha p}{2} + \left[\frac{\gamma p}{\lambda} \vee \left(\frac{(\alpha-d)p}{2\lambda} -dp\right)\right]\right\}\wedge \left(d+\frac{\alpha p}{2}-\gamma p\right).
    \end{equation}
    Suppose that Assumptions \ref{ass_nu}, \ref{ass_lambda} $(\lambda,\gamma)$, and \ref{ass2} are satisfied.
Assume that $\theta\in(d-1,d-1+p)$ if $D$ is a bounded $C^{1,\sigma}$ convex domain, and $\theta\in(d-\alpha/2,d-\alpha/2+\alpha p/2)$ if $D$ is a bounded $C^{1,\sigma}$ open set.
Then for any nonnegative $u_0\in U_{p,\theta}^{\gamma}(D)$, there is a unique solution $u$ to \eqref{eq4082244} such that $u\in \frH_{p,\theta,\alpha,loc}^{\gamma}(D,\infty)$. Moreover, $u\geq0$ for all $t\in(0,\tau)$ (a.s.).
\end{theorem}

\begin{remark}
\label{compatible}
(i) \eqref{eq4081106}  is compatible with the condition $\theta\in (d-1, d-1+p)$, as well as with the condition $\theta\in(d-\alpha/2,d-\alpha/2+\alpha p/2)$, which is assumed in Theorem \ref{thm_super}. To see this,  since 
$d+\frac{\alpha p}{2}-\gamma p>d$, it suffices to verify the inequality
\begin{equation*}
    d-\frac{\alpha}{2}<d-1+p-\frac{\alpha p}{2} + \left[\frac{\gamma p}{\lambda} \vee \left(\frac{(\alpha-d)p}{2\lambda} -dp\right)\right].
\end{equation*}
This inequality easily follows from $\alpha\in(0,2)$ and
\begin{equation*}
    \frac{\gamma p}{\lambda} \vee \left(\frac{(\alpha-d)p}{2\lambda} -dp\right) \geq 0.
\end{equation*}
(ii) Obviously, if $d\geq 2$ then \eqref{eq4081106} is 
$$
    \theta <  \left(d+p-1-\frac{\alpha p}{2} + \frac{\gamma p}{\lambda} \right) \wedge \left(d+\frac{\alpha p}{2}-\gamma p\right).
$$
If $d=1$ and $W$ is a space-time white noise, then \eqref{eq4081106} becomes a simpler inequality, as discussed in Remark \ref{d=1} below.
\end{remark}

\begin{remark}
It can be seen from the conditions in Assumption \ref{ass_lambda} $(\lambda,\gamma)$ that two parameters $\lambda$ and $\gamma$ are interdependent.  Specifically, choosing $\lambda$ near its upper bound requires taking  $\gamma>0$ to be sufficiently small.
    Below we describe the admissible ranges  of $\lambda$ and $\gamma$ for the different types of noise  introduced in Example \ref{ex_noise}.

    $(i)$ Space-time white noise:\\
 As noted in Remark \ref{rem4121114} $(i)$,   we need 
\begin{equation}
\label{white}
d=1 \quad \text{and} \quad 0<\gamma<\alpha/2-1/2-\lambda.
\end{equation}
  This requires $\alpha>1$ and  leads to the admissible range of $\lambda$:
    \begin{equation*}
        0\leq \lambda<\alpha/2-1/2.
    \end{equation*}
    Moreover, when  $\lambda$ is chosen  close to its upper bound,  $\gamma$ must be taken close to $0$.

    $(ii)$ Riesz type kernel. \\
For $\Pi(dx)=|x|^{-\beta}dx$,  the condition becomes
    \begin{equation*}
        \alpha-\beta-2\gamma-2d\lambda>0.
    \end{equation*}
    Thus,  we need $\alpha>\beta$ and  the admissible range for $\lambda$ is
    \begin{equation*}
        0\leq \lambda < \alpha/2d-\beta/2d.
    \end{equation*}

    $(iii)$ Bounded continuous functions:\\
  In this case, for any $d\geq1$,  $\lambda$ and $\gamma$ must satisfy
    \begin{equation*}
        \alpha/2-\gamma-d\lambda\in(0,d).
    \end{equation*}
This yields the following admissible range for $\lambda$:
    \begin{equation*}
        0\leq \lambda < \alpha/2d.
    \end{equation*}
 \end{remark}

\begin{remark}
\label{d=1}
We discuss on the range of $\theta$ and $p$ when $W$ is space-time white noise. First, note that due to \eqref{white}, we have $(1+\alpha)/\gamma > 2/(1-2\lambda)$, and therefore \eqref{p} becomes 
$$p>2/(1-2\lambda).
$$  Similarly, one can check that \eqref{eq4081106} becomes  
$
\theta<1+\alpha p/2-\gamma p.
$
Therefore, for Theorem \ref{thm_super} to hold, the admissible range of $\theta$   is 
$$
0<\theta< p \wedge (1+\alpha p/2-\gamma p)$$
 if $D$ is a bounded open interval. 
\end{remark}

\section{Proof of Theorem \ref{thm_white}}
\label{proof semilinear}

\subsection{Linear SPDEs}
In this subsection we prove a version of Theorem \ref{thm_white} for  linear SPDEs. Throughout this subsection we assume $f$ and $g$ are independent of the unknown $u$.

First, we establish  higher regularity of the solution $u$, assuming that  free terms are sufficiently regular and  that  $u\in\frH_{p,\theta,\alpha}^{\mu}(D,\tau)$ for some $\mu\geq 0$.

\begin{lemma} \label{lem1081726}
Let  $\gamma \geq \mu \geq0$, $\alpha\in(0,2), \sigma\in(0,1)$, $p\in[2,\infty)$, and $\tau$ be a stopping time. 
 Assume that $D$ is a bounded $C^{1,\sigma}$ open set and $\theta\in(d-1-\alpha p/2,d-1+p+\alpha p/2)$.
Suppose that $u_0\in U_{p,\theta}^{\gamma}(D)$, $f\in \psi^{\alpha/2}\bH_{p,\theta}^{\gamma-\alpha}(D,\tau)$, $g \in \bH_{p,\theta}^{\gamma-\alpha/2}(D,\tau,l_2)$, and 
 $u$ is a solution to \eqref{linear} with $L_t=-(-\Delta)^{\alpha/2}$ in the space $u\in\frH_{p,\theta,\alpha}^{\mu}(D,\tau)$. Then we have $u\in \frH_{p,\theta,\alpha}^{\gamma,\alpha}(D,\tau)$, and furthermore
\begin{align} \label{eq1081729}
 \|\psi^{-\alpha/2}u\|_{\bH_{p,\theta}^{\gamma}(D,\tau)} 
 &\leq N\Big(\|\psi^{-\alpha/2}u\|_{\bH^{\mu}_{p,\theta}(D,\tau)} + \|u_0\|_{U_{p,\theta}^{\gamma}(D)} \nonumber 
 \\
 &\quad \qquad + \|\psi^{\alpha/2}f\|_{\bH_{p,\theta}^{\gamma-\alpha}(D,\tau)} +  \|g\|_{\bH_{p,\theta}^{\gamma-\alpha/2}(D,\tau,l_2)} \Big),
\end{align}
where $N=N(d,p,\theta,\gamma,\mu,\alpha,\sigma,D)$ is indenepent of $\tau$.
\end{lemma}

\begin{proof}

\textbf{0}. We recall some function spaces on $\bR^d$.  For $p \geq 2$ and a stopping time $\tau\in(0,\infty]$, we define
$$
\bH_p^\gamma(\tau):= L_p(\Omega\times(0,\tau),\cP;H_p^\gamma), \quad \bL_p(\tau):=\bH_p^0(\tau)=L_p(\Omega\times(0,\tau),\cP;L_p),
$$
$$
\bH_p^\gamma(\tau,l_2):= L_p(\Omega\times(0,\tau),\cP;H_p^\gamma(l_2)), \quad \bL_p(\tau,l_2):=\bH_p^0(\tau,l_2),
$$
$$
U_p^\gamma:= L_p(\Omega,\cF_0; B_{p,p}^{\gamma-\alpha/p}).
$$
As discussed in Section \ref{sec_func}, if $u$ is a distribution on $D$ with compact support then $u$ can be considered as a distribution on $\bR^d$. In particular, if $u\in \bH^{\gamma}_{p,\theta}(D,\tau)$ has compact support in $D$, then $u\in \bH^{\gamma}_p(\tau)$.

\textbf{1}.   We note that it is sufficient to prove the lemma only for  the case $\gamma\leq\mu+\alpha/2$. This is because, if  $\gamma>\mu+\alpha/2$,  we can proceed as follows.   Since the assumptions on $u_0, f$, and $g$ are satisfied for any $\gamma'\leq \gamma$,  the lemma  applies for all $\gamma'\leq \mu+\alpha/2$ in place of $\gamma$. In particular, this implies $u\in \bH^{\mu+\alpha/2}_p(\tau)$. Now, suppose $\gamma=\mu+k \alpha/2+c$ where $k\in \bN_+$ and $c\in (0,\alpha/2)$. Then,  we can apply the lemma iteratively with  $\mu'=\mu+\alpha/2, \mu+2\alpha/2, \cdots, \mu+(k-1)\alpha/2$, to establish the result for $\gamma'=\mu+k \alpha/2$. Finally, since $\gamma\leq (\mu+k \alpha/2)+\alpha/2$, we can apply the lemma once more with $\mu'=\mu+k \alpha/2$ to extend the result to  $\gamma$.

\vspace{1mm}

\textbf{2}. 
We prove the lemma for the case  $\gamma\leq\mu+\alpha/2$, following the approach used in the proof of \cite[Lemma 5.8]{K99}.
For each $n\in\bZ$, denote 
$$
u_n(t,x):=u(e^{n\alpha}t, e^n x), \quad u_{0n}(x):=u_0(e^nx)
$$
$$
f_n(t,x) := f(e^{n\alpha}t,e^n x), \quad g_n(t,x) := g(e^{n\alpha}t,e^n x), \quad w^k(n)_t:= e^{-n\alpha/2}w^k_{e^{n\alpha}t}.
$$
Let $\{\zeta_n : n\in \bZ\}$ be a collection of functions satisfying \eqref{zeta1}-\eqref{zeta3} with $(c_1,c_2)=(1,e^2)$.
 Then $u_n(\cdot,\cdot){\zeta}_{-n}(e^n\cdot) \in \bH^{\mu}_{p}(e^{-n\alpha}\tau)$ and it is a solution (in the sense of distribution) to the equation
$$
\begin{aligned}
&d(u_n(\cdot,\cdot)\zeta_{-n}(e^n\cdot))(t,x) \nonumber
\\
&= \left(-(-\Delta)^{\alpha/2} (u_n(\cdot,\cdot)\zeta_{-n}(e^n\cdot))(t,x) + F_n(t,x) + e^{n\alpha}f_n(t,x) \zeta_{-n}(e^n x) \right) dt
\\
&\quad + \left(e^{n\alpha/2} g_n^k(t,x) \zeta_{-n}(e^n x)\right) dw^k(n)_t,
\end{aligned}
$$
in $\Omega\times (0,e^{-n\alpha}\tau) \times \bR^d$ with $u_n(0,x) = u_{0n}(x) \zeta_{-n}(e^n x)$, where
\begin{align*}
F_n(t,x) &:= (-\Delta)^{\alpha/2}(u_n(\cdot,\cdot)\zeta_{-n}(e^n\cdot))(t,x) - \zeta_{-n}(e^nx) (-\Delta)^{\alpha/2}u_n(t,x).
\end{align*}

 Due to \eqref{eq4250951} and \cite[Lemma 4.3]{CKR23} with $\gamma'=\mu-\alpha/2\geq \gamma-\alpha$, $F_n(t,\cdot)$ makes sense in $\bR^d$ and satisfies
\begin{align} \label{eq1082115}
    \sum_{n\in\bZ}e^{n(\theta-\alpha p/2)}\|F_n(e^{-n\alpha}t,\cdot)\|_{H_p^{\gamma-\alpha}}^p &\leq \sum_{n\in\bZ}e^{n(\theta-\alpha p/2)}\|F_n(e^{-n\alpha}t,\cdot)\|_{H_p^{\mu-\alpha/2}}^p \nonumber
    \\
    &\leq N\|\psi^{-\alpha/2}u(t,\cdot)\|_{H_{p,\theta}^{\mu}(D)}^p.
\end{align}
Here, we note that \cite[Lemma 4.3]{CKR23}  holds in arbitrary open sets.

Since $u_0\in U_{p,\theta}^{\gamma}(D)$, $f\in \psi^{\alpha/2}\bH_{p,\theta}^{\gamma-\alpha}(D,\tau)$, and $g \in \bH_{p,\theta}^{\gamma-\alpha/2}(D,\tau,l_2)$, we have
\begin{eqnarray*}
&u_{0n}(\cdot) \zeta_{-n}(e^n \cdot)\in U_p^\gamma,&
\\
&F_n, e^{n\alpha}f_n(\cdot,\cdot) \zeta_{-n}(e^n \cdot)\in\bH_{p}^{\gamma-\alpha}(e^{-n\alpha}\tau),&
\\
&e^{n\alpha/2} g_n^k(\cdot,\cdot) \zeta_{-n}(e^n\cdot) \in \bH_{p}^{\gamma-\alpha/2}(e^{-n\alpha}\tau,l_2).&
\end{eqnarray*}
Thus, by \cite[Theorem 4.3]{H21}, we obtain that $u_n(\cdot,\cdot)\zeta_{-n}(e^n \cdot)\in \bH^{\gamma}_{p}(e^{-n\alpha}\tau)$ and
\begin{align*} 
&\|(-\Delta)^{\alpha/2}(u(\cdot,e^n\cdot)\zeta_{-n}(e^n\cdot))\|_{\bH_p^{\gamma-\alpha}(\tau)}^p 
\\
&=e^{n\alpha}\|(-\Delta)^{\alpha/2}(u_n(\cdot,\cdot)\zeta_{-n}(e^n\cdot))\|_{\bH_p^{\gamma-\alpha}(e^{-n\alpha}\tau)}^p 
\\
&\leq Ne^{n\alpha}\|F_{n}(\cdot,\cdot)\|_{\bH_p^{\gamma-\alpha}( e^{-n\alpha}\tau)}^p + Ne^{n\alpha}\|e^{n\alpha}f_{n}(\cdot,\cdot) \zeta_{-n}(e^n\cdot)\|_{\bH_p^{\gamma-\alpha}( e^{-n\alpha}\tau)}^p
\\
&\quad + N e^{n\alpha} \|u_{0n}(\cdot) \zeta_{-n}(e^n \cdot)\|_{U_p^\gamma}^p + Ne^{n\alpha} \|e^{n\alpha/2} g_n(\cdot,\cdot) \zeta_{-n}(e^n\cdot)\|_{\bH_p^{\gamma-\alpha/2}(e^{-n\alpha}\tau,l_2)}^p
\\
&=N\|F_{n}(e^{-n\alpha}\cdot,\cdot)\|_{\bH_p^{\gamma-\alpha}(\tau)}^p + N\|e^{n\alpha} \zeta_{-n}(e^n\cdot) f(\cdot,e^n\cdot)\|_{\bH_p^{\gamma-\alpha/2}(\tau)}^p
\\
&\quad + N \|e^{n\alpha/p}u_{0n}(\cdot) \zeta_{-n}(e^n \cdot)\|_{U_p^\gamma}^p + N\|e^{n\alpha/2} \zeta_{-n}(e^n\cdot) g(\cdot,e^n\cdot)\|_{\bH_p^{\gamma-\alpha/2}(\tau,l_2)}^p.
\end{align*}
We note that in \cite{H21}, $(2.19)$ from \cite{KK12} was used. This result was established with bounded stopping times, and the constant in this result is independent of stopping times. Therefore, we see that $N$ in the above inequalities is independent
 of both $\tau$ and $n$.
By \eqref{eq1082115},
\begin{eqnarray*}
&\sum_{n\in\bZ}e^{n(\theta-\alpha p/2)}\|(-\Delta)^{\alpha/2}(u(\cdot,e^n\cdot)\zeta_{-n}(e^n\cdot))\|_{\bH_p^{\gamma-\alpha}(\tau)}^p 
\\
 &\leq N\Big(\|\psi^{-\alpha/2}u\|_{\bH_{p,\theta}^{\mu}(D,\tau)}^p + \|u_0\|_{U_{p,\theta}^{\gamma}(D)} + \|\psi^{\alpha/2}f\|_{\bH_{p,\theta}^{\gamma-\alpha}(D,\tau)} + \|g\|_{\bH_{p,\theta}^{\gamma-\alpha/2}(D,\tau,l_2)}^p\Big). 
 \end{eqnarray*}
Therefore, this and the relation 
 $$\|u\|_{H^{\gamma}_{p}} \approx \left(\|u\|_{H_p^{\gamma-\alpha}}+\|(-\Delta)^{\alpha/2}u\|_{H_p^{\gamma-\alpha}}\right)$$
lead to 
  \eqref{eq1081729} for $\gamma \leq \mu+\alpha/2$. The lemma is proved. 
\end{proof}

Next, we introduce a probabilistic representation of the solution. Let
$X=\{X_t, t\geq0\}$ be a rotationally symmetric $\alpha$-stable $d$-dimensional L\'evy process such that
\begin{equation*}
   \bE e^{i\xi\cdot X_t}=e^{-|\xi|^\alpha t}, \quad \forall \xi\in\bR^d.
 \end{equation*}
For $x\in \bR^d$, let $\kappa_{D}:=\kappa^x_D:=\inf\{t\geq0: x+X_t\not\in D\}$ be the first exit time of $X$ from $D$.
For bounded measurable functions $f$, we denote
\begin{equation*}
  P_t^D f(x) = \bE [f(x+X_t); \kappa_D>t], \quad t>0, \, x\in \bR^d.
\end{equation*}
Then $\{P_t^D\}_{t\geq0}$ is a Feller semigroup in $L_\infty(D)$ when $D$ is a $C^{1,\sigma}$ open set (see page 68 of \cite{C86}).

\begin{lemma} \label{lem1072321}
  Let $\alpha\in(0,2), \sigma\in(0,1), \gamma\in [0,\alpha]$, $p\in[2,\infty)$, and $\tau$ be a stopping time. 
 Assume that $\theta\in(d-1,d-1+p)$ if $D$ is a bounded $C^{1,\sigma}$ convex domain, and $\theta\in(d-\alpha/2,d-\alpha/2+\alpha p/2)$ if $D$ is a bounded $C^{1,\sigma}$ open set. Suppose that Assumption \ref{ass_nu} is satisfied.
 
 $(i)$ For any $f\in \psi^{-\alpha/2}\bH_{p,\theta}^{\gamma-\alpha}(D,\tau)$, there is a unique solution $u$ to
\begin{equation} \label{eq1072348}
\begin{cases}
d u = (L_tu +f)dt, \quad &(t,x)\in(0,T)\times D,
\\
u(0,x)=0,\quad & x\in D,
\\
u(t,x)=0,\quad &(t,x)\in(0,T)\times D^c
\end{cases}
\end{equation}
in the space $\frH_{p,\theta,\alpha}^{\gamma}(D,\tau)$. Moreover, for this solution, we have
\begin{align} \label{eq1061253}
\|u\|_{\frH_{p,\theta,\alpha}^{\gamma}(D,\tau)}\leq N \|\psi^{\alpha/2} f\|_{\bH_{p,\theta}^{\gamma-\alpha}(D,\tau)},
\end{align}
where $N=N(d,p,\theta,\gamma,\alpha,\Lambda_0,\Lambda_1,\sigma,D)$.

$(ii)$ When $L_t=-(-\Delta)^{\alpha/2}$, the claims in $(i)$ hold true for all $\gamma\in\bR$ where the constant $N$ is independent of $\Lambda_0$ and $\Lambda_1$.
\end{lemma}

\begin{proof}
$(i)$
\textbf{Case 1.} Let $\gamma=\alpha$.
We first establish the a priori estimate \eqref{eq1061253} for $\gamma=\alpha$, assuming that   $u\in \frH_{p,\theta,\alpha}^\alpha(D,\tau)$ is a solution to equation \eqref{eq1072348}. 
Since $\bS u=0$,  there exists $\Omega'\subset\Omega$ of full probability such that for each fixed argument $\omega\in \Omega'$, $u(\omega)$ satisfies the (deterministic) equation
\begin{equation*}
  du(\omega)= \left(L_t u(\omega) + f(\omega)\right)dt.
\end{equation*}
In other words, for any $\phi\in C_c^\infty(D)$ and $t\leq \tau(\omega)$,
\begin{equation*}
    (u(\omega,t,\cdot),\phi)_D = \int_0^t (L_s u(\omega,s,\cdot) + f(\omega,s,\cdot),\phi)_D ds.
\end{equation*}
Here we remark that $\Omega'$ is independent of $\phi$, which follows from the fact that the space of test functions $\cD(D)$ is separable.
Thus, by \cite[Theorem 2.7]{DR24},
\begin{align*}
  \|\psi^{-\alpha/2}u(\omega)\|_{L_p((0,\tau(\omega));H_{p,\theta}^\alpha(D))} \leq N \|\psi^{\alpha/2}f(\omega)\|_{L_p((0,\tau(\omega));L_{p,\theta}(D))},
\end{align*}
where $N$ is independent of $\omega\in \Omega'$ and $T$.  This easily leads to the a priori estimate \eqref{eq1061253}.

Next we prove the existence of solution. Note that for any $\lambda \in [0,1]$, the operator $L^{\lambda}_t:=(1-\lambda)L_t+\lambda (-(-\Delta)^{\alpha/2})$ satisfies  Assumption \ref{ass_nu}, with constants $\Lambda_0$ and $\Lambda_1$ that can be 
chosen independently of $\lambda$ (possibly by adjusting them to other suitable constants).  As a result, the a priori estimate \eqref{eq1061253} also holds for each operator  $L^{\lambda}_t$, with a constant $N$ that does not depend on  $\lambda$.
Therefore, by applying the method of continuity (see e.g. the proof of \cite[Theorem 5.1]{KAA}),  it is sufficient to consider the case where
 $L_t=-(-\Delta)^{\alpha/2}$.

Assume that $f\in \bH_0^\infty(D,\tau)$ and define
\begin{equation} \label{eq6301523}
  u(t,x):= \int_0^t  P_{t-s}^D f(s,\cdot)(x)ds=\int^{\infty}_0 1_{s<t} P_{t-s}^D f(s,\cdot)(x)ds, \quad t\leq \tau.
\end{equation}
Here, we note that $P_t^Df(x)=0$ for $x\in D^c$, which implies that $u$ is a $\cD'(\bR^d)$-valued function.
 Then by \cite[Lemma 5.1]{DR24}, $u$ is a solution to \eqref{eq1072348} with $L_t=-(-\Delta)^{\alpha/2}$, and for each $\omega\in \Omega$, $\psi^{-\alpha/2}u(\omega)\in L_p((0,\tau);H_{p,\theta}^\alpha(D))$. To conclude that $u\in \frH_{p,\theta,\alpha}^\alpha(D,\tau)$, it remains to  prove  the desired measurability of $u$. Since $u$ is continuous in $t$, we only need to prove that it is $\cF_t$-adapted. Recall that $f$ is given in the form \eqref{eq6301049}, which leads to
\begin{equation*}
  P_{t-s}^D f(s,\cdot)(x) = \sum_{i=1}^n 1_{(\tau_{i-1},\tau_i]}(s) P_{t-s}^Df_i (x).
\end{equation*}
By \cite[Lemma 5.1]{DR24}, for each $i$, $P_{t-s}^Df_i\in \psi^{\alpha/2}H_{p,\theta}^\alpha(D)$.
Furthermore, since $P_{t-s}^Df_i (x)$ is independent of $\omega$, for each $t>0$, $(\omega,s)\to 1_{s<t}P_{t-s}^D f(s,\cdot)$ is a $\psi^{\alpha/2}H_{p,\theta}^\alpha(D)$-valued $\cF_t\times \cB(\bR_+)$-measurable function. Thus, by Fubini's theorem, $u$ is $\cF_t$-apapted.

For general $f$, take $f_n \in \bH_0^\infty(D,\tau)$ such that $f_n \to f$ in  $\psi^{-\alpha/2}\bH_{p,\theta}^{\gamma-\alpha}(D,\tau)$. Let $u_n\in \frH_{p,\theta,\alpha}^{\gamma,\alpha}(D,\tau)$ be the solution to equation \eqref{eq1072348} with $f_n$.  Then considering the a priori estimate for $u_n-u_m$, we find the $u_n $ is a Cauchy sequence in $\frH_{p,\theta,\alpha}^{\gamma}(D,\tau)$ and its limit $u\in \frH_{p,\theta,\alpha}^{\gamma}(D,\tau)$ satisfies equation \eqref{eq1072348} with general $f$. 

\vspace{1mm}
\textbf{Case 2.} Let $\gamma=0$. Again, 
we first prove the a priori estimate \eqref{eq1061253}. As in the case $\gamma=\alpha$, it suffices to consider the deterministic case; we assume that $L_t$ and $u$ are independent of $\omega$, and $\tau$ is a  constant. 
Suppose $u\in C_c^\infty([0,\tau]\times D)$ and $u(0,x)=0$, then by integration by parts,  for any deterministic function
 $v\in C_c^\infty([0,\tau]\times D)$ such that $v(\tau,x)=0$,
\begin{align*}
  \int_0^\tau \int_D u(\partial_t v+L_tv) \,dxdt &= - \int_0^\tau \int_D v(\partial_t u-L_tu) \,dxdt.
\end{align*}
 This and Lemma \ref{lem_prop} $(iii)$ yield that
\begin{align} \label{eq1071416}
  \left|\int_0^\tau \int_D u(\partial_t v+L_tv) \,dxdt\right|\leq N \|\psi^{\alpha/2}(\partial_t u-Lu)\|_{\bH_{p,\theta}^{-\alpha}(D,\tau)}\|\psi^{-\alpha/2}v\|_{\bH_{p',\theta'}^{\alpha}(D,\tau)},
  \end{align}
where $1/p+1/p'=1$ and $\theta/p+\theta'/p' = d$.
By  the result when  $\gamma=\alpha$, for any deterministic function $h\in \psi^{-\alpha/2}\bL_{p',\theta'}(D,\tau)$, there is a solution $w \in \psi^{\alpha/2}\bH_{p',\theta'}^\alpha(D,\tau)$ to the following backward equation
\begin{equation*}
\begin{cases}
\partial_t w(t,x)=-L_tw(t,x)+h(t,x),\quad &(t,x)\in(0,\tau)\times D,
\\
w(\tau,x)=0,\quad & x\in D,
\\
w(t,x)=0, \quad & (t,x)\in(0,\tau)\times D^c,
\end{cases}
\end{equation*}
satisfying
\begin{equation*}
  \|\psi^{-\alpha/2}w\|_{\bH_{p',\theta'}^\alpha(D,\tau)} \leq N \|\psi^{\alpha/2} h\|_{\bL_{p',\theta'}(D,\tau)}.
\end{equation*}
Here, we emphasize that $w$ is nonrandom, which can be derived from the representation of solution \eqref{eq6301523}. Thus, by Remark \ref{rem2081915} $(ii)$, there is a sequence of functions $v_n\in C_c^\infty([0,\tau]\times D)$ so that $v_n(\tau)=0$, $v_n\to w$ and $\partial_t v_{n} +L_tv_n \to h$ in their corresponding spaces.
By applying \eqref{eq1071416} with $v_n$ instead of $v$, and letting $n\to \infty$, we obtain that for any $h\in \psi^{-\alpha/2}\bL_{p',\theta'}(D,\tau)$,
\begin{align*}
  \left|\int_0^\tau \int_D uh \,dxdt\right| &\leq N \|\psi^{\alpha/2}(u_t-L_tu)\|_{\bH^{-\alpha}_{p,\theta}(D,\tau)} \|\psi^{\alpha/2} h\|_{\bL_{p',\theta'}(D,\tau)}.
\end{align*}
Since $h\in \psi^{-\alpha/2}\bL_{p',\theta'}(D,\tau)$ is arbitrary, it follows from Lemma \ref{lem_prop} $(iii)$ that
\begin{align} \label{eq6301604}
  \|\psi^{-\alpha/2}u\|_{\bL_{p,\theta}(D,\tau)} \leq N \|\psi^{\alpha/2}(u_t-Lu)\|_{\bH^{-\alpha}_{p,\theta}(D,\tau)},
\end{align}
which proves the a priori estimate \eqref{eq1061253} when $\gamma=0$, given that  $u\in C_c^\infty([0,\tau]\times D)$.
For general $u\in \frH_{p,\theta,\alpha}^0(D,\tau)$, again by Remark \ref{rem2081915} $(ii)$, there is a sequence $u_n \in C_c^\infty([0,\tau]\times D)$ such that $u_n\to u$ in $\frH_{p,\theta,\alpha}^0(D,\tau)$. Moreover, each $u_n$ satisfies \eqref{eq6301604} in place of $u$. Letting $n\to\infty$ in the inequality yields the desired estimate for $u$.  

For the solvability of the equation, due to an approximation argument, we  may assume $f\in \bH_0^\infty(D,\tau)$. Consequently the solvability in  $\frH_{p,\theta,\alpha}^0(D,\tau)$ follows from the result for $\gamma=\alpha$ and the relation 
 $\frH_{p,\theta,\alpha}^\alpha(D,\tau)\subset \frH_{p,\theta,\alpha}^0(D,\tau)$.

\vspace{1mm}

\textbf{Case 3.} Let $\gamma \in (0,\alpha)$.
In Cases \textbf{1} and \textbf{2}, it is proved that the operator $f\to u$ is bounded from $\bH_{p,\theta}^{-\alpha}(D,\tau)$ and $\bH_{p,\theta}^{0}(D,\tau)$ to $\frH_{p,\theta,\alpha}^{0}(D,\tau)$ and $\frH_{p,\theta,\alpha}^{\alpha}(D,\tau)$, respectively.  Thus, by the complex
interpolation of operators (see e.g. \cite[Theorem C.2.6]{HVVW16}), we also have the solvability and the desired estimate  for any $\gamma\in(0,\alpha)$.
\vspace{2mm}

$(ii)$ Now we deal with the case when $L_t=-(-\Delta)^{\alpha/2}$ and $\gamma\in\bR$. By the result in $(i)$, it suffices to consider the cases when $\gamma>\alpha$ and $\gamma<0$. Moreover, as above, we only need to prove the a priori estimate \eqref{eq1061253}. For $\gamma>\alpha$, by applying \eqref{eq1081729} with $\mu=\alpha$, $u_0=0$, and $g=0$,
\begin{align*}
  \|u\|_{\frH_{p,\theta,\alpha}^{\gamma}(D,\tau)}\leq N \left(\|\psi^{\alpha/2} f\|_{\bH_{p,\theta}^{\gamma-\alpha}(D,\tau)} + \|\psi^{-\alpha/2} u\|_{\bH_{p,\theta}^{\alpha}(D,\tau)} \right).
\end{align*}
Then it remains to apply \eqref{eq1061253} with $\gamma=\alpha$.
For $\gamma<0$, one just needs to repeat the above duality and approximation arguments.
The lemma is proved. 
\end{proof}

Next, we prove the solvability of equation \eqref{linear} when  $L_t=-(-\Delta)^{\alpha/2}$, $u_0=0$, and $f=0$.

\begin{lemma} \label{lem1082302}
Let $\alpha\in(0,2), \sigma\in(0,1)$, $p\in[2,\infty)$, $T\in(0,\infty)$, and $\tau\leq T$ be a bounded stopping time. 
Assume that $\theta\in(d-1-\alpha p/2,d-1+p+\alpha p/2)$ and $D$ is a bounded $C^{1,\sigma}$ open set.
Then for any $g\in \bH_0^\infty(D,\tau,l_2)$, there exists $u\in\frH_{p,\theta,\alpha}^{\gamma}(D,\tau)$ for any $\gamma\in\bR$, which is a solution to \eqref{linear} with $L_t=-(-\Delta)^{\alpha/2}$ and $u_0=f=0$.
\end{lemma}

\begin{proof}
Since $g=(g^1,g^2,\dots) \in \bH_0^\infty(D,\tau,l_2)$, there is $k_0\in\bN$ so that $g^k=0$ for $k>k_0$. For each $k \leq k_0$, we represent the component $g^k$ of $g$ as
\begin{equation*}
    g^k(t,x):= \sum_{j=1}^{k_0}  1_{(\tau_{i-1}^j,\tau_i^j]}(t) g^k_j(x),
\end{equation*}
where $g^k_j\in C_c^\infty(D)$ and $0\leq\tau_0^k\leq \dots \leq \tau_{k_0}^k$ are bounded stopping times.
    Define
    \begin{equation*}
        G(t,x):= \sum_{k=1}^\infty \int_0^t g^k(s,x)dw_s^k = \sum_{k,j=1}^{k_0} g^k_j(x) \left(w_{\tau_i^j}^k - w_{\tau_{i-1}^j}^k\right),
    \end{equation*}
    which belongs to the space $\psi^{-\alpha/2}\bH_{p,\theta}^\gamma(D,\tau)$ for any $\gamma\in\bR$. Thus, by Lemma \ref{lem1072321} $(ii)$ with $-(-\Delta)^{\alpha/2} G$ in place of $f$, there exists a unique solution $v\in \frH_{p,\theta,\alpha}^\gamma(D,\tau)$ to the equation
    \begin{equation*}
        dv = \left(-(-\Delta)^{\alpha/2}v -(-\Delta)^{\alpha/2} G\right) dt, \quad (t,x)\in (0,\tau)\times D.
    \end{equation*}
    Then one can easily check that $u:=v+G$ satisfies
        \begin{align*}
        du &= d(v+G) = \left(-(-\Delta)^{\alpha/2}v -(-\Delta)^{\alpha/2} G\right) dt + \sum_{k=1}^\infty g^k dw_t^k
        \\
        &= -(-\Delta)^{\alpha/2}u dt + \sum_{k=1}^\infty g^k dw_t^k,
    \end{align*}
    which implies that $u$ is the desired solution.
\end{proof}

\begin{remark}
    The solution in Lemma \ref{lem1082302} can be represented as
    \begin{align}
\label{rep}
u(t,x)=\sum_{k=1}^\infty \int_0^t P_{t-s}^Dg^k(s,\cdot)(x) dw_s^k, \quad t\leq \tau.
\end{align}
The desired measurability can be proved as in the proof of Lemma \ref{lem1072321} (cf. \eqref{eq6301523}).
We now prove that $u$ is a solution to \eqref{linear}.
Using stochastic Fubini theorem (see e.g. \cite[Lemma 2.7]{K11}), for $\phi\in C_c^\infty(D)$,
\begin{align*}
&\int_0^t (u(s,\cdot),-(-\Delta)^{\alpha/2}\phi)_D ds \nonumber
\\
&= - \sum_{k=1}^\infty \int_0^t \int_D \int_0^s P^D_{s-u}g^k(u,\cdot) (x) (-\Delta)^{\alpha/2} \phi(x) dw^k_u dxds \nonumber
\\
&= \sum_{k=1}^\infty \int_0^t \int_u^t (P^D_{s-u}g^k(u,\cdot), -(-\Delta)^{\alpha/2}\phi)_D dsdw^k_u \nonumber
\\
&= - \sum_{k=1}^\infty \int_0^t \int_u^t (g^k(u,\cdot), P_{s-u}^D (-\Delta)^{\alpha/2}\phi)_D dsdw^k_u.
\end{align*}
Since $\phi\in C_c^\infty(D)$, the relation $\partial_t P_t^D \phi = - P_t^D(-\Delta)^{\alpha/2}\phi$ (see \cite[Lemma 8.4]{ZZ18}) yields
\begin{align*}
&\sum_{k=1}^\infty \int_0^t \int_u^t (g^k(u,\cdot), -P_{s-u}^D (-\Delta)^{\alpha/2}\phi)_D dsdw^k_u 
\\
&= \sum_{k=1}^\infty \int_0^t \int_u^t (g^k(u,\cdot), \partial_s P_{s-u}^D\phi)_D dsdw^k_u
\\
&=\sum_{k=1}^\infty \int_0^t (g^k(u,\cdot), P_{t-u}^D\phi)_D dw^k_u - \sum_{k=1}^\infty \int_0^t (g^k(u,\cdot), \phi)_D dw^k_u
\\
&=  (u(t,\cdot), \phi)_D - \sum_{k=1}^\infty \int_0^t (g^k(u,\cdot), \phi)_D dw^k_u.
\end{align*}
For the last equality, we apply stochastic Fubini theorem again.  To justify  the integral representation of the solution, we require the uniqueness of the solution, which follows from Lemma \ref{lem1092139} below.  We  note that the representation formula \eqref{rep} is not used elsewhere in the paper.
\end{remark}

In the proof of the following lemma, we say that
$\widetilde{\psi}$ is a regularized distance function on a $C^{1,\tau}$ open set $D$ if
\begin{equation} \label{eq4052240}
  N^{-1}\widetilde{\psi}(x) \leq d_x \leq N\widetilde{\psi}(x), \quad \widetilde{\psi}\in C^{1,\tau}(\overline{D}), \quad |D^2_x\widetilde{\psi}(x)|\leq Nd_x^{\tau-1}.
\end{equation}
For the existence of such functions, we refer the reader to \cite{GH80, KK04}.

\begin{lemma} \label{lem1092139}
  Let $\alpha\in(0,2), \sigma\in(0,1)$, $p\in[2,\infty)$, $T\in(0,\infty)$, and $\tau\leq T$ be a bounded stopping time. 
Assume that $\theta\in(d-1,d-1+p)$ if $D$ is a bounded $C^{1,\sigma}$ convex domain, and $\theta\in(d-\alpha/2,d-\alpha/2+\alpha p/2)$ if $D$ is a bounded $C^{1,\sigma}$ open set. Then for any $u_0\in U_{p,\theta}^\alpha(D)$, $f \in \psi^{-\alpha/2}\bL_{p,\theta}(D,\tau)$, and $g \in \bH_{p,\theta}^{\alpha/2}(D,\tau,l_2)$, there is a unique solution $u$ to \eqref{linear} with $L_t=-(-\Delta)^{\alpha/2}$ and $u_0=0$ such that $u\in \frH_{p,\theta,\alpha}^{\alpha}(D,\tau)$. Moreover, for this solution, we have
\begin{align} \label{eq1091753}
\|u\|_{\frH_{p,\theta,\alpha}^{\alpha}(D,\tau)}\leq N\left( \|u_0\|_{U_{p,\theta}^\alpha(D)} + \|\psi^{\alpha/2}f\|_{\bL_{p,\theta}(D,\tau)} + \|g\|_{\bH_{p,\theta}^{\alpha/2}(D,\tau,l_2)}\right),
\end{align}
where $N=N(d,p,\theta,\gamma,\alpha,\Lambda_0,\Lambda_1,\sigma,D,T)$.
\end{lemma}

\begin{proof}
As in Case \textbf{1} in proof of Lemma \ref{lem1072321} $(i)$, one just needs to prove the lemma when $g\in \bH_0^\infty(D,\tau,l_2)$.

By Lemmas \ref{lem1072321} and \ref{lem1082302}, there is a solution $u\in \frH_{p,\theta,\alpha}^\alpha(D,\tau)$. Moreover, to prove \eqref{eq1091753}, due to Lemma \ref{lem1081726}, one just needs to show 
\begin{align} \label{eq1091757}
&\|\psi^{-\alpha/2}u\|_{\bL_{p,\theta}(D,\tau)} \nonumber
\\
&\leq N \left( \|\psi^{-\alpha/2+\alpha/p}u_0\|_{L_{p,\theta}(D)} + \|\psi^{\alpha/2}f\|_{\bL_{p,\theta}(D,\tau)} + \|g\|_{\bL_{p,\theta}(D,\tau,l_2)} \right).
\end{align}

By Proposition \ref{lem2072238} $(ii)$, we can take a sequence of functions $u_m$ in the set \eqref{eq1091823} so that $u_m\to u$ in $\frH_{p,\theta,\alpha}^\gamma(D)$.
Then, due to Lemma \ref{lem1072135}, $u_m(0,\cdot)\to u_0$, $\bD u_m + (-\Delta)^{\alpha/2} u_m \to u + (-\Delta)^{\alpha/2} u =f$ and $\bS u_m \to g$ in their corresponding spaces. Hence, we may assume that $u$ is in the set \eqref{eq1091823} (with sufficiently large $n$).
By Remark \ref{rem2081915} $(i)$, we also can assume that $f$ and $g$ are continuous in $x$-variable.
Thus, for given $x\in D$, we can test \eqref{linear} with $\phi^\varepsilon(\cdot-x):= \phi(\frac{\cdot-x}{\varepsilon})$ where $\phi\in C_c^\infty(D)$, and let $\varepsilon\to 0$. Then, for each $x\in D$ we have
\begin{equation*}
u(t,x)= u_0(x) + \int_0^t \left(-(-\Delta)^{\alpha/2} u(s,x)+f(s,x)\right)ds + \sum_{k=1}^\infty\int_0^t g^k(s,x)dw^k_s,
\end{equation*}
for all $t\leq \tau$ (a.s.).

Let $\kappa >0$, $x\in D$, and $\widetilde{\psi}$ be a regularized distance function.  By applying It\^o's formula to $\widetilde{\psi}^ce^{-\kappa t}|u|^p$ with $c=\theta-d+\alpha-\alpha p/2$,
\begin{align} \label{eq2082117}
  e^{-\kappa t} \widetilde{\psi}^c |u(t,x)|^p &=  \widetilde{\psi}^c |u_0(x)|^p \nonumber
  \\
  &\quad+ \int_0^t \Big[ -p e^{-\kappa s} \widetilde{\psi}^c|u|^{p-2}u (-\Delta)^{\alpha/2}u + p e^{-\kappa s} \widetilde{\psi}^c|u|^{p-2}u f \nonumber
  \\
  &\qquad\quad+ \frac{1}{2}p(p-1) e^{-\kappa s} \widetilde{\psi}^c |u|^p |g|_{l_2}^2 -\kappa e^{-\kappa s} \widetilde{\psi}^c |u|^p \Big]  ds \nonumber
  \\
  &\quad+ \sum_{k=1}^\infty\int_{0}^t p e^{-\kappa s} \widetilde{\psi}^c |u|^{p-2}u g^k dw_s^k,
\end{align}
for all $t\leq \tau$ (a.s.).
Since $u$ is in \eqref{eq1091823},
\begin{align*}
    \int_0^t e^{-2\kappa s} \sum_{k=1}^\infty |u|^{2p-2} |g^k|^2 ds \leq N(T,\kappa) \sup_{s\leq t} |u|^{2p-2} \left(\int_0^t |g|_{l_2}^p ds\right)^{2/p} <\infty \text{ (a.s.)}.
\end{align*}
which implies that the stochastic integral in \eqref{eq2082117} is a local martingale. Thus, there is an increasing sequence of stopping times $\tau_n$ such that if we set $t=\tau_n$ in \eqref{eq2082117} and take the expectation, then the stochastic integral vanishes.
 Hence, by integrating with respect to $x$, we have
\begin{align} \label{eq2082242}
  0&\leq \bE \int_{D} \widetilde{\psi}^c |u_0|^p dx \nonumber
  \\
  &\quad+ \bE\int_0^\tau \int_{D} \Big[ -p e^{-\kappa s} \widetilde{\psi}^c|u|^{p-2}u (-\Delta)^{\alpha/2}u + p e^{-\kappa s} \widetilde{\psi}^c|u|^{p-2}u f \nonumber
  \\
  &\qquad\qquad+ \frac{1}{2}p(p-1) e^{-\kappa s} \widetilde{\psi}^c |u|^p |g|_{l_2}^2 -\kappa e^{-\kappa s} \widetilde{\psi}^c |u|^p \Big]  dx ds.
\end{align}
Here, we note that since $u$ in the set \eqref{eq1091823}, we can use Fubini's theorem to interchange the expectation and the $dxds$-integrals.

Since $D$ is bounded, by $(4.9)$ in \cite{DR24}, there exist constants $\delta, N_1>0$ such that
\begin{align*}
  \int_D |u|^{p}\widetilde{\psi}^{c-\alpha} dx &\leq  N_1\int_D -p\widetilde{\psi}^c|u|^{p-2}u (-\Delta)^{\alpha/2}u dx + N \int_{D} |u|^p 1_{d_x>\delta} dx
  \\
  &\leq  N_1\int_D -p\widetilde{\psi}^c|u|^{p-2}u (-\Delta)^{\alpha/2}u dx + N \int_{D} |u|^p \widetilde{\psi}^{c} dx.
\end{align*}
This together with \eqref{eq2082242} yields 
\begin{align*}
&\bE \int_0^\tau \int_D e^{-\kappa T} |u|^{p}\widetilde{\psi}^{c-\alpha} dx ds \leq N \bE \int_{D} \widetilde{\psi}^c |u_0|^p dx
\\
&\qquad+ N \bE \int_0^\tau \int_D \Big[ \widetilde{\psi}^c |u|^{p-1} |f| +\widetilde{\psi}^c |u|^{p-2} |g|_{l_2}^2 + (N-\kappa) |u|^p \tilde{\psi}^{c}  \Big] dx ds.
\end{align*}
By taking sufficiently large $\kappa>0$ and using Young's inequality, we obtain  estimate \eqref{eq1091757}. The lemma is proved.
\end{proof}

\subsection{Proof of Theorem \ref{thm_white}}
We now ready to prove Theorem \ref{thm_white}.

\begin{proof}

\textbf{Step 1.} Assume $L_t=-(-\Delta)^{\alpha/2}$, and $f$ and $g$ are independent of $u$.

When $\gamma\geq\alpha$, the desired result follows from Lemmas \ref{lem1072321} and \ref{lem1092139}.
Thus, we only need to consider the case when $\gamma<\alpha$. Due to the uniqueness result of Lemma \ref{lem1072321}, it suffices to find a solution satisfying the estimate \eqref{eq1092122}.

By Lemma \ref{lem1072321}, we can take a solution $v$ to \eqref{linear} with $(0,f,0)$ instead of $(u_0,f,g)$. Then by considering $u-v$, we may assume that $f=0$ during the proof.

First, we consider the case $\alpha-1\leq \gamma<\alpha$. By \cite[Lemma A.5]{S23}, one can find 
\begin{eqnarray*}
  u_0^i\in U_{p,\theta}^{\gamma+1}(D), \quad g^i\in \bH_{p,\theta}^{\gamma+1-\alpha/2}(D,\tau,l_2),
\end{eqnarray*}
such that
\begin{eqnarray*}
  u_0=u_{0}^0+ \sum_{i=1}^d D_i \left(\psi u_{0}^i\right), \quad g=g^0+ \sum_{i=1}^d D_i \left( \psi g^i\right),
\end{eqnarray*}
and
\begin{eqnarray} \label{eq1092142}
&\|u_0\|_{U_{p,\theta}^{\gamma}(D)} \approx \sum_{i=0}^d \|u_0^i\|_{U_{p,\theta}^{\gamma+1}(D)},& \nonumber
\\
&\|g\|_{\bH_{p,\theta}^{\gamma-\alpha/2}(D,\tau,l_2)} \approx \sum_{i=0}^d\|g^i\|_{\bH_{p,\theta}^{\gamma+1-\alpha/2}(D,\tau,l_2)}.&
\end{eqnarray}
Since $\gamma+1\geq\alpha$, there is a unique solution $v^i$ to \eqref{linear} with $(u_0^i,0,g^i)$ instead of $(u_0,f,g)$, and it satisfies
\begin{align}
\|v^i\|_{\frH_{p,\theta,\alpha}^{\gamma+1}(D,\tau)}\leq N \left(\|u_0^i\|_{U_{p,\theta}^{\gamma+1}(D)} + \|g^i\|_{\bH_{p,\theta}^{\gamma+1-\alpha/2}(D,\tau,l_2)}\right).
\end{align}
If we denote $v:=v_0+ \sum_{i=1}^d D_i \left(\psi v_i\right)$, then $v \in \frH_{p,\theta,\alpha}^{\gamma}(D,\tau)$ and it is a solution to
\begin{equation*}
\begin{cases}
dv=\left(-(-\Delta)^{\alpha/2}v+h\right)dt + \sum_{k=1}^\infty g^k dw^k_t,\quad &(t,x)\in(0,\tau)\times D,
\\
v(0,\cdot)=u_0,\quad & x\in D,
\\
v(t,\cdot)=0,\quad &(t,x)\in [0,\tau]\times D^c,
\end{cases}
\end{equation*}
where $h=\sum_{i=1}^d D_i\left(\psi(-\Delta)^{\alpha/2}v_i\right) + \sum_{i=1}^d (-\Delta)^{\alpha/2}\left(D_i(\psi v_i)\right)$ in $(0,\tau)\times D$.
By \cite[Theorem 4.1]{L00}, one can see that $D_i\psi$ is a bounded operator from $\psi^{-\alpha/2}\bH_{p,\theta}^{\gamma+1-\alpha}(D,\tau)$ to $\psi^{-\alpha/2}\bH_{p,\theta}^{\gamma-\alpha}(D,\tau)$.
This together with Lemma \ref{lem1072135} leads to $h\in \psi^{-\alpha/2}\bH_{p,\theta}^{\gamma-\alpha}(D,\tau)$. Thus, by Lemma \ref{lem1072321}, there exists a unique solution $w$ to
\begin{equation*}
\begin{cases}
dw=\left(-(-\Delta)^{\alpha/2}w+h\right)dt,\quad &(t,x)\in(0,\tau)\times D,
\\
u(0,\cdot)=0,\quad & x\in D,
\\
u(t,\cdot)=0,\quad &(t,x)\in [0,\tau]\times D^c.
\end{cases}
\end{equation*}
Moreover, $w$ satisfies 
\begin{align} \label{eq1092143}
\|w\|_{\frH_{p,\theta,\alpha}^{\gamma}(D,\tau)} &\leq N \|\psi^{\alpha/2} h\|_{\bH_{p,\theta}^{\gamma-\alpha}(D,\tau)} \nonumber
\\
&\leq N \sum_{i=1}^d \|\psi^{-\alpha/2} v_i \|_{\bH_{p,\theta}^{\gamma+1-\alpha}(D,\tau)} \leq N \|\psi^{-\alpha/2} v_i \|_{\bH_{p,\theta}^{\gamma+1-\alpha}(D,\tau)}.
\end{align}
Therefore, $u:=v-w$ becomes a solution to \eqref{linear} with $f=0$, and \eqref{eq1092122} follows from \eqref{eq1092142}--\eqref{eq1092143}.

Now it remains to consider the case when $\gamma<\alpha-1$.  To do this, one needs to repeat the previous argument for the range $\gamma-n-1\leq \gamma < \alpha-n$ for $n=1,2,\dots$, going step by step for $n=1,2,\cdots$.

\textbf{Step 2.}  We prove the theorem for general  $L_t$, assuming that $f$ and $g$ are independent of $u$.

First, we  prove that  the a priori estimate \eqref{eq1092122} holds given that  $u\in \frH_{p,\theta,\alpha}^\gamma(D,\tau)$ is a solution to \eqref{linear}. By Step 1,  there is unique solution $v\in \frH_{p,\theta,\alpha}^\gamma(D,\tau)$ to equation 
\begin{equation*}
\begin{cases}
dv=\left(-(-\Delta)^{\alpha/2}v+f\right)dt + \sum_{k=1}^\infty g^kdw^k_t,\quad &(t,x)\in(0,\tau)\times D,
\\
v(0,\cdot)=u_0,\quad & x\in D,
\\
v(t,\cdot)=0,\quad &(t,x)\in [0,\tau]\times D^c
\end{cases}
\end{equation*}
Moreover, we also have
\begin{align} \label{eq1092200}
\|v\|_{\frH_{p,\theta,\alpha}^{\gamma}(D,\tau)}\leq N \left(\|u_0\|_{U_{p,\theta}^{\gamma}(D)} + \|\psi^{\alpha/2} f\|_{\bH_{p,\theta}^{\gamma-\alpha}(D,\tau)} + \|g\|_{\bH_{p,\theta}^{\gamma-\alpha/2}(D,\tau,l_2)}\right).
\end{align}
Then $w:=u-v \in \frH_{p,\theta,\alpha}^\gamma(D,\tau)$ and it satisfies
\begin{equation*}
\begin{cases}
dw=\left(L_tw + (L_tv + (-\Delta)^{\alpha/2}v) \right)dt,\quad &(t,x)\in(0,\tau)\times D,
\\
v(0,x)=0,\quad & x\in D,
\\
v(t,x)=0,\quad &(t,x)\in [0,\tau]\times D^c.
\end{cases}
\end{equation*}
By Lemma \ref{lem1072135},  we have $L_tv + (-\Delta)^{\alpha/2}v\in \psi^{-\alpha/2}\bH_{p,\theta}^{\gamma-\alpha}(D,\tau)$, and therefore using Lemma \ref{lem1072321} we get
\begin{align} \label{eq1092201}
  \|w\|_{\frH_{p,\theta,\alpha}^{\gamma}(D,\tau)}\leq N \|\psi^{-\alpha/2}v\|_{\bH_{p,\theta}^{\gamma}(D,\tau)}.
\end{align}
Since $u=v+w$, the desired estimate for $u$ follows from \eqref{eq1092200} and \eqref{eq1092201}.

As explained in the proof of Lemma \ref{lem1072321}, the a priori estimate \eqref{eq1092122} holds for any operator  $L^{\lambda}_t$ with $\lambda\in [0,1]$, and the constant  $N$ is independent of $\lambda$. Therefore, by the method of continuity, it is sufficient to prove the solvability of the problem when $\lambda=1$, i.e., when $L_t=-(-\Delta)^{\alpha/2}$.  This case has already been addressed in Step 1.  Thus, the theorem is proved under the assumption that $f$ and $g$ are independent of $u$.

\textbf{Step 3.} We prove the theorem for the semi-linear equation  when $\tau=T$ is a constant.

We first prove the existence result. By the result of Step 2, we can define 
$u^0 \in \frH_{p,\theta,\alpha}^{\gamma}(D,T)$ as  the solution to \eqref{linear} with the inhomogeneous terms $f(0)$ and $g(0)$, which are independent of $u$.
Similarly, for  $n\geq0$,   we  define $u^{n+1}$ inductively as the solution of
\begin{equation*}
\begin{cases}
du^{n+1}=\left(L_tu^{n+1}+f(u^n)\right)dt + \sum_{k=1}^\infty g^k(u^n) dw_t^k,\quad &(t,x)\in(0,T]\times D,
\\
u^{n+1}(0,\cdot)=u_0,\quad & x\in D,
\\
u^{n+1}(t,\cdot)=0,\quad &(t,x)\in [0,T]\times D^c.
\end{cases}
\end{equation*}
Then we get $(u^{n+1}-u^n)(0,\cdot)=0$ and
\begin{align*}
&d(u^{n+1}-u^n)
\\
&=\left(L_t(u^{n+1}-u^n)+ f(u^{n})-f(u^{n-1})\right)dt + \sum_{k=1}^\infty\left(g^k(u^n)-g^k(u^{n-1})\right)dw_t^k.
\end{align*}
Here, we denote $u^{-1}:=0$.
Then by the result of Step 2 and Assumption \ref{ass_semi}, for each $t\leq T$,
\begin{align}
&\|u^{n+1} - u^n\|^p_{\frH_{p,\theta,\alpha}^{\gamma}(D,t)} \nonumber
\\
&\leq N\left( \|f(u^n)-f(u^{n-1})\|^p_{\bH_{p,\theta+\alpha p/2}^{\gamma-\alpha}(D,t)}+ \| g(u^n)-g(u^{n-1})\|^p_{\bH_{p,\theta}^{\gamma-\alpha/2}(D,t,l_2)} \right) \nonumber
\\
&\leq N \varepsilon \|u^n-u^{n-1}\|^p_{\bH_{p,\theta-\alpha p/2}^{\gamma}(D,t)} + N_\varepsilon \|u^n-u^{n-1}\|^p_{\bH_{p,\theta}^{\gamma-\alpha/2} (D,t)} \nonumber
\\
&\leq N\varepsilon \|u^n-u^{n-1}\|^p_{\frH_{p,\theta,\alpha}^{\gamma}(D,t)} + N_\varepsilon \bE\left[\int_0^t \|u^n-u^{n-1}\|_{\frH_{p,\theta,\alpha}^{\gamma}(D,s)}^p ds\right], \label{gronwell}
\end{align}
where $N_\varepsilon$ depends also on $\varepsilon$, but it is independent of $n$.
For the last inequality above, we used
\begin{align*}
\|u\|_{\bH_{p,\theta}^{\gamma-\alpha/2}(D,t)}^p \leq \int_0^t \left( \bE \sup_{r\leq s} \|u(r,\cdot)\|_{H_{p,\theta}^{\gamma-\alpha/2}(D)}^p \right) ds \leq N \int_0^t \|u\|_{\frH_{p,\theta,\alpha}^{\gamma}(D,s)}^p ds,
\end{align*}
which is due to \eqref{embed_sup}.
By the induction, for any $t\leq T$,
\begin{align*}
&\|u^{n+1}-u^n\|^p_{\frH_{p,\theta,\alpha}^{\gamma}(D,t)} \leq (N\varepsilon)^n \|u^1-u^0\|^p_{\frH_{p,\theta,\alpha}^{\gamma}(D,t)} 
\\
&\quad+ \sum_{k=1}^n \binom{n}{k} (N\varepsilon)^{n-k} N_\varepsilon^k \bE\left[\int_0^t \frac{(t-s)^{k-1}}{(k-1)!} \|u^1-u^0\|_{\frH_{p,\theta,\alpha}^{\gamma}(D,s)}^p ds\right],
\end{align*}
where $\binom{n}{k}$ denotes the standard binomial coefficient.
Therefore,
\begin{align*}
\|u^{n+1}-u^{n}\|^p_{\frH_{p,\theta,\alpha}^{\gamma}(D,T)}  &\leq \left((2N\varepsilon)^n \sup_{k\geq0} \frac{(N_\varepsilon T/\varepsilon)^{k}}{k!} \right) \|u^1-u^0\|_{\frH_{p,\theta,\alpha}^{\gamma}(D,T)}^p.
\end{align*}
Let us choose a sufficiently small $\varepsilon>0$ such that $2N\varepsilon<1/2$.
Since the supremum above is finite, we have
\begin{equation*}
    \sum_{n=1}^\infty \|u^{n+1}-u^{n}\|^p_{\frH_{p,\theta,\alpha}^{\gamma}(D,T)}<\infty,
\end{equation*}
which implies that $u^n$ is a Cauchy  sequence in $\frH_{p,\theta,\alpha}^{\gamma}(D,T)$. One can easily show that its limit, say $u$, is a solution to \eqref{linear}.

Next, we prove the uniqueness of solution.
 Let $u$ and $v$ be solutions to \eqref{linear}. Then $w:=u-v$ satisfies
\begin{equation*}
\begin{cases}
dw=\left(L_tw+f(u)-f(v)\right)dt + \sum_{k=1}^\infty \left(g^k(u) -g^k(v)\right) dw_t^k,\, &(t,x)\in(0,T)\times D,
\\
w(0,\cdot)=0,\, & x\in D,
\\
w(t,\cdot)=0,\, &(t,x)\in [0,T]\times D^c.
\end{cases}
\end{equation*}
Then, following arguments used to prove \eqref{gronwell}, we get  for $\varepsilon>0$ and  $t\leq T$,
\begin{align*}
    \|w\|^p_{\frH_{p,\theta,\alpha}^{\gamma}(D,t)} \leq N\varepsilon \|w\|^p_{\frH_{p,\theta,\alpha}^{\gamma}(D,t)} + N_\varepsilon \int_0^t \|w\|_{\frH_{p,\theta,\alpha}^{\gamma}(D,s)}^p ds.
\end{align*}
 By taking sufficiently small $\varepsilon$ and applying Gronwall's lemma, we see that $w=0$, which proves the uniqueness result.  

Lastly, we prove estimate \eqref{eq1092122}. 
Let $u\in \frH^{\gamma}_{p,\theta,\alpha}(D,T)$ be the solution to \eqref{linear}, and define $u^0\in  \frH^{\gamma}_{p,\theta,\alpha}(D,T)$ as above.  Then by  the arguments used to prove \eqref{gronwell},  for any $t\leq T$
\begin{align*}
    &\|u-u^{0}\|^p_{\frH_{p,\theta,\alpha}^{\gamma}(D,t)} \nonumber
    \\
    &\leq N \left( \|f(u)-f(0)\|_{\bH_{p,\theta+\alpha p/2}^{\gamma-\alpha}(D,t)} + \|g(u)-g(0)\|_{\bH_{p,\theta}^{\gamma-\alpha/2}(D,t,l_2)}\right) \nonumber
    \\
&\leq N \varepsilon \|u\|^p_{\bH_{p,\theta-\alpha p/2}^{\gamma}(D,t)} + N_\varepsilon \|u\|^p_{\bH_{p,\theta}^{\gamma-\alpha/2} (D,t)} \nonumber
\\
&\leq N\varepsilon \|u\|^p_{\frH_{p,\theta,\alpha}^{\gamma}(D,t)} + N_\varepsilon \int_0^t \|u\|_{\frH_{p,\theta,\alpha}^{\gamma}(D,s)}^p ds.
\end{align*}
This together with the estimate for $u^0$ obtained in Step 2 yields 
\begin{align*}
\|u\|^p_{\frH_{p,\theta,\alpha}^{\gamma}(D,t)} 
&\leq  N \int_0^t \|u\|^p_{\frH_{p,\theta,\alpha}^{\gamma}(D,s)} ds
\\
&\quad+ N \|f(0)\|^p_{\bH_{p,\theta+\alpha p/2}^{\gamma-\alpha}(D,T)} + N\|g(0)\|^p_{\bH_{p,\theta}^{\gamma-\alpha/2}(D,T,l_2)},
\end{align*}
for any $t\leq T$.
Applying Gronwall's lemma, we obtain the desired estimate.

\textbf{Step 4.} We prove the therem for general $\tau\leq T$. 
We first  prove the existence  result  and estimate \eqref{eq1092122}. 
Denote
\begin{equation*}
    \bar{f}(t,x,u):=1_{t\leq \tau}f(t,x,u), \quad \bar{g}(t,x,u):=1_{t\leq \tau}g(t,x,u).
\end{equation*}
Then one can easily check that $\bar{f}$ and $\bar{g}$ still satisfy Assumption \ref{ass_semi} with $\tau=T$. Therefore, by Step 3, there exists a solution $u\in \frH^{\gamma}_{p,\theta,\alpha}(D,T)$ to equation \eqref{linear} with $\bar{f}$ and $\bar{g}$ in place of $f$ and $g$, respectively. Since $\bar{f}(u)=f(u)$ and $\bar{g}(u)=g(u)$ for $t\leq \tau$, we conclude that    $u\in \frH^{\gamma}_{p,\theta,\alpha}(D,\tau)$ is a solution to equation \eqref{linear}. The estimate for $u$ also follows from the relation 
$
 \|u\|_{ \frH^{\gamma}_{p,\theta,\alpha}(D,\tau)} \leq  \|u\|_{\frH^{\gamma}_{p,\theta,\alpha}(D,T)}$.

Next we prove the uniqueness result. 
Suppose $u\in \frH_{p,\theta,\alpha}^\gamma(D,\tau)$ is a solution to \eqref{linear}.
 Then, by Step 2,  there is  a solution $\tilde{u}\in \frH_{p,\theta,\alpha}^\gamma(D,T)$ to 
\begin{equation} \label{eq2101309}
\begin{cases}
d\tilde{u}=\left(L_t\tilde{u}+\bar{f}(u)\right)dt + \sum_{k=1}^\infty \bar{g}^k(u) dw_t^k,\quad &(t,x)\in(0,T)\times D,
\\
\tilde{u}(0,\cdot)=u_0,\quad & x\in D,
\\
\tilde{u}(t,\cdot)=0,\quad &(t,x)\in [0,T]\times D^c.
\end{cases}
\end{equation}
Note that in above equation we have $\bar{f}(u)$ and $\bar{g}(u)$, not $\bar{f}(\tilde{u})$ and $\bar{g}(\tilde{u})$. Thus,  $u-\tilde{u}$ satisfies $(u-\tilde{u})(0,\cdot)=0$, and
\begin{equation*}
    d(u-\tilde{u})=L_t(u-\tilde{u}) dt
\end{equation*}
in $(0,\tau)\times D$. The uniqueness result in Lemma \ref{lem1072321} yields that  $u=\tilde{u}$ in $(0,\tau)$. Therefore, in equation \eqref{eq2101309}, we can replace $\bar{f}(u)$ and $\bar{g}(u)$ by $\bar{f}(\tilde{u})$ and $\bar{g}(\tilde{u})$, respectively. This means that  $\tilde{u}$ is a solution to the semi-linear equation \eqref{linear} in $\frH_{p,\theta,\alpha}^\gamma(D,T)$. Therefore, the uniqueness result for general $\tau$ follows from the case $\tau=T$. The theorem is proved.
\end{proof}

\section{Proof of Theorem \ref{mainthm}} 
\label{sec_colored}

We first recall Definition \ref{def_solW}, which shows that equation \eqref{eq_colored} can be rewritten as follows:
\begin{equation} \label{eq3171657}
\begin{cases}
du=(L_t u + f(u))dt + \sum_{k=1}^\infty \xi h(u)(\Pi*e_k) dw^k_t, \quad &(t,x)\in(0,\tau)\times D,
\\
u(0,x)=u_0,\quad & x\in D,
\\
u(t,x)=0,\quad &(t,x)\in[0,\tau]\times D^c.
\end{cases}
\end{equation}
Here, in the case when $\cH$ is a finite dimensional space, we set $e_k=0$ for all sufficiently large $k$.

Since \eqref{eq_colored} is equivalent to \eqref{eq3171657},  our approach to proving  Theorem \ref{mainthm} is to define  $g^k(u):=\xi h(u)(\Pi*e_k)$ and verify that  the functions $f$ and $g$ satisfy Assumption \ref{ass_semi}. Once this is established,   Theorem \ref{thm_white} can be applied.

To carry out this verification, we rely on sharp estimates for the convolution  $(R_{\alpha/2-\gamma}^{s/(s-1)}*R_{\alpha/2-\gamma}^{s/(s-1)})(x)$, where  $R_{\alpha/2-\gamma}$ is the kernel of the operator $(1-\Delta)^{(\gamma-\alpha/2)/2}$, which will be described below. 

 Observe that for any $\varphi\in\cS(\bR^d)$,
\begin{align} \label{eq1102244}
\cF_d\{(1-\Delta)^{-\beta/2}\varphi\}(\xi)&=(1+|\xi|^{2})^{-\beta/2}\cF_d\{\varphi\}(\xi) \nonumber
\\
&=c \cF_d\{\varphi\}(\xi) \int_{0}^{\infty}t^{\beta/2}e^{-t}e^{-|\xi|^{2}t}\frac{1}{t}dt,
\end{align}
where $c=c(\gamma)>0$. Define
\begin{align*}
R_{\beta}(x):=\int_{0}^{\infty}t^{\beta/2-1}e^{-t}q(t,x) dt,
\end{align*}
where $q(t,x):=\cF^{-1}_d(e^{-|\cdot|^2 t})(x)$ is the Green's function for the operator $\partial_t-\Delta$ in $\bR^{d+1}$.
Obviously, we have
\begin{equation*}
\int_{\bR^d} R_{\beta}(x)dx =\int_{0}^{\infty}t^{\beta/2-1}e^{-t}\int_{\bR^{d}} q(t,x)dxdt = \int_{0}^{\infty} t^{\beta/2-1}e^{-t} dt < \infty.
\end{equation*}
This yields that we can apply Fubini's theorem to get
$$
\cF_d \left\{ R_{\beta} \right\}(\xi)=N(\gamma,d) \int_{0}^{\infty} t^{\beta/2-1} e^{-t}e^{-|\xi|^{2}t} dt.
$$
 Hence, due to \eqref{eq1102244} for any $\beta>0$ we have 
\begin{equation} \label{negative repre}
(1-\Delta)^{-\beta/2}\varphi=N(\beta,d)\int_{\bR^{d}}R_{\beta}(x-y)\varphi(y)dy.
\end{equation}
 
Below we state the sharp bound of $|R_{\alpha/2-\gamma}^{s/(s-1)}*R_{\alpha/2-\gamma}^{s/(s-1)}(x)|^{(s-1)/s}$, which is presented in the proof of \cite[Proposition 1.7]{CHP24}. We define
       \begin{equation*}
    H_{s,\gamma}(x) :=
    \begin{cases}
     |x|^{\alpha-2\gamma-\frac{d}{s}-d} \quad &\text{if } \alpha/2-\gamma-\frac{d}{2s} \in (0,d/2),
    \\
    |\log|x|| \quad &\text{if } \alpha/2-\gamma-\frac{d}{2s} = d/2,
    \\
    1 \quad &\text{if } \alpha/2-\gamma-\frac{d}{2s} > d/2.
\end{cases}
\end{equation*}

\begin{lemma}
    Let $s\in(1,\infty]$ and $\gamma\in(0,\alpha/2)$ such that $\alpha/2-\gamma-\frac{d}{2s} \in (0,d)$. Then there exist $N=N(s,\alpha,\gamma,d)>0$ and $c=c(s,\alpha,\gamma,d)>0$ such that
\begin{equation} \label{eq4032308}
    |R_{\alpha/2-\gamma}^{s/(s-1)}*R_{\alpha/2-\gamma}^{s/(s-1)})(x)|^{(s-1)/s} \leq NH_{s,\gamma}(x)e^{-c|x|}.
\end{equation}
\end{lemma}

\begin{proof}
    In \cite[Proposition 1.7]{CHP24}, the estimate was derived by considering two separate cases  $|x|<1$ and $|x|\geq1$; there exists $c_0=c_0(s,\alpha,\gamma,d)>0$ such that
        \begin{equation} \label{eq3232140}
    |R_{\alpha/2-\gamma}^{s/(s-1)}*R_{\alpha/2-\gamma}^{s/(s-1)}(x)|^{(s-1)/s} \leq N e^{-c_0|x|}, \quad |x|\geq 1,
\end{equation}
and
    \begin{equation}
        N^{-1} H_{s,\gamma}(x) \leq |R_{\alpha/2-\gamma}^{s/(s-1)}*R_{\alpha/2-\gamma}^{s/(s-1)}(x)|^{(s-1)/s} \leq N H_{s,\gamma}(x), \quad |x|<1. \label{equiv}
    \end{equation}
Combining these two bounds yields the desired estimate, completing the proof of the lemma.
\end{proof}

We now state an equivalent condition of \eqref{eq3071615}.

\begin{lemma}
    Let $s\in(1,\infty]$ and $\gamma\in(0,\alpha/2)$ such that $\alpha/2-\gamma-\frac{d}{2s} \in (0,d)$. Assume that $\Pi(dx)$ be a nonnegative and nonnegative definite tempered measure on $\bR^d$. Then \eqref{eq3071615} is equivalent to
    \begin{equation*}
        \int_{\bR^d} |R_{\alpha/2-\gamma}^{s/(s-1)}*R_{\alpha/2-\gamma}^{s/(s-1)}(x)|^{(s-1)/s} \, \Pi(dx) <\infty.
    \end{equation*}
\end{lemma}

\begin{proof}
    Let us denote
    \begin{align*}
       I_1 + I_2 &:= \int_{|x|<1} |R_{\alpha/2-\gamma}^{s/(s-1)}*R_{\alpha/2-\gamma}^{s/(s-1)}(x)|^{(s-1)/s} \, \Pi(dx)
       \\
       &\quad + \int_{|x|\geq 1} |R_{\alpha/2-\gamma}^{s/(s-1)}*R_{\alpha/2-\gamma}^{s/(s-1)}(x)|^{(s-1)/s} \, \Pi(dx).
    \end{align*}
By \eqref{eq3101709} and \eqref{eq3232140}, we always have $I_2<\infty$.  Moreover, \eqref{equiv} certainly shows that  $I_1<\infty$ is equivalent to condition \eqref{eq3071615}. This completes the proof of the lemma.
\end{proof}

\begin{lemma}
    Let $s\in(1,\infty]$ and $\gamma\in(0,\alpha/2)$ such that $\alpha/2-\gamma-\frac{d}{2s} \in (0,d)$.     Suppose that Assumption \ref{ass_D} $(s,\gamma)$ holds. Then for any $a>a_0>0$,
\begin{align} \label{eq3151438}
    &\int_{\bR^d} |R_{\alpha/2-\gamma}^{s/(s-1)}*R_{\alpha/2-\gamma}^{s/(s-1)}(ax)|^{(s-1)/s} \Pi(dx) \nonumber
    \\
    &\leq
    \begin{cases}
        Na^{\alpha-2\gamma-\frac{d}{s}-d} \quad &\text{ if } \, \alpha/2-\gamma-\frac{d}{2s} \in (0,d/2),
        \\
N \quad &\text{ if } \, \alpha/2-\gamma-\frac{d}{2s} \in [d/2,d),
    \end{cases}
\end{align}
where $N=N(a_0,s,\alpha,\gamma,d)$.
\end{lemma}

\begin{proof}
Assume that $\alpha/2-\gamma-\frac{d}{2s} \in (0,d/2)$.
Since $a>a_0$, it follows from \eqref{eq4032308} that
\begin{align*}
    &\int_{\bR^d} |R_{\alpha/2-\gamma}^{s/(s-1)}*R_{\alpha/2-\gamma}^{s/(s-1)}(ax)|^{(s-1)/s} \Pi(dx)
    \\
    &\leq N \int_{\bR^d} |ax|^{\alpha-2\gamma-\frac{d}{s}-d} e^{-c_0|ax|} \Pi(dx)
    \\
    &\leq N a^{\alpha-2\gamma-\frac{d}{s}-d} \int_{\bR^d} |x|^{\alpha-2\gamma-\frac{d}{s}-d} e^{-c|a_0x|} \Pi(dx),
\end{align*}
where the last integral is finite due to \eqref{eq3101709} and \eqref{eq3071615}.

Now consider the case $\alpha/2-\gamma-d/2s=d/2$. If $a|x|<1$, then $a_0|x|<a|x|<1$ and we have
\begin{equation*}
    |\log |ax||e^{-a|x|} \leq \log(1/a_0|x|)e^{-a_0|x|} \leq (N(a_0)+|\log|x||)e^{-a_0|x|}.
\end{equation*}
If $a|x|>1$, then there exists $c>0$ such that
\begin{equation*}
    |\log |ax||e^{-a|x|} \leq e^{-ca|x|} \leq e^{-ca_0|x|}.
\end{equation*}
Combining these two estimates, we get
\begin{align*}
    &\int_{\bR^d} |R_{\alpha/2-\gamma}^{s/(s-1)}*R_{\alpha/2-\gamma}^{s/(s-1)}(ax)|^{(s-1)/s} \Pi(dx)
    \\
    &\leq N \int_{\bR^d} \left( 1+ |\log|x|| \right) e^{-c|a_0x|} \Pi(dx)<\infty.
\end{align*}
Finally, the case $\alpha/2-\gamma-d/2s \in (d/2,d)$ can be handled in a similar manner.
This completes the proof of the lemma.
\end{proof}

Recall that $\cH$ is separable, and there exists a complete orthonormal system $\{e_k\}\subset \cS(\bR^d)$ for $\cH$.
The following fundamental identity will be used in the proof of Lemma \ref{lem3291751}:
\begin{equation} \label{eq3141037}
    \sum_{k=1}^\infty |\langle f, e_k \rangle_{\cH}|^2 = \| f \|_{\cH}^2, \quad  \forall f\in \cH.
\end{equation}

\begin{lem} \label{lem3291751}
Let $\alpha\in(0,2)$, $\gamma\in(0,\alpha/2)$, $p\geq 2s/(s-1)$, and $g^k(u):=h(u)(\Pi*e_k)$. Suppose that Assumptions \ref{ass_D} $(s,\gamma)$ and \ref{ass_h} $(s,\gamma)$ hold.
Then for any $t$ and $\varepsilon>0$, there exists a constant $N_\varepsilon=N_\varepsilon(d,p,\theta,s,\theta_0,\gamma,\alpha)>0$ such that
\begin{align*}
&\|g(t,\cdot,u)-g(t,\cdot,v)\|_{H_{p,\theta}^{\gamma-\alpha/2}(D,l_2)}^p \nonumber
\\
&\leq \varepsilon \|u-v\|^p_{H_{p,\theta-\alpha p/2}^{\gamma}(D)} + N_\varepsilon \|u-v\|_{H_{p,\theta}^{\gamma-\alpha/2}(D)}^p.
\end{align*}
\end{lem}

\begin{proof}
For simplicity, we omit the variable $t$.
Since $D$ is bounded, there exists $n_0 = n_0(D) \in \bN$ such that
\begin{align} \label{eq1102328}
&\|g(u)-g(v)\|_{H_{p,\theta}^{\gamma-\alpha/2}(D)}^p \nonumber
\\
&= \sum_{n=-\infty}^{n_0} e^{n\theta} \| \xi(e^n\cdot)\zeta_{-n}(e^n\cdot)(g(u)(e^n\cdot)-g(v)(e^n\cdot) ) \|_{H_p^{\gamma-\alpha/2}}^p \nonumber
\\
&= \sum_{n=-\infty}^{n_0} e^{n\theta} \int_{\bR^d} |(1-\Delta)^{(\gamma-\alpha/2)/2} \left(\xi(e^n\cdot) \zeta_{-n}(e^n\cdot) (g(u)(e^n\cdot)-g(v)(e^n\cdot)) \right) (x)|_{l_2}^p dx \nonumber
\\
&=: \sum_{n=-\infty}^{n_0} e^{n\theta} \int_{\bR^d} |I|_{l_2}^p(x) dx,
\end{align}
Let us put $I=(I^1,I^2,\dots)$.  Then, by \eqref{negative repre},
\begin{align} \label{eq1111253}
I^k(x) &= \int_{\bR^d} R_{\alpha/2-\gamma}(x-y) \xi(e^ny) \zeta_{-n}(e^ny)(g^k(u)(e^ny)-g^k(v)(e^ny)) dy \nonumber
\\
&= Ne^{-nd} \int_{\bR^d} R_{\alpha/2-\gamma}(x-e^{-n}y) \xi(y) \zeta_{-n}(y) (h(u)(y)-h(v)(y)) (\Pi*e_k)(y) dy.
\end{align}
Since $(\Pi*e_k)(y)=(\Pi,e_k(y-\cdot))_{\bR^d}$, by Fubini's theorem and the change of variables $y\to y+z$,
\begin{align} \label{eq1111300}
  &\int_{\bR^d} R_{\alpha/2-\gamma}(x-e^{-n}y) \xi(y) \zeta_{-n}(y) (h(u)(y)-h(v)(y)) (\Pi*e_k)(y) dy \nonumber
  \\
  &=\int_{\bR^d} \int_{\bR^d} R_{\alpha/2-\gamma}(x-e^{-n}y) \xi(y) \zeta_{-n}(y) (h(u)(y)-h(v)(y)) e_k(y-z) \, \Pi(dz) dy \nonumber
    \\
  &=\int_{\bR^d} \int_{\bR^d} R_{\alpha/2-\gamma}(x-e^{-n}y-e^{-n}z) \nonumber
  \\
  &\qquad \qquad \quad \times \xi(y+z) \zeta_{-n}(y+z) (h(u)(y+z)-h(v)(y+z)) e_k(y) \, \Pi(dz) dy \nonumber
  \\
  &= \int_{\bR^d} \left[\left(R_{\alpha/2-\gamma}(x-e^{-n}\cdot) \xi \zeta_{-n} (h(u)-h(v))\right)*\overline{e}_k\right](z) \, \Pi(dz) \nonumber
  \\
  &= \langle R_{\alpha/2-\gamma}(x-e^{-n}\cdot) \xi \zeta_{-n} (h(u)-h(v)),e_k \rangle_{\cH}.
\end{align}
Putting $w:=u-v$ and using \eqref{eq4011722}, \eqref{eq3141037}, \eqref{eq1111253}, and \eqref{eq1111300}, we have
\begin{align} \label{eq1112043} 
  &|I|_{l_2}^2(x) = Ne^{-2nd} \sum_{k=1}^\infty \left| \langle R_{\alpha/2-\gamma}(x-e^{-n}\cdot) \xi \zeta_{-n} (h(u)-h(v)),e_k \rangle_{\cH} \right|^2 \nonumber
  \\
  &=Ne^{-2nd}  \| R_{\alpha/2-\gamma}(x-e^{-n}\cdot) \xi \zeta_{-n} (h(u)-h(v)) \|_{\cH}^2 \nonumber
  \\
  &\leq Ne^{-2nd} \int_{\bR^d} \int_{\bR^d} \Big|R_{\alpha/2-\gamma}(x-e^{-n}y) \xi(y) \zeta_{-n}(y) w(y) \nonumber
  \\
  &\qquad \qquad \qquad \times R_{\alpha/2-\gamma}(x-e^{-n}(y-z)) \xi(y-z) \zeta_{-n}(y-z) w(y-z) \Big| dy \Pi(dz) \nonumber
  \\
  &=:N e^{-2nd} \int_{\bR^d} J(x,z) \Pi(dz).
\end{align}
Here, for the last inequality, we used \eqref{eq5111852}.

Below we only handle the case $s<\infty$, as the   the case $s=\infty$ can be treated in a similar manner.
By H\"older's inequality and Assumption \ref{ass_h} $(s,\gamma)$,
\begin{align*}
    &\int_{\bR^d} |\xi(y) \xi(y-z)|^s |\zeta_{-n}(y)\zeta_{-n}(y-z)| dy
    \\
    &\leq \left( \int_{\bR^d} |\xi(y)|^{2s} |\zeta_{-n}(y)|^2 dy  \right)^{1/2} \left( \int_{\bR^d} |\xi(y-z)|^{2s} |\zeta_{-n}(y-z)|^2 dy \right)^{1/2}
    \\
    &= \int_{\bR^d} |\xi(y)|^{2s} |\zeta_{-n}(y)|^2 dy
    \\
    &=Ne^{nd} \int_{\bR^d} |\xi(e^ny)|^{2s} |\zeta_{-n}(e^ny)|^2 dy \leq Ne^{n(d-\theta_0)} \|\xi\|_{L_{2s,\theta_0}(D)}^{2s} \leq Ne^{n(d-\theta_0)}.
\end{align*}
Thus, for $r:=s/(s-1)$, using H\"older's inequality and the change of variables $y\to y+e^nx$,
\begin{align*}
&J(x,z) \leq Ne^{n(d/s-\theta_0/s)}
\\
&\quad \times \Big( \int_{\bR^d} |R_{\alpha/2-\gamma}(e^{-n}y) R_{\alpha/2-\gamma}(e^{-n}(y-z))|^r  w_n(e^nx+y)w_n(e^nx+y-z) dy \Big)^{1/r},
\end{align*}
where $w_n(y):=|w(y)|^{r}|\zeta_{-n}(y)|$.
Then we apply Minkowski's inequality (note $p\geq 2s/(s-1)=2r$) and H\"older's inequality to obtain
\begin{align} \label{eq3141507}
&\int_{\bR^d} |I|_{l_2}^p(x) dx \leq Ne^{-ndp} \int_{\bR^d} \left|\int_{\bR^d} J(x,z) \Pi(dz)\right|^{p/2} dx \nonumber
\\
&\leq N e^{n(-dp+dp/2s-\theta_0p/2s)} \nonumber
  \\
  &\quad\times \Bigg\{ \int_{\bR^d} \bigg[ \int_{\bR^d} \left( \int_{\bR^d} \left| w_n(e^nx+y) w_n(e^nx+y-z) \right|^{p/2r} dx \right)^{2r/p} \nonumber
  \\
  &\qquad \qquad \qquad \times  |R_{\alpha/2-\gamma}(e^{-n}y)|^r |R_{\alpha/2-\gamma}(e^{-n}(y-z))|^r  dy \bigg]^{1/r} \Pi(dz) \Bigg\}^{p/2} \nonumber
  \\
  &= N e^{n(-dp+dp/2s-\theta_0p/2s)} \nonumber
  \\
  &\quad \times \left( \int_{\bR^d} | w_n(e^nx)|^{p/r} dx \right) \times \left( e^{nd/r} \int_{\bR^d} |R_{\alpha/2-\gamma}^r * R_{\alpha/2-\gamma}^r(e^{-n}z)|^{1/r} \Pi(dz)  \right)^{p/2}.
\end{align}
When $\alpha/2-\gamma-\frac{d}{2s} \in (0,d/2)$, by \eqref{eq3151438} with $a_0=e^{-n_0}$ and $r=s/(s-1)$, \eqref{eq1102328}, \eqref{eq1112043}, and \eqref{eq3141507},
\begin{align} \label{eq3151036}
  &\|g(u)-g(v)\|_{H_{p,\theta}^{\gamma-\alpha/2}(D)}^p \nonumber
  \\
  &\leq N \sum_{n=-\infty}^{n_0}  e^{n(\theta+dp/2s-\theta_0p/2s-\alpha p/2+\gamma p)} \int_{\bR^d} |(u-v)(e^nx)|^p|\zeta_{-n}(e^nx)|^{p/r} dx \nonumber
  \\
  &\leq N \|u-v\|_{L_{p,\theta+dp/2s-\theta_0p/2s-\alpha p/2+\gamma p}(D)}^p.
\end{align}
Since $\theta_0<2s\gamma+d$, it follows that 
\begin{equation*}
    \theta+dp/2s-\theta_0p/2s-\alpha p/2+\gamma p > \theta -\alpha p/2,
\end{equation*}
which allows us to Lemma \ref{lem_prop} $(v)$ to obtain the desired result.

Finally, we consider the case when $\alpha/2-\gamma-\frac{d}{2s} \in [d/2,d)$. Using a   similar argument as above, we obtain
\begin{equation} \label{eq3151037}
    \|g(u)-g(v)\|_{H_{p,\theta}^{\gamma-\alpha/2}(D)}^p \leq N \|u-v\|_{L_{p,\theta-dp/2-\theta_0p/2s}(D)}^p.
\end{equation}
Since
\begin{equation*}
    \theta-dp/2-\theta_0p/2s >\theta-\alpha p/2,
\end{equation*}
 we can again apply  Lemma \ref{lem_prop} $(v)$ to obtain the desired estimate. The lemma is proved.
\end{proof}

We now are ready to prove Theorem  \ref{mainthm}.
\begin{proof}[Proof of Theorem \ref{mainthm}]
Since $L_{p,\theta-\alpha p/2}(D) \subset H_{p,\theta+\alpha p/2}^{\gamma-\alpha}(D)$,  for any $t>0$ we have 
\begin{align} \label{eq3152015}
\|f(t,\cdot,u)-f(t,\cdot,v)\|_{H_{p,\theta+\alpha p/2}^{\gamma-\alpha}(D)}^p &\leq N \|f(t,\cdot,u)-f(t,\cdot,v)\|_{L_{p,\theta+\alpha p/2}(D)}^p \nonumber
\\
&= N \int_D |f(t,x,u(x))-f(t,x,v(x))|^p d_x^{\theta-d+\alpha p/2} dx \nonumber
\\
&\leq N \int_D |u(x)-v(x)|^p d_x^{\theta-d+\alpha p/2} dx \nonumber
\\
&= N \|u-v\|_{L_{p,\theta + \alpha p/2}(D)}^p.
\end{align}
Let us take $\kappa \in (0,1)$ such that $0=\kappa (\gamma-\alpha) + (1-\kappa)(\gamma-\alpha/2)$ and define $\tilde{\theta}=\kappa (\theta-\alpha p/2) + (1-\kappa)\theta$.
Since $L_{p,\theta_0}(D)\subset L_{p,\theta+\alpha p/2}(D)\subset H_{p,\theta+\alpha p/2}^{\gamma-\alpha}(D)$, by \eqref{eq3152015} and Lemma \ref{lem_prop} $(v)$,
\begin{align*}
    \|f(t,\cdot,u)-f(t,\cdot,v)\|_{H_{p,\theta+\alpha p/2}^{\gamma-\alpha}(D)}^p &\leq N \|u-v\|^p_{L_{p,\tilde{\theta}}(D)}
    \\
    &\leq \varepsilon \|u-v\|^p_{H_{p,\theta-\alpha p/2}^{\gamma}(D)} + N_\varepsilon \|u-v\|_{H_{p,\theta}^{\gamma-\alpha/2}(D)}^p.
\end{align*}
Combing this with Lemma \ref{lem3291751}, we conclude that $f$ and $g$ satisfy Assumption \ref{ass_semi}.
Moreover, by following the same reasoning  used to prove \eqref{eq3151036} and \eqref{eq3151037},  one can show  that
$$
\|g(0)\|_{H_{p,\theta}^{\gamma-\alpha/2}(D)} \leq N \|h(0)\|_{L_{p,\theta-\theta_0p/2s+\overline{\theta}p}(D,l_2)}.
$$
 Therefore, the theorem follows from  Theorem \ref{thm_white}. 
\end{proof}

\section{Proof of  Theorem \ref{thm_super}} \label{proofsuper}

\begin{lemma} \label{lem4081801}
    Let $\lambda \geq 0$, $\alpha\in(0,2), \sigma\in(0,1), \gamma\in (0,\alpha/2)$, $p\in[2,\infty)$,  and $T\in(0,\infty)$. Let  $\tau\leq T$ be a bounded stopping time. 
Assume that $\theta\in(d-1,d-1+p)$ if $D$ is a bounded $C^{1,\sigma}$ convex domain, and $\theta\in(d-\alpha/2,d-\alpha/2+\alpha p/2)$ if $D$ is a bounded $C^{1,\sigma}$ open set. 
Suppose that Assumptions \ref{ass_nu}, \ref{ass_lambda} $(\lambda,\gamma)$, and \ref{ass2} are satisfied.
Then for any $m\in \bN$ and $u_0\in U_{p,\theta}^{\gamma}(D)$, there is a unique solution $u_m$ to 
\begin{equation} \label{eq4100000}
\begin{cases}
du=L_tu dt + \xi|0\vee u\wedge m|^{1+\lambda} \dot{W}, \quad &(t,x)\in(0,\infty)\times D,
\\
u(0,x)=u_0,\quad & x\in D,
\\
u(t,x)=0,\quad &(t,x)\in[0,\infty)\times D^c,
\end{cases}
\end{equation}
in the space $\frH_{p,\theta,\alpha}^{\gamma}(D,\tau)$.
In particular, if $u_0\geq0$, then $u_m\geq0$.
\end{lemma}

\begin{proof}
Let $h_m(u) := \xi|0\vee u \wedge m|^{1+\lambda}$. Then by the mean-value theorem,
\begin{equation*}
    |h_m(u)-h_m(v)|\leq N(1+\lambda)m^{\lambda}|u-v|,
\end{equation*}
which implies that Assumption \ref{ass_h} $(s,\gamma)$ with $s=\infty$ is satisfied. Thus, by applying Theorem \ref{mainthm} with $h_m$ instead of $h$, we obtain a unique solution $u_m$ to \eqref{eq4100000} such that $u_m \in \frH_{p,\theta,\alpha}^{\gamma}(D,\tau)$.

For simplicity, we fix $m$ and put $u=u_m$. We now prove the positivity of solution $u$.
Take a sequence of $\zeta_n\in C_c^\infty(D)$ satisfying \eqref{zeta1}-\eqref{zeta3}. Here, we further assume that $\sum_{n\in\bZ}\zeta_n=1$ in $D$. Letting $u_0^n:=u_0\times \sum_{|m|\leq n} \zeta_m$, one can show that $u_0^n \geq0$, $u_0^n \in U_{2,d}^\gamma(D)\cap U_{p,\theta}^\gamma(D)$, and $u_0^n \to u_0$ in $U_{p,\theta}^\gamma(D)$. Let us consider
    \begin{equation} \label{eq4051907}
\begin{cases}
dv=L_t v dt + \sum_{k=1}^n \xi|0\vee v \wedge m|^{1+\lambda}(\Pi*e_k) dw^k_t, \quad &(t,x)\in(0,\tau)\times D,
\\
v(0,x)=u_0^n,\quad & x\in D,
\\
v(t,x)=0,\quad &(t,x)\in[0,\tau]\times D^c,
\end{cases}
\end{equation}
where $\{e_k\}$ is a complete orthonormal system of $\cH$. Note that using \eqref{eq3101709} and the fact $e_k\in \cS(\bR^d)$, one can easily show that $(\Pi*e_k)(x)=(\Pi,e_k(x-\cdot))_{\bR^d}$ is bounded for each $k$. Thus, if we define
\begin{align*}
    g^n_m := \left(\xi|0\vee u \wedge m|^{1+\lambda}(\Pi*e_1),\cdots,\xi|0\vee u \wedge m|^{1+\lambda}(\Pi*e_n),0,0,\cdots \right),
\end{align*}
then $g^n_m \in\bL_{p,\theta}(D,\tau,l_2)$ and it satisfies \eqref{eq106121}. This implies that for each $n$, we can find a solution $u^n$ to \eqref{eq4051907} satisfying $u^n\in \frH_{p,\theta,\alpha}^{\alpha/2}(D,\tau) \subset \frH_{p,\theta,\alpha}^{\gamma}(D,\tau)$.

 Since $1_{u<0}g^{n,k}_m = 1_{u<0}\xi|0\vee u \wedge m|^{1+\lambda}(\Pi*e_k) = 0$ for each $k$, we can apply the maximum principle (Theorem \ref{thm_max}) to obtain that $u^n\geq0$ for each $n$.

Now it remains to show that $u^n\to u$ in $\frH_{p,\theta,\alpha}^{\gamma}(D,\tau)$, which leads to $u\geq0$. Note that $v:=u^n-u$ satisfies $v(0,\cdot)=u_0^n-u$, and
\begin{equation*}
    dv=L_t v dt + \left( g_m(u)-g_m^n(u^n) \right) dw^k_t
\end{equation*}
in $(0,\tau)\times D$, where $g_m:=\lim_{n\to\infty} g_m^n$. By \eqref{eq1092122},
\begin{align} \label{eq4052034}
\|u^n-u\|_{\frH_{p,\theta,\alpha}^\gamma(D,\tau)} &\leq N\left(\|u_0^n-u_0\|_{U_{p,\theta}^\gamma(D)} + \|g_m(u)-g_m^n(u^n)\|_{\bH_{p,\theta}^{\gamma-\alpha/2}(D,\tau,l_2)}\right) \nonumber
\\
&\leq N\left(\|u_0^n-u_0\|_{U_{p,\theta}^\gamma(D)} + \|g_m(u)-g_m(u^n)\|_{\bH_{p,\theta}^{\gamma-\alpha/2}(D,\tau,l_2)}\right) \nonumber
\\
&\quad+ N \|g_m(u^n)-g_m^n(u^n)\|_{\bH_{p,\theta}^{\gamma-\alpha/2}(D,\tau,l_2)}.
\end{align}
As in \eqref{eq3151036} and \eqref{eq3151037}, one can show that
\begin{align*}
    \|g_m(u)-g_m(u^n)\|_{\bH_{p,\theta}^{\gamma-\alpha/2}(D,\tau,l_2)} &\leq N(m)\|u-u^n\|_{\bL_{p,\tilde{\theta}}(D,\tau)} 
    \\
    &\leq N(m) \|u-u^n\|_{\bL_{p,\theta-\alpha p/2}(D,\tau)},
\end{align*}
where $\tilde{\theta}=(\theta-\alpha p/2+\gamma p)\vee (\theta-dp/2)$. Obviously, the fact $g_m(u^n)\in \bH_{p,\theta}^{\gamma-\alpha/2}(D,\tau,l_2)$ yields that
\begin{equation*}
    \lim_{n\to\infty} \|g_m(u^n)-g_m^n(u^n)\|_{\bH_{p,\theta}^{\gamma-\alpha/2}(D,\tau,l_2)}=0.
\end{equation*}
Thus, letting $n\to\infty$ in \eqref{eq4052034}, we obtain the desired result. The lemma is proved.
\end{proof}

\begin{lemma}
Suppose the assumptions of Lemma \ref{lem4081801} are satisfied,  and let $u_m$ be a solution to \eqref{eq4100000} such that $u_m\in \frH_{p,\theta,\alpha}^\gamma(D,\tau)$.
In addition, we assume that Assumption \ref{ass_lambda} $(\lambda,\gamma)$ holds,
    \begin{equation*}
        p>\frac{d+\alpha}{\gamma} \wedge \frac{2}{1-2\lambda},
    \end{equation*}
    and
\begin{equation} \label{eq4081803}
    \theta <d-1+p-\frac{\alpha p}{2} + \left[\frac{\gamma p}{\lambda} \vee \left(\frac{(\alpha-d)p}{2\lambda} -dp\right)\right].
\end{equation}
Then we have
\begin{equation} \label{eq4090015}
    \lim_{R\to\infty}\sup_m \bP\left( \sup_{t\leq\tau,x\in D}|\psi^{\gamma+\frac{\theta-d}{p}-\frac{\alpha}{2}}u_m|>R
 \right)=0.
\end{equation}
\end{lemma}

\begin{proof}
Recall, as discussed in Remark \ref{compatible},   that condition \eqref{eq4081803}  is compatible with the condition $\theta\in (d-1, d-1+p)$, as well as with the condition $\theta\in(d-\alpha/2,d-\alpha/2+\alpha p/2)$, which is assumed in Lemma \ref{lem4081801}.

     By It\^o's formula, $v:=e^{-\kappa t} u_m \in \frH_{p,\theta,\alpha}^\gamma(D,\tau)$ satisfies
    \begin{equation} \label{eq4062244}
        dv = (L_tv - \kappa v)dt + \sum_{k=1}^\infty e^{-\kappa t}\xi|u_m\wedge m|^{1+\lambda}(\Pi*e_k)dw_t^k
    \end{equation}
    in $(0,\tau)\times D$. Here, we emphasize that in the stochastic part, $|u_m\wedge m|^{1+\lambda}= |0\vee u_m\wedge m|^{1+\lambda}$ due to $u_m\geq0$. 
Let us consider a regularized distance function $\widetilde{\psi}$ satisfying \eqref{eq4052240}.
Since $\theta<d-1+p$, we can take $\beta\in(0,\alpha/2)$ so that
\begin{equation} \label{eq4071126}
    \frac{\alpha}{2}+\frac{\theta}{p}-\frac{d-1}{p}-1<\beta < \frac{\gamma}{\lambda}\vee\left(\frac{\alpha-d}{2\lambda}-d\right).
\end{equation}
Then for $p'$ and $\theta'$ such that $1/p+1/p'=1$ and $\theta/p+\theta'/p'=d$, we have $\beta p'+\theta'-\alpha p'/2>d-1$ and
$\widetilde{\psi}^\beta \in L_{p',\theta'-\alpha p'/2}(D)$. Moreover, it follows from the proof of \cite[Lemma 3.10]{DR24} that $|L_t\widetilde{\psi}^\beta| \leq N\widetilde{\psi}^{\beta-\alpha}$, which leads to $L_t\widetilde{\psi}^\beta \in L_{p',\theta'+\alpha p'/2}(D)$ for fixed $(\omega,t)$. By the solvability result of the elliptic equation (see Remark \ref{rem4061023}), we deduce that $\widetilde{\psi}^\beta\in H_{p',\theta'-\alpha p'/2}^{\alpha/2}(D)$. This together with Lemma \ref{lem_prop} $(iii)$ implies that we can test \eqref{eq4062244} by $\widetilde{\psi}^\beta$. Hence, for any stopping time $\tau_0\leq \tau$,
\begin{align} \label{eq4061120}
    \bE(e^{-\kappa \tau_0} u_m(\tau_0,\cdot),\widetilde{\psi}^\beta)_D &= \bE(u_0,\widetilde{\psi}^\beta)_D + \bE\int_0^{\tau_0} e^{-\kappa s} (u_m, L_s\widetilde{\psi}^\beta - \kappa \widetilde{\psi}^\beta )_D ds \nonumber
    \\
    &\quad+ \bE \int_0^{\tau_0} e^{-\kappa s}(\xi|u_m\wedge m|^{1+\lambda}(\Pi*e_k),\widetilde{\psi}^\beta)_D dw_s^k.
\end{align}
Since
\begin{align*}
     &\bE \int_0^{\tau} \sum_{k=1}^\infty e^{-\kappa s}|(\xi|u_m\wedge m|^{1+\lambda}(\Pi*e_k),\widetilde{\psi}^\beta)_D|^2 ds 
     \\
     &\leq N \bE \int_0^{\tau} \||u_m\wedge m|^{1+\lambda}\widetilde{\psi}^\beta\|_{\cH}^2 ds \leq N(m) \bE \int_0^{\tau} \|1_{D}\|_{\cH}^2 ds <\infty
\end{align*}
by \eqref{eq3141037}, it follows that the last stochastic integral in \eqref{eq4061120} is a square integrable martingale and hence vanishes under expectation.

Next, we consider the second term in \eqref{eq4061120}. It follows from \cite[Lemma 3.11]{DR24} that there exist $\delta>0$ and $N>0$ such that
\begin{equation*}
    L_t\widetilde{\psi}^{\beta}(x) \leq -N\widetilde{\psi}^{\beta-\alpha}(x)
\end{equation*}
for any $d_x<\delta$.
On the other hand, since $\widetilde{\psi} \in C_b^2(D_{\delta})$ where $D_{\delta}:=\{d_x\geq\delta\}$, we have $|L_t\widetilde{\psi}^\beta|\leq N \leq N\widetilde{\psi}^\beta$ on $D_{\delta}$. Thus, by using $u_m\geq0$,
\begin{align*}
    &(u_m, L_s\widetilde{\psi}^\beta)_D \leq -N\int_{D\setminus D_\delta} u_m \widetilde{\psi}^{\beta-\alpha} dx + \int_{D_\delta} u_m L_s\widetilde{\psi}^{\beta} dx
    \\
    &\leq N\int_{D_\delta} u_m \widetilde{\psi}^{\beta} dx \leq N\int_{D} u_m \widetilde{\psi}^{\beta} dx = N(u_m,\widetilde{\psi}^{\beta})_D.
\end{align*}
Hence, we see that for sufficiently large $\kappa>0$,
    \begin{align*}
        \bE(e^{-\kappa T} u_m(\tau_0,\cdot),\widetilde{\psi}^\beta)_D &\leq \bE(e^{-\kappa \tau_0} u_m(\tau_0,\cdot),\widetilde{\psi}^\beta)_D
        \\
        &\leq \bE(u_0,\widetilde{\psi}^\beta)_D +  (N-\kappa)\bE\int_0^{\tau_0} e^{-\kappa s} (u_m,\widetilde{\psi}^{\beta})_D ds
        \\
        &\leq \bE(u_0,\widetilde{\psi}^\beta)_D.
    \end{align*}
    Since $\tau_0\leq\tau$ is arbitrary, this and \cite[Theorem III.6.8]{K95} lead to
    \begin{equation*}
        \bE\sup_{t\leq\tau} \|\widetilde{\psi}^\beta u_m(t,\cdot)\|_{L_1(D)}^{1/2} = \bE\sup_{t\leq\tau} |(u_m(t,\cdot),\widetilde{\psi}^\beta)_D|^{1/2} \leq Ne^{\kappa T} \bE|(u_0,\widetilde{\psi}^\beta)_D|^{1/2} =: N_0,
    \end{equation*}
    where $N_0$ is independent of $m$. This actually implies that for any $S>0$,
    \begin{equation} \label{eq4081730}
        \bP \left(\sup_{t\leq\tau} \|\widetilde{\psi}^\beta u_m(t,\cdot)\|_{L_1(D)} >S \right) \leq N_0/\sqrt{S}.
    \end{equation}

    Note that $\|\widetilde{\psi}^\beta u_m(t,\cdot)\|_{L_1(D)} = (u_m(t,\cdot),\widetilde{\psi}^\beta)_D$ is predictable. Thus,
    \begin{equation*}
        \tau_m(S):= \inf\{t\geq0: \|\widetilde{\psi}^\beta u_m(t,\cdot)\|_{L_1(D)} >S \}
    \end{equation*}
    is a stopping time. Let us define
    \begin{equation*}
        \xi_m:=\frac{\xi|u_m\wedge m|^{1+\lambda}}{u_m}
    \end{equation*}
    where $0/0:=0$. Let $\theta_0:=d+\beta$. Then for $t\leq \tau_m(S)$ and $s=1/2\lambda$,
    \begin{align*}
    \|\xi_m(t,\cdot)\|_{L_{2s,\theta_0}(D)}^{2s} &\leq N\int_{D} |\xi|^{2s} |u_m\wedge m|^{2\lambda s} \widetilde{\psi}^{\theta_0-d} dx \\
    &\leq N \|\xi\|_{L_{\infty}}^{2s} \|\widetilde{\psi}^\beta u_m(t,\cdot)\|_{L_1(D)} \leq NS.
    \end{align*}
    Note that due to \eqref{eq4071126}, $\theta_0$ satisfies the condition in Assumption \ref{ass_h} $(s,\gamma)$ with $s=1/2\lambda$. Thus, since $p> 2/(1-2\lambda)$, we can apply Theorem \ref{mainthm} to get
    \begin{equation*}
        \|u_m\|_{\frH_{p,\theta,\alpha}^\gamma(D,\tau\wedge \tau_m(S))} \leq N(S) \|u_0\|_{U_{p,\theta}^\gamma(D)},
    \end{equation*}
    where $N(S)$ is independent of $m$. Combining Lemma \ref{lem_prop} $(iv)$ and \eqref{eq4071511} with $\nu\alpha \leq \gamma-d/p$, we see that
    \begin{equation} \label{eq4082319}
       \bE\sup_{t\leq \tau\wedge \tau_m(S), x\in D} |\psi^{\gamma+\frac{\theta-d}{p}-\frac{\alpha}{2}}u_m|^p \leq N\|u_m\|_{\frH_{p,\theta,\alpha}^\gamma(D,\tau)} \leq N(S).
    \end{equation}
    Here, we note that we used the condition $p>(d+\alpha)/\gamma$.
    Using this, \eqref{eq4081730}, and Chebyshev’s inequality,
    \begin{align*}
&\bP\left(\sup_{t\leq\tau,x\in D}|\psi^{\gamma+\frac{\theta-d}{p}-\frac{\alpha}{2}}u_m|>R\right)
        \\
        &\leq \bP\left(\sup_{t\leq\tau \wedge \tau_m(S),x\in D}|\psi^{\gamma+\frac{\theta-d}{p}-\frac{\alpha}{2}}u_m|>R\right) + \bP(\tau_m(S)\leq\tau)
        \\
        &\leq \frac{N(S)}{R} + \bP\left(\sup_{t\leq\tau}\|\widetilde{\psi}^\beta u_m(t,\cdot)\|_{L_1(D)} >S\right) \leq \frac{N(S)}{R} + \frac{N_0}{\sqrt{S}}.
    \end{align*}
    Letting $R\to\infty$ and $S\to\infty$ in order, we obtain the desired result. The lemma is proved.
\end{proof}

Now we prove Theorem \ref{thm_super}.

\begin{proof}[Proof of Theorem \ref{thm_super}]

    We first prove  the uniqueness result. Let us assume that $u,v\in \frH_{p,\theta,\alpha,loc}^{\gamma}(D,\infty)$ are two solutions to equation \eqref{eq4082244}. Take a sequence of bounded stopping times $\tau_n\to\infty$ such that $u,v\in \frH_{p,\theta,\alpha}^\gamma(D,\tau_n)$. 
 As in \eqref{eq4082319}, by Lemma \ref{lem_prop} $(iv)$ and \eqref{eq4071511} with $\nu\alpha \leq\gamma-d/p$,
    \begin{equation*}
        \psi^{\gamma+\frac{\theta-d}{p}-\frac{\alpha}{2}}u, \, \psi^{\gamma+\frac{\theta-d}{p}-\frac{\alpha}{2}}v \in C_b((0,\tau_n)\times D).
    \end{equation*}
    Due to \eqref{eq4081106}, $\gamma+\frac{\theta-d}{p}-\frac{\alpha}{2}<0$, which implies that $u,v\in C_b((0,\tau_n)\times D)$. Hence, we can consider stopping times
    \begin{equation*}
        \tau^m_n(u):=\inf\{t\leq \tau_n:\sup_{x\in D} u>m\}
    \end{equation*}
    and
        \begin{equation*}
        \tau^m_n(v):=\inf\{t\leq \tau_n:\sup_{x\in D} v>m\}.
    \end{equation*}
    Then one can easily see that $\tau^m_n(u),\tau^m_n(v)\to\tau_n$ as $m\to\infty$.
    For $t\leq \tau^m_n(u)\wedge \tau^m_n(v)$, both $u$ and $v$ satisfy \eqref{eq4100000}. Thus, from the uniqueness result in Lemma \ref{lem4081801}, we have $u=v$ if $t \leq \tau_n$. This proves the uniqueness.

    Next, we  handle the existence result. 
    By Lemma \ref{lem4081801}, we can take a solution $u_m$ to \eqref{eq4100000} in the space $\frH_{p,\theta,\alpha}^{\gamma}(D,T)$ for any nonrandom $T>0$.  For any $R>0$, define the  stopping time
    \begin{equation*}
        \tau_m^R:=\inf\{t\leq T: \sup_{x\in D} |\psi^{\gamma+\frac{\theta-d}{p}-\frac{\alpha}{2}}u_m|>\frac{R}{c_0}\},
    \end{equation*}
    where $c_0:=\sup_{x\in D}\psi^{-\gamma-\frac{\theta-d}{p}+\frac{\alpha}{2}}$. Then for any $t\leq \tau_m^R$,
    \begin{equation*}
    \sup_{x\in D} |u_m| \leq c_0\sup_{x\in D}|\psi^{\gamma+\frac{\theta-d}{p}-\frac{\alpha}{2}}u_m| \leq R.
    \end{equation*}
    Thus, if $m\geq R$, then both $u_m$ and $u_R$ solve
    \begin{equation*}
        du=L_tu dt + |u \wedge R|^{1+\lambda} \dot{W}, \quad t\leq \tau_m^R\wedge T
    \end{equation*}
    and
        \begin{equation*}
        du=L_tu dt + |u \wedge m|^{1+\lambda} \dot{W}, \quad t\leq \tau_R^R\wedge T.
    \end{equation*}
    By the uniqueness result in Lemma \ref{lem4081801}, $u_m=u_R \in \frH_{p,\theta,\alpha}^\gamma((\tau_m^R\vee \tau_R^R)\wedge T,D)$ for any $T>0$. 

For $t\leq \tau_R^R \leq (\tau_m^R\vee \tau_R^R)$,
\begin{equation*}
    \sup_{x\in D} |\psi^{\gamma+\frac{\theta-d}{p}-\frac{\alpha}{2}}u_R| = \sup_{x\in D} |\psi^{\gamma+\frac{\theta-d}{p}-\frac{\alpha}{2}}u_m|\leq \frac{R}{c_0}.
\end{equation*}
This leads to $\tau_R^R\leq \tau_m^R$. A similar argument yields that $\tau_R^R\leq \tau_m^m$.

Now we define
\begin{equation*}
    u(t,x):=u_m(t,x), \quad t\leq \tau_m^m.
\end{equation*}
Then $u$ satisfies \eqref{eq4082244}. To show that $u\in \frH_{p,\theta,,\alpha,loc}^\gamma(D,\infty)$, it remains to prove that $\tau_m:=\tau_m^m\wedge m \to\infty$ as $m\to\infty$. By \eqref{eq4090015}, for any $T>0$,
\begin{align*}
    \limsup_{m\to\infty} \bP(\tau_m^m\leq T) &= \limsup_{m\to\infty} \bP\left(\sup_{t\leq T,x\in D} |\psi^{\gamma+\frac{\theta-d}{p}-\frac{\alpha}{2}}u_m|>\frac{m}{c_0}\right)
    \\
    &\leq \lim_{m\to\infty} \sup_n \bP\left(\sup_{t\leq T,x\in D} |\psi^{\gamma+\frac{\theta-d}{p}-\frac{\alpha}{2}}u_n|>\frac{m}{c_0}\right)=0,
\end{align*}
which easily implies that $\tau_m^m\to \infty$ as $m\to \infty$. The theorem is proved.
\end{proof}

\appendix

\section{Proof of Maximum principle} \label{appA}

\begin{proof}
We follow arguments used in the proof of \cite[Theorem 1.1]{K07}. 

Let $r(z):=|z|^2 1_{z<0}$. Then obviously, $r'(z)=2z 1_{z<0}$ and $r''(z)=2\times 1_{z<0}$.
We take a sequence of functions $r_n\in C_b^2(\bR)$ such that
\begin{equation*}
    |r_n(z)|\leq N|z|^2, \quad |r_n'(z)|\leq N|z|, \quad |r''_n(z)|\leq N,
\end{equation*}
and $(r_n,r_n',r_n'')\to (r,r',r'')$ (a.e.) as $n\to \infty$ (see e.g. \cite[Remark 3.1]{K07}).

\textbf{1.} Assume  (cf. Proposition \ref{lem2072238} and Remark \ref{rem2081915}) that
    \begin{equation} \label{eq4261114}
        u\in \frH_{2,\theta,\alpha}^n(D,\tau) \bigcap \bigcup_{k=1}^\infty L_2(\Omega,C([0,\tau],C_c^n(G_k))),
    \end{equation}
    and $\bD u$ and $\bS u$ are continuous in $x$-variable.
    In this step, we aim to prove the following inequality
\begin{align} \label{eq3310116}
    \bE \int_{D} |u^{(-)}(t\wedge \tau,x))|^2 dx &\leq \bE \int_D |u^{(-)}(0,x)|^2 dx + \bE\int_0^{t\wedge \tau} \int_D u^{(-)}(s,x) f(s,x) dxds \nonumber
    \\
    &\quad + \bE\int_0^{t\wedge \tau} \int_D 1_{u(s,x)<0} |g|_{l_2}^2(s,x) dxds,
\end{align}
where  $z^{(-)}:=z\wedge 0$ for $z\in \bR$.

For given $x\in D$, we can test \eqref{linear} for $\phi^\varepsilon(\cdot-x):= \phi(\frac{\cdot-x}{\varepsilon})$ where $\phi\in C_c^\infty(D)$, and let $\varepsilon\to 0$. Then we obtain that for each $x\in D$,
    \begin{equation*}
u(t,x)= \int_0^t \left(L_s u(s,x)+f(s,x)\right)ds + \sum_{k=1}^\infty \int_0^t g^k(s,x)dw^k_s, \quad t\leq \tau \, \text{(a.s.).}
\end{equation*}
where $f=\bD u-L_tu$ and $g=\bS u$.
Thus, by It\^o's formula, for $\lambda>0$ and $t\leq \tau$ (a.s.),
\begin{align} \label{eq4261115}
    r_n(u(t,x)) &= r_n(u(0,x)) + \int_0^t \Big( r_n'(u(s,x)) (L_su(s,x) + f(s,x))\Big) ds \nonumber
    \\
    &\quad+ \int_0^t \frac{1}{2} r_n''(u(s,x)) |g|_{l_2}^2(s,x) ds + \sum_{k=1}^\infty \int_0^t r_n'(u(s,x)) g^k(s,x)dw_s^k.
\end{align}
By \eqref{eq4261114},
\begin{align*}
    \int_0^t  \sum_{k=1}^\infty |r_n'(u)|^{2} |g^k|^2 ds \leq N(T) \sup_{s\leq t} |u|^{2} \times \int_0^t |g|_{l_2}^2 ds <\infty \text{ (a.s.)}.
\end{align*}
As in \eqref{eq2082117}, if we take the expectation, the stochastic integral in \eqref{eq4261115} vanishes.
Then by integrating with respect to $x$ and letting $n\to \infty$, we get
\begin{align*}
    \bE \int_{D}r(u(t\wedge \tau,x)) dx &= \bE \int_D r(u(0,x)) dx + \bE \int_0^{t\wedge \tau} \int_D r'(u(s,x))L_su(s,x) dxds
    \\
    &\quad+ \bE\int_0^{t\wedge \tau} \int_D r'(u(s,x))f(s,x) dxds
    \\
    &\quad + \bE\int_0^{t\wedge \tau} \int_D \frac{1}{2} r''(u(s,x))|g|_{l_2}^2(s,x) dxds =: I_1 + \cdots + I_4.
\end{align*}

Note that $I_1+I_3+I_4$ is the right-hand side of \eqref{eq3310116}. Thus,  the inequality will follow once we show  that $I_2\leq0$. 
Note that for any $u,v\in H_2^1(\bR^d)$, by Fubini's theorem and a change of variables,
\begin{align*}
    &\int_{\bR^d} u(x) L_tv(x) dx
    \\
    &= \frac{1}{2} \int_{\bR^d} \int_{\bR^d} \left( u(x)(v(x+y)-v(x)) + u(x) (v(x-y)-v(x)) \right) \nu_t(dy)dx
    \\
    &= \frac{1}{2} \int_{\bR^d} \int_{\bR^d} \left( u(x-y)(v(x)-v(x-y)) + u(x+y)(v(x)-v(x+y)) \right) \nu_t(dy)dx,
\end{align*}
where $\nu_t(dy)=\nu_t(\omega,dy)$.
From this, the following identity can be readily obtained:
\begin{align} \label{eq3302351}
        \int_{\bR^d} u(x) L_tv(x) dx &= -\frac{1}{4} \int_{\bR^d} \int_{\bR^d} \left( (u(x+y)-u(x)(v(x+y)-v(x))\right)  \nu_t(dy)dx \nonumber
    \\
    &\quad -\frac{1}{4} \int_{\bR^d} \int_{\bR^d} \left( (u(x-y)-u(x)) (v(x-y)-v(x)) \right) \nu_t(dy)dx.
\end{align}
 Note that for $z,w\in \bR$, we always have
\begin{equation} \label{eq3302352}
    (z^{(-)}-w^{(-)})(z-w)\geq0.
\end{equation}
Since $r'(z)=2z^{(-)}$, \eqref{eq3302351} and \eqref{eq3302352} yield that
\begin{align*}
    \int_D r'(u(s,x))L_su(s,x) dx = 2\int_{\bR^d} u^{(-)}(s,x) L_su(s,x) dx \leq 0,
\end{align*}
which leads to $I_2\leq0$. Here, we remark that we used the fact that $u^{(-)}(s,\cdot)\in H_2^1(\bR^d)$ (see e.g. \cite[Exercise 1.3.18]{K08}).
Hence, we obtain \eqref{eq3310116}.

\textbf{2.}
In this step, we show that \eqref{eq3310116} holds for general $u\in \frH_{2,d,\alpha}^{\alpha/2}(D,\tau)$.
By Proposition \ref{lem2072238} and Remark  Remark \ref{rem2081915} $(i)$, there is a sequence of functions
    \begin{equation*}
        u_n \in \frH_{2,\theta,\alpha}^n(D,\tau) \bigcap \bigcup_{k=1}^\infty L_2(\Omega,C([0,\tau],C_c^n(G_k)))
    \end{equation*}
    such that $u_n\to u$ in $\frH_{2,d,\alpha}^{\alpha/2}(D,\tau)$, and $\bD u_n$ and $\bS u_n$ are continuous in $x$. From the Step \textbf{1}, we have \eqref{eq3310116} with $u_n$ instead of $u$. Let $f_n:=\bD u_n -L_t u_n$. Then, for $s\in(0,t\wedge \tau)$, by Lemma \ref{lem_prop} $(iii)$,
    \begin{align*}
        \left| \int_{D} u_n^{(-)}(s,x) f_n(s,x)dx \right| &= |(u_n^{(-)}(s,\cdot), f_n(s,\cdot)_D |
        \\
        &\leq N \|\psi^{-\alpha/2} u_n^{(-)}(s,\cdot)\|_{H_{2,d}^{\alpha/2}(D)} \|\psi^{\alpha/2} f_n(s,\cdot)\|_{H_{2,d}^{-\alpha/2}(D)}.
    \end{align*}
    Since the operator $v\to v^{(-)}$ is bounded in $H_2^1(\bR^d)$, by \eqref{def. Hptheta}, it is also bounded in $\psi^{\alpha/2} H_{2,d}^{\alpha/2}(D)$. Hence,
    \begin{equation*}
        \lim_{n\to \infty} \int_{D} u_n^{(-)}(s,x) f_n(s,x)dx = \int_{D} u^{(-)}(s,x) f(s,x)dx.
    \end{equation*}
    The other terms in \eqref{eq3310116} for $u_n$ can be handled similarly. In particular, for the convergence of the left-hand side, one needs to use \eqref{embed_sup}.
    Thus, we obtain \eqref{eq3310116} for $u$. The second step is proved.

    \textbf{3.}
Thanks to \eqref{eq3302353}, one can easily find that the right-hand side of \eqref{eq3310116} is less than $0$, which yields that
\begin{eqnarray*}
    \bE \int_{D} |u^{(-)}(t\wedge \tau,x)|^2 dx \leq 0.
\end{eqnarray*}
This easily implies that for each $t$ $u\geq0$ (a.s.). Since $u$ is continuous in $t$ by Proposition \ref{prop holder}, the desired result is obtained. The theorem is proved.
\end{proof}

\section{Some properties of stochastic Banach spaces on $\bR^d$} \label{appB}

We present several properties of the stochastic Banach space $\cH^{\gamma}_{p,\alpha}(\tau)$,  defined over the entire space. This space generalizes the one introduced in \cite{KAA} for the case $\alpha=2$. 

\begin{definition}
 Let $p\geq 2$ and $\gamma\in \bR$.   For any $\cD'(\bR^d)$-valued function $u$ defined on $\Omega\times[0,\tau)$, we write 
$u\in \cH^{\gamma}_{p,\alpha}(\tau)$ if $u\in \bH^{\gamma}_p(\tau)$, $u(0,\cdot)\in B_{p,p}^{\gamma+\alpha-\alpha/p}$ and there exist $f\in \bH^{\gamma-\alpha}_p(\tau)$ and $g\in \bH^{\gamma-\alpha/2}_p(\tau,l_2)$ such that for any 
  $\phi\in C^{\infty}_c(\bR^d)$,
$$
(u(t,\cdot), \phi)_{\bR^d}=(u(0,\cdot),\phi)_{\bR^d} +\int^t_0 (f(s,\cdot), \phi)_{\bR^d}ds + \sum_{k=1}^\infty \int_0^t (g^k(s,\cdot),\phi)_{\bR^d} dw_t^k
$$
holds for all $t\leq \tau$ (a.s.).
In this case, we write $\bD u:=f$ and $\bS u:=g$. The norm in  $\cH_{p,\alpha}^{\gamma}(\tau)$ is defined as 
\begin{align*}
\|u\|_{\cH_{p,\alpha}^{\gamma}(T)} := \| u\|_{\bH_{p}^{\gamma}(T)} + \|u_t\|_{\bH_{p}^{\gamma-\alpha}(T)}+ \| u(0,\cdot) \|_{B_{p,p}^{\gamma-\alpha /p}}. \nonumber
\end{align*}
\end{definition}

\begin{lem} \label{L2 whole}
Let $\gamma\in \bR$, $\tau\leq T$ be a bounded stopping time. Then for any function $u\in \cH_{2,\alpha}^{\gamma}(\tau)$, we have $u\in C([0,\tau],H_2^{\gamma-\alpha/2})$ (a.s.) and
\begin{align} \label{L2 sup}
&\bE \sup_{t\leq \tau} \| u(t,\cdot)\|_{H_2^{\gamma-\alpha/2}}^2 \nonumber
\\
&\leq N(d) \left(\|u_0\|^2_{B_{2,2}^{\gamma-\alpha/2}} + a^2\|u\|^2_{\bH_2^{\gamma}(\tau)}+a^{-2}\| \bD u\|_{\bH_2^{\gamma-\alpha}(\tau)}+ \|\bS u\|^2_{\bH_2^{\gamma-\alpha/2}(\tau,l_2)}\right). 
\end{align} 
In particular, $N$ is independent of $T$ and $a$.
\end{lem}

\begin{proof}
First, due to the isometry $(1-\Delta)^{\gamma/2-\alpha/4} : \cH_{2,\alpha}^{\gamma}(\tau)\to\cH_{2,\alpha}^{\alpha/2}(\tau)$, it is enough to prove the case $\gamma = \alpha/2$.

Next, we prove \eqref{L2 sup} for $a=1$.
We use Sobolev's mollifiers. Take a nonnegative function $\zeta\in C_c^\infty(\bR^d)$ such that $\int_{\bR^d} \zeta dx=1$. For $\varepsilon>0$, define $\zeta^\varepsilon(x):=\varepsilon^{-d}\zeta(x/\varepsilon)$ and for any distributions $v$ on $\bR^d$, denote $v^\varepsilon=v*\zeta^\varepsilon(x)$. Here, for each $x\in \bR^d$, $u^\varepsilon$ satisfies
\begin{align*}
u^\varepsilon (t,x) = u_0^\varepsilon (x) + \int_0^t f^\varepsilon(s,x) ds + \sum_{k=1}^\infty \int_0^t (g^k)^\varepsilon (s,x) dw^k_s, \quad \forall t\leq \tau \, (a.s.)
\end{align*}
where $f:=\bD u$ and $g:= \bS u$. By It\^o's formula,  for all $t\leq \tau$ (a.s.)
\begin{align*}
|u^\varepsilon (t,x)|^2 &= |u_0^\varepsilon (x)|^2 + 2\int_0^t u^\varepsilon (s,x)f^\varepsilon(s,x) ds 
\\
&\quad+ \int_0^t |g^\varepsilon(s,x)|_{l_2}^2 ds + 2\sum_{k=1}^\infty \int_0^t u^\varepsilon (s,x)(g^k)^\varepsilon (s,x) dw^k_s.
\end{align*}
Thus, we integrate both sides with respect to $x$ and apply (stochastic) Fubini's theorem to get
\begin{align} \label{21.01.10-1}
\|u^\varepsilon (t,\cdot)\|_{L_2}^2 &= \|u_0^\varepsilon\|_{L_2}^2 + 2 \int_0^t \int_{\bR^d} u^\varepsilon (s,x)f^\varepsilon(s,x) dx ds  \nonumber
\\
&\quad+ \|g^\varepsilon\|_{\bL_2(t,l_2)}^2 + 2\sum_{k=1}^\infty \int_0^t \int_{\bR^d} u^\varepsilon (s,x)(g^k)^\varepsilon (s,x) dx dw^k_s.
\end{align}
Since $L_2=H_2^0=B_2^0$, $\|u_0^\varepsilon\|_{L_2} = \|u_0^\varepsilon\|_{B_{2,2}^0}$. 
By Young's inequality,
\begin{align} \label{21.01.10-2}
\left| \int_0^t \int_{\bR^d} u^\varepsilon (s,x)f^\varepsilon(s,x) dx ds \right | &= \left| \int_0^t \left( (1-\Delta)^{-\alpha/4} u^\varepsilon (s,\cdot), (1-\Delta)^{\alpha/4}f^\varepsilon(s,\cdot) \right)_{\bR^d} ds \right| \nonumber 
\\
&\leq \int_0^t\|u^\varepsilon (s,\cdot)\|_{H_2^{\alpha/2}} \|f^\varepsilon (s,\cdot)\|_{H_2^{-\alpha/2}} ds \nonumber
\\
&\leq \frac{1}{2} \int_0^t\|u^\varepsilon (s,\cdot)\|_{H_2^{\alpha/2}}^2 ds + \frac{1}{2} \int_0^t\|f^\varepsilon (s,\cdot)\|_{H_2^{-\alpha/2}}^2 ds.
\end{align}
By the Burkholder-Davis-Gundy inequality, Minkowski's inequality, and H\"older's inequality,
\begin{align*}
&\bE \left[\sup_{t\leq \tau}\left|\sum_{k=1}^\infty \int_0^t u^\varepsilon (s,x)(g^k)^\varepsilon (s,x) dx dw^k_s \right| \right] 
\\
&\leq N\bE \left[ \left| \int_0^\tau \left| \int_{\bR^d} u^\varepsilon (s,x)(g^k)^\varepsilon (s,x) dx  \right|_{l_2}^2 ds  \right|^{1/2} \right]
\\
&\leq N \bE \left[ \left| \int_0^\tau  \|u^\varepsilon(s,\cdot)\|_{L_2}^2 \|g^\varepsilon(s,\cdot)\|_{L_2(l_2)}^2 ds \right|^{1/2} \right].
\end{align*}
Thus, using Young's inequality, for any $\delta>0$, we get
\begin{align} \label{21.01.10-3}
&\bE \left[\sup_{t\leq \tau}\left|\sum_{k=1}^\infty \int_0^t u^\varepsilon (s,x)(g^k)^\varepsilon (s,x) dx dw^k_s \right| \right] \nonumber
\\
&\leq N \bE \left[ \sup_{t\leq \tau} \|u^\varepsilon(t,\cdot)\|_{L_2} \left| \int_0^T \|g^\varepsilon(s,\cdot)\|_{L_2(l_2)}^2 ds \right|^{1/2} \right] \nonumber
\\
&\leq N \delta \bE \sup_{t\leq \tau} \|u^\varepsilon(t,\cdot)\|_{L_2}^2  + N_\delta \|g\|_{\bL_2(\tau,l_2)}^2.
\end{align}
Since $u^\varepsilon \in \cH_{2,\alpha}^\sigma(\tau)$ for any $\sigma\in \bR$, \cite[Theorem 2.6]{KK12} yields
$$
\bE \sup_{t\leq \tau} \|u^\varepsilon(t,\cdot)\|_{L_2}^2<\infty.
$$
Therefore, taking $\delta$ sufficiently small, \eqref{21.01.10-1}, \eqref{21.01.10-2}, and \eqref{21.01.10-3} together imply  \eqref{L2 sup} with $a=1$ and $u^\varepsilon$ in place  of $u$.
Also, for $\varepsilon_1,\varepsilon_2>0$, by considering $u^{\varepsilon_1}-u^{\varepsilon_2}$ in place of $u^\varepsilon$, we conclude $u^\varepsilon$ is a Cauchy sequence in $L_2(\Omega,C([0,\tau],L_2))$. Thus \eqref{L2 sup} with $a=1$ is proved.

We now prove \eqref{L2 sup} with general $a>0$. Note that for any constant $c>0$, $u_c(t,\cdot):= u(ct,\cdot)$ satisfies
\begin{align} \label{22.01.10.2342}
du_c(t,\cdot) = cf(ct,\cdot) dt + \sqrt{c} g^k(ct,\cdot) dw^k(c)_t,
\end{align}
where $f=\bD u$, $g=\bS u$ and $w^k(c)_t=c w^k_{\sqrt{c}t}$. Hence,  applying \eqref{L2 sup} with $a=1$, we have
\begin{align*}
&\bE \sup_{t\leq \tau} \|u(t,\cdot)\|_{L_2} = \bE \sup_{t\leq \tau/c} \|u_c(t,\cdot)\|_{L_2}
\\
&\leq N\left( \|u_0\|_{L_2} + \|u_c\|_{\bH_2^{\alpha/2}(\tau/c)}+\| cf(ct,\cdot)\|_{\bH_2^{-\alpha/2}(\tau/c)}+ \|\sqrt{c}g(ct,\cdot)\|_{\bL_2(\tau/c,l_2)} \right)
\\
&=N\left( \|u_0\|_{L_2} + c^{-1/2}\|u\|_{\bH_2^{\alpha/2}(\tau)}+ c^{1/2}\| f\|_{\bH_2^{-\alpha/2}(\tau)}+ \|g\|_{\bL_2(\tau,l_2)} \right).
\end{align*}
Choosing $a=c^{-1/2}$,  we get \eqref{L2 sup}, and thus the lemma is proved.
\end{proof}

\begin{lem}
Let $p\in(2,\infty)$, $\alpha\in(0,2)$, $\gamma\in \bR$ and $1/p<\mu<\nu\leq1/2$. For $a>0$, $0\leq s\leq t\leq \tau$ and $u\in \cH^{\gamma}_{p,\alpha}(\tau)$, 
\begin{align*}
    &\bE \|u(t\wedge \tau)-u(s \wedge \tau)\|_{H_p^{\gamma-\nu\alpha}}^p \nonumber
    \\
    &\leq N|t-s|^{\nu p-1}a^{(2\nu-1)p}\left(a^p\|u\|_{\bH_p^{\gamma}(\tau)}^p + a^{-p}\| \bD u\|_{\bH_p^{\gamma-\alpha}(\tau)}^p + \|\bS u\|_{\bH_p^{\gamma-\alpha/2}(\tau,l_2)}^p \right),
\end{align*}
where $N=N(\alpha,p,\nu,\theta)$. In particular, we have
\begin{align}
\label{Lp whole C}
    &\bE \|u(t)-u_0\|_{C^{\mu-1/p}([0,\tau];{H_p^{\gamma-\nu\alpha})}}^p \nonumber
    \\
    &\leq NT^{(\nu-\mu)p} a^{(2\mu-1)p} \left(a^p\|u\|_{\bH_p^{\gamma}(\tau)}^p +a^{-p}\| \bD u\|_{\bH_p^{\gamma-\alpha}(\tau)}^p + \|\bS u\|_{\bH_p^{\gamma-\alpha/2}(\tau,l_2)}^p \right).
\end{align}
\end{lem}

\begin{proof}
As in the proof of Lemma \ref{L2 whole}, it is enough to consider the case $\gamma=\nu\alpha$. 
In addition, for $t\leq T$, by  replacing $\bD u, \bS u$, and $u$ with the modified processes $\widetilde{\bD u}(t) := \bD u1_{(0,\tau)}(t)$, $\widetilde{\bS u}(t) := \bS u1_{(0,\tau)}(t)$, and $\widetilde{u}(t) := u(t\wedge \tau)$, we may  assume without loss of generality that $\tau = T$.

Using equality \eqref{22.01.10.2342}, we can apply a stochastic analogue of the scaling argument used at the beginning of the proof of \cite[Theorem 7.3]{K21},  which allows us to reduce  to case $a=T=1$. Therefore, it suffices to consider the case $a=T=1$. Finally, for the case $a=T=1$,  it is enough to repeat the proof of \cite[Theorem 7.2]{KAA}, which treats the case $\alpha=2$.  The proof goes through  for any $\alpha\in (0,2)$ due to \cite[Lemma A.2]{H21}. The lemma is proved.
\end{proof}

\end{document}